\documentclass[graybox]{svmult}

\usepackage{type1cm}        
\usepackage{makeidx}         
\usepackage{graphicx}        
\usepackage{multicol}        
\usepackage[bottom]{footmisc}

\usepackage{newtxtext}       %
\usepackage{newtxmath}       
\usepackage{color}

\makeindex             

\begin{document}

\title*{From hyperbolic Dehn filling to surgeries in representation varieties}
\titlerunning{Surgeries in representation varieties} 
\author{Georgios Kydonakis}
\institute{Georgios Kydonakis \at Max-Planck-Institut f\"{u}r Mathematik, Bonn, Germany\\
\email{kydonakis@mpim-bonn.mpg.de}}
%
%
\maketitle

\abstract*{Hyperbolic Dehn surgery and the bending procedure provide two ways which can be used to describe hyperbolic deformations of a complete hyperbolic structure on a 3-manifold. Moreover, one can obtain examples of non-Haken manifolds without the use of Thurston's Uniformization Theorem. We review these gluing techniques and present a logical continuity between these ideas and gluing methods for Higgs bundles. We demonstrate how one can construct certain model objects in representation varieties $\text{Hom} \left( \pi_{1} \left( \Sigma \right), G \right) $ for a topological surface $\Sigma$ and a semisimple Lie group $G$. Explicit examples are produced in the case of $\Theta$-positive representations lying in the smooth connected components of the $\text{SO} \left(p,p+1 \right)$-representation variety.}

\abstract{Hyperbolic Dehn surgery and the bending procedure provide two ways which can be used to describe hyperbolic deformations of a complete hyperbolic structure on a 3-manifold. Moreover, one can obtain examples of non-Haken manifolds without the use of Thurston's Uniformization Theorem. We review these gluing techniques and present a logical continuity between these ideas and gluing methods for Higgs bundles. We demonstrate how one can construct certain model objects in representation varieties $\text{Hom} \left( \pi_{1} \left( \Sigma \right), G \right) $ for a topological surface $\Sigma$ and a semisimple Lie group $G$. Explicit examples are produced in the case of $\Theta$-positive representations lying in the smooth connected components of the $\text{SO} \left(p,p+1 \right)$-representation variety.}

\section{Introduction}
A Dehn surgery on a 3-manifold $M$ containing a link $L \subset S^{3}$ is a 2-step process involving the removal of an open tubular neighborhood of the link (drilling) and then gluing back a solid torus using a homeomorphism from the boundary of the solid torus to each of the torus boundary components of $M$ (filling). Of particular interest are the many inequivalent ways one can perform the filling step of the operation, thus providing a way to represent certain examples of 3-dimensional manifolds. In fact, the so-called \textit{fundamental theorem of surgery theory} by Lickorish and Wallace implies that every closed orientable and connected 3-manifold can be obtained by performing a Dehn surgery on a link in a 3-sphere. 

William Thurston introduced hyperbolic geometry into this operation, thus opening the way to certain breakthroughs in 3-manifold theory. His \textit{hyperbolic Dehn filling theorem} implies that the complete hyperbolic structure on the interior of a compact 3-manifold with boundary has a space of hyperbolic deformations parameterized by the generalized Dehn filling coefficients describing the metric completion of the ends of the interior. Among the various and deep advances marked by this result, we highlight here the fact that using hyperbolic Dehn surgery theory one can also obtain examples of non-Haken manifolds, whose hyperbolicity cannot be shown by Thurston's Uniformization Theorem for Haken manifolds. In general, such examples of non-Haken manifolds are not easy to construct otherwise. Deformations of hyperbolic cone structures can, moreover, be better understood when viewed through this prism. In the course of proving Thurston's theorem, one shows not only the existence of a 1-parameter family of cone 3-manifold structures, but can also obtain a path of corresponding holonomies in the representation variety $\text{Hom}( {{\pi }_{1}}( M ), \text{SL}( 2,\mathbb{C} ) )$.  

Deformations of hyperbolic structures on $n$-manifolds can be also described by the \textit{bending} procedure. This involves the construction of a family of quasiconformal homeomorphisms of the hyperbolic $(n+1)$-space, which is required to converge under some compatibility conditions. In the case of a surface, the embedded totally geodesic hypersurfaces are simple closed curves along which bending is possible.

Hyperbolic Dehn surgery was originally developed in dimension 3. In this chapter we describe a set of similar ideas of surgery techniques in representation varieties $\text{Hom}( {{\pi }_{1}}( M ), G)$, where $M$ this time is a closed connected and oriented topological surface of genus $g \ge 2$ and $G$ is a semisimple Lie group. The Teichm\"{u}ller space, viewed as the moduli space of marked hyperbolic structures on $\Sigma$, can be realized as a connected component of the representation variety $\text{Hom}( {{\pi }_{1}}( M ), \text{PSL}( 2,\mathbb{R} ) )$. The recently-emerged field of \textit{higher Teichm\"{u}ller theory} involves the study of certain connected components of the representation varieties $\text{Hom}( {{\pi }_{1}}( M ), G)$, which share essential geometric, topological and dynamical properties with the classical Teichm\"{u}ller space. 

We describe here a gluing construction in $\text{Hom}( {{\pi }_{1}}( M ), G)$ ``in the tradition'' of Thurston's hyperbolic Dehn filling procedure. The parameters involved in this construction are the genus of the surface $\Sigma$ and the holonomy of a surface group representation along the boundary of $\Sigma$. 

The non-abelian Hodge correspondence referring to a homeomorphism between representation varieties and moduli spaces of Higgs bundles over a Riemann surface (with underlying topological surface $\Sigma$ as above) allows us to develop a gluing procedure for the corresponding holomorphic objects, and this makes it easier to determine the connected component where these newly constructed model objects lie, due to an explicit computation of appropriate topological invariants that emerge for their holomorphic counterparts. The deformations involved in the construction are rather expressed in terms of appropriate complex gauge transformations on these holomorphic objects.

In this way, one can construct specific models in certain subsets of  representation varieties $\text{Hom}( {{\pi }_{1}}( M ), G)$, that are hard to be obtained otherwise; in particular, model representations that do not factor as $\rho: \pi_{1} \left( \Sigma \right) \to \text{SL}\left( 2,\mathbb{R} \right) \to G$. These models can be used in turn to describe their deformations in the representation variety and use them as a means to study open subsets (or connected components) of objects with certain geometric properties. As an example, we study here model $\Theta$-positive representations that exhaust the smooth $p\cdot \left( 2g-2 \right)-1$ exceptional components of the $\text{SO} \left(p,p+1 \right)$-character variety for $p>2$; similar models have been also constructed for the $2g-3$ exceptional components of the $\text{Sp}(4, \mathbb{R})$-character variety.

This comparison of ideas points towards further ways in developing tools to study certain subsets of representation varieties, quantitative aspects of this holomorphic gluing strategy or universal bounds for the rational parameters involved.

\section{Hyperbolic Dehn Surgery}
\label{sec:2}
In this section we review the basic concepts involved in the hyperbolic Dehn surgery operation. Even though the results of the technique summarized here do not directly apply for the case of character varieties that we study next, these provide a motivation and an interesting counterpoint to the fundamental ideas behind these surgery methods.  
  
\subsection{Dehn surgery}
\label{sec:2.1}
Dehn surgery\index{Dehn surgery} is a method that has found profound relevance in 3-manifold topology and knot theory. It provides a way to represent 3-dimensional manifolds using a ``drilling and filling'' process. First, a solid torus is removed from a 3-manifold (drilling) and then it is re-attached in many inequivalent ways (filling). This two-stage operation was introduced by Max Dehn in Kapitel II of his 1910 article \textit{\"{U}ber die Topologie des dreidimensionales Raumes} \cite{Dehn} as a method for constructing \textit{Poincar\'{e} spaces}, that is, non-simply connected 3-manifolds with the same topology as the 3-sphere. The texts of Boyer \cite{Boyer}, Gordon \cite{Gordon1}, \cite{Gordon2}, Luecke \cite{Luecke} offer a broad survey on this construction with numerous references for further study.  

The basic parameter of the Dehn surgery operation, in particular referring to the filling stage of the operation, is that of a \textit{slope} on a torus; we briefly introduce this next. Let $M$ be an orientable 3-manifold and $T\subset \partial M$, a toral boundary component of $M$. Denote by $K$ a knot lying in the interior of $M$ and let $N(K) \subset \text{int}(M)$ be a closed tubular neighborhood of $K$. For a homeomorphism $f:\partial \left( {{S}^{1}}\times {{D}^{2}} \right)\to T$, consider the identification space $M\left( T;f \right) :=\left( {{S}^{1}}\times {{D}^{2}} \right){{\cup }_{f}}M$ obtained by identifying the points of $\partial \left( {{S}^{1}}\times {{D}^{2}} \right)$ with their images by $f$. We shall call $M\left( T;f \right)$, a \textit{Dehn filling} of $M$ along $T$. A \textit{Dehn surgery} on a knot $K$ is then a filling of the exterior of the knot $K$, ${{M}_{K}}:= M\backslash \text{int}\left( N\left( K \right) \right)$, along $\partial N\left( K \right)$. 

Note that a filling  $M\left( T;f \right)$ depends only on the isotopy class of the attaching homeomorphism $f:\partial \left( {{S}^{1}}\times {{D}^{2}} \right)\to T$. In fact, the dependence of $f$ is much weaker, for, if ${{C}_{0}}=\left\{ pt \right\}\times \partial {{D}^{2}}\subset \partial \left( {{S}^{1}}\times {{D}^{2}} \right)$, then $M\left( T;f \right)$ depends only on the isotopy class of the curve $f\left(C_{0}\right)$ in $T$. 

\begin{definition}
A \textit{slope}\index{slope} on a torus $T$ is defined as the isotopy class of an essential unoriented simple closed curve on $T$. If $K$ is a knot in a 3-manifold $M$, then a slope of $K$ is any slope on $\partial N \left(K \right)$. 
\end{definition}
One has the following proposition:
\begin{proposition}
A Dehn filling of $M$ along a torus $T\subset \partial M$ is determined up to orientation preserving homeomorphism, by a slope on $T$. Furthermore, any slope on $T$ arises as the slope of a Dehn filling of $M$. 
\end{proposition}

The set of slopes on a torus $T$ is parameterized by the set of $\pm $-pairs of primitive homology classes in ${{H}_{1}}\left( T \right)$. In particular, for the 3-sphere ${{S}^{3}}$ with its usual orientation based on the right-hand rule, the set of slopes of knots in ${{S}^{3}}$ is canonically identified with $\mathbb{Q}\cup \left\{ \frac{1}{0} \right\}$; we may thus realize a slope $r$ of a knot $K$ by a fraction $\frac{p}{q}\in \mathbb{Q}\cup \left\{ \frac{1}{0} \right\}$.

\begin{definition}
Let $K$ be a knot in ${{S}^{3}}$. An \textit{integral slope}\index{integral slope}\index{slope!integral} of $K$ is a slope corresponding to an integer. We will call \textit{integral surgery}\index{integral surgery}\index{surgery!integral} a surgery on $K$ whose slope is integral.  
\end{definition}

One may now consider the problem of existence and uniqueness of a surgery presentation of a given closed connected orientable 3-manifold by surgery on a finite number of knots in $S^{3}$. By a set of \textit{surgery data} $\left( L;{{r}_{1}},\ldots ,{{r}_{n}} \right)$ we shall mean a link $L={{K}_{1}}\cup \cdots \cup {{K}_{n}}$ lying in the interior of a 3-manifold $M$, together with a slope ${{r}_{i}}$ for each knot ${{K}_{i}}$. Let $L\left( {{r}_{1}},\ldots ,{{r}_{n}} \right)$ denote the manifold obtained by performing the Dehn surgeries prescribed by the surgery data. In the special case when $M={{S}^{3}}$  and each ${{r}_{i}}$ is an integral slope, the surgery data $\left( L;{{r}_{1}},\ldots ,{{r}_{n}} \right)$ is often called a \textit{framed link}\index{framed link}.  

The following result is known as the \textit{fundamental theorem of surgery theory}\index{fundamental theorem of surgery theory}; it was proved using different and independent approaches by Lickorish and Wallace:

\begin{theorem}[Lickorish \cite{Lickorish}, Wallace \cite{Wallace}]. Let $M$ be a closed connected orientable 3-manifold. There exists a framed link  $\left( L;{{r}_{1}},\ldots ,{{r}_{n}} \right)$ in $S^{3}$ such that $M$ is homeomorphic to $L\left( {{r}_{1}},\ldots ,{{r}_{n}} \right)$.
\end{theorem}

For the problem of uniqueness of a surgery presentation of a given manifold, Kirby \cite{Kirby} introduced two moves on (integrally) framed links which do not alter the presented manifold; he also proved that two framed links represent manifolds which are orientation preserving homeomorphic if and only if they are related by a finite sequence of these moves, nowadays called \textit{Kirby moves}\index{Kirby move}. This problem was completely analyzed by Rolfsen in \cite{Rolfsen}.    

\subsection{Hyperbolic Dehn surgery}
\label{sec:2.1}
A breakthrough in 3-manifold theory as well as in knot theory was signified by the introduction by Thurston of hyperbolic geometry into the Dehn surgery operation. Necessary and sufficient conditions for the complete gluing of a hyperbolic 3-manifold were given by Seifert in \cite{Seifert}. The concept of link of a cusp point of a hyperbolic 3-manifold was introduced by Thurston in his seminal 1979 lecture notes \cite{Thurston}. 

The celebrated \textit{hyperbolic Dehn filling theorem} of Thurston (Theorem 5.9 in \cite{Thurston}) provides a parameterization of a set of hyperbolic deformations of a complete hyperbolic structure on the interior of a compact 3-manifold with boundary; the parameters, called \textit{generalized Dehn filling coefficients}, describe the metric completion of the ends of the interior.

Among the various and deep advances in 3-manifold theory marked by this result, we will highlight here the fact that using hyperbolic Dehn surgery theory one can also obtain examples of non-Haken manifolds, whose hyperbolicity cannot be shown by Thurston's Uniformization Theorem for Haken manifolds; in fact, the proof of Thurston's theorem does not depend on uniformization. Deformations of hyperbolic cone structures can, moreover, be better understood when viewed through this prism. Another important aspect to be stressed next is the role the generalized Dehn filling coefficients play in the perception of the spaces of hyperbolic deformations parameterized by these coefficients. 

The Theorem was first proven in Thurston's notes \cite{Thurston} in the manifold case and has later been extended in the case of orbifolds by Dunbar and Meyerhoff \cite{DuMe}. A detailed review of the proof in both cases can be found in Appendix B of \cite{BoPo} using (in the manifold case) an argument of Zhou \cite{Zhou1}. We will follow next the approach of \cite{BoPo} for our purposes. 

Let $M$ be a compact 3-manifold with boundary $\partial M=T_{1}^{2}\cup \cdots \cup T_{k}^{2}$, a non-empty union of tori, whose interior $\text{int}\left( M \right)$ is complete hyperbolic with finite volume. For each boundary component $T_{j}^{2}$ of $M$, with $j=1,\ldots ,k$, fix two oriented simple closed curves ${{\mu }_{j}}$ and ${{\lambda }_{j}}$ generating the fundamental group ${{\pi }_{1}}\left( T_{j}^{2} \right)$. The holonomy of ${{\mu }_{j}}$ and ${{\lambda }_{j}}$ can be viewed as affine transformations of $\mathbb{C}=\partial {{\mathbb{H}}^{3}}\backslash \left\{ \infty  \right\}$ ($\infty $ being a point fixed by ${{\mu }_{j}}$ and  ${{\lambda }_{j}}$). Then, one can introduce holomorphic parameters ${{u}_{j}}$ and ${{v}_{j}}$ as branches of the logarithm of the linear part of the holonomy around ${{\mu }_{j}}$ and ${{\lambda }_{j}}$ respectively. For $U\subset {{\mathbb{C}}^{k}}$ a neighborhood of the origin, associate to each $u\in U$ a point ${{\rho }_{u}}\in \mathsf{\mathcal{X}}\left( M \right)={\text{Hom}\left( {{\pi }_{1}}\left( M \right),\text{SL}\left( 2,\mathbb{C} \right) \right)}//{\text{SL}\left( 2,\mathbb{C} \right)}\;$ in the $\text{SL}\left( 2,\mathbb{C} \right)$-character variety; this can be done by considering an analytic section 
$$s:V\subset \mathsf{\mathcal{X}}\left( M \right)\to \text{Hom}\left( {{\pi }_{1}}\left( M \right),\text{SL}\left( 2,\mathbb{C} \right) \right),$$
such that $s\left( {{\chi }_{0}} \right)={{\rho }_{0}}$, where ${{\rho }_{0}}$ is a lift of the holonomy representation of $\text{int}\left( M \right)$ and ${{\chi }_{0}}\in \mathsf{\mathcal{X}}\left( M \right)$ its character. Then, one has the following important lemma:

\begin{lemma}[Lemma B.1.6 in \cite{BoPo}]
For $j=1,\ldots ,k$, there is an analytic map ${{A}_{j}}:U\to \text{SL}\left( 2,\mathbb{C} \right)$ such that for every $u\in U$:
	\[{{\rho }_{u}}\left( {{\mu }_{j}} \right)={{\varepsilon }_{j}}{{A}_{j}}\left( u \right)\left( \begin{matrix}
   {{e}^{{{{u}_{j}}}/{2}\;}} & 1  \\
   0 & {{e}^{-{{{u}_{j}}}/{2}\;}}  \\
\end{matrix} \right){{A}_{j}}{{\left( u \right)}^{-1}},\,\,\,\,\,\,\,\,\,\,\,\,\,\,\,\text{with }{{\varepsilon }_{j}}=\pm 1,\]
\end{lemma} 
while the commutativity between ${{\lambda }_{j}}$ and ${{\mu }_{j}}$ implies the following: 
\begin{lemma}[Lemma B.1.7 in \cite{BoPo}]
There exist unique analytic functions ${{v}_{j}},{{\tau }_{j}}:U\to \mathbb{C}$ such that ${{v}_{j}}\left( 0 \right)=0$ and, for every $u\in U$,
	\[{{\rho }_{u}}\left( {{\lambda }_{j}} \right)=\pm {{A}_{j}}\left( u \right)\left( \begin{matrix}
   {{e}^{{{{v}_{j}}\left( u \right)}/{2}\;}} & {{\tau }_{j}}\left( u \right)  \\
   0 & {{e}^{-{{{v}_{j}}\left( u \right)}/{2}\;}}  \\
\end{matrix} \right){{A}_{j}}{{\left( u \right)}^{-1}}.\]
In addition:
\begin{enumerate}
\item ${{\tau }_{j}}\left( 0 \right)\in \mathbb{C}-\mathbb{R}$;
\item $\text{sinh}\left( {{{v}_{j}}}/{2}\; \right)={{\tau }_{j}}\text{sinh}\left( {{{u}_{j}}}/{2}\; \right)$;
\item ${{v}_{j}}$ is odd in ${{u}_{j}}$ and even in ${{u}_{l}}$, for $l\ne j$;
\item ${{v}_{j}}\left( u \right)={{u}_{j}}\left( {{\tau }_{j}}\left( u \right)+O\left( {{\left| u \right|}^{2}} \right) \right)$.
\end{enumerate}
\end{lemma}
We are finally set to define the generalized Dehn filling coefficients:

\begin{definition}[Thurston \cite{Thurston}] 
For $u \in U$ we define the \textit{generalized Dehn filling coefficients}\index{generalized Dehn filling coefficient} of the $j$-th cusp $\left( {{p}_{j}},{{q}_{j}} \right)\in {{\mathbb{R}}^{2}}\cup \left\{ \infty  \right\}\cong {{S}^{2}}$ by the formula

$$\left\{ \begin{matrix}
   \left( {{p}_{j}},{{q}_{j}} \right)  & =\infty , & \text{if  }{{u}_{j}}=0  \\
   {{p}_{j}}{{u}_{j}}+{{q}_{j}}{{v}_{j}} & =2\pi \sqrt{-1} & \text{if  }{{u}_{j}}\ne 0.  
\end{matrix} \right.$$
\end{definition}
These coefficients are well-defined and the map 
\begin{align*}
U &  \to {{S}^{2}}\times \cdots \times {{S}^{2}} \\
u & \mapsto \left( \left( {{p}_{1}},{{q}_{1}} \right),\ldots ,\left( {{p}_{k}},{{q}_{k}} \right) \right)
\end{align*}
defines a homeomorphism between $U$ and a neighborhood of $\left\{ \infty ,\ldots ,\infty  \right\}$.

\begin{remark}
If ${{p}_{j}},{{q}_{j}}\in \mathbb{Z}$ are coprime, then the completion at the $j$-th torus is a non-singular hyperbolic 3-manifold, which topologically is the Dehn filling with surgery meridian ${{p}_{j}}{{\mu }_{j}}+{{q}_{j}}{{\lambda }_{j}}$. One may also perform $\left( p,q \right)$-Dehn surgery also when $p$ and $q$ are not necessarily coprime integers; this refers to \textit{orbifold Dehn surgery}\index{orbifold Dehn surgery}\index{Dehn surgery!orbifold}, as in \cite{DuMe}.  For instance, $\left( p,0 \right)$-Dehn surgery on a knot $K\subset {{S}^{3}}$ provides an orbifold with base ${{S}^{3}}$ and singular set the knot $K$ with cone angle ${2\pi }/{p}\;$. 
\end{remark}

The statement of the theorem is the following:

\begin{theorem}[Hyperbolic Dehn filling theorem\index{hyperbolic Dehn filling theorem}, Thurston \cite{Thurston}]\label{Thurston_Dehn_hyp}
Let $M$ be a compact 3-manifold with boundary $\partial M=T_{1}^{2}\cup \cdots \cup T_{k}^{2}$, a non-empty union of tori, whose interior $\text{int}\left( M \right)$ is complete hyperbolic with finite volume. There exists a neighborhood of $\left\{ \infty ,\ldots \infty  \right\}$ in ${{S}^{2}}\times \cdots \times {{S}^{2}}$, such that the complete hyperbolic structure on $\text{int}\left( M \right)$  has a space of hyperbolic deformations parameterized by the generalized Dehn filling coefficients in this neighborhood. 
\end{theorem}     

The first major step in the proof involves the construction of the algebraic deformation of the holonomies around each boundary component of the manifold $M$. The second step is to associate generalized Dehn filling coefficients to the aforementioned deformation. The third and final step in the proof involves the construction of the developing maps with the given holonomies. In particular, let ${{D}_{0}}:\widetilde{\text{int}\left( M \right)} \to {{\mathbb{H}}^{3}}$ be the developing map for the complete structure on $\text{int}\left( M \right)$ with holonomy ${{\rho }_{0}}$. Then, for each $u\in U$, there is a developing map  ${{D}_{u}}:\widetilde{\text{int}\left( M \right)} \to {{\mathbb{H}}^{3}}$ with holonomy  ${{\rho }_{u}}$, such that the completion of  $\text{int}\left( M \right)$ is given by the generalized Dehn filling coefficients of $u$. 

We remark here that the family of maps ${{\left\{ {{D}_{u}} \right\}}_{u\in U}}$ is continuous in $u$ in the compact ${{\mathsf{\mathcal{C}}}^{1}}$-topology and that the result above shows not only the existence of a 1-parameter family of cone 3-manifold structures, but also gives a path of corresponding holonomies in the representation variety $\text{Hom}\left( {{\pi }_{1}}(M),\text{SL}\left( 2,\mathbb{C} \right) \right)$.

\subsection{Haken manifolds and Thurston's Uniformization}

The notion of \textit{Haken manifold}\index{Haken manifold} involves a large class of closed 3-manifolds and play an important role in the study of the topology of 3-manifolds. These were introduced by Wolfgang Haken \cite{Haken1} as a class of compact irreducible 3-manifolds containing incompressible surfaces, for which  he showed in \cite{Haken2} that they admit a hierarchy to a union of 3-balls by cutting along essential embedded surfaces. This property allows one to produce certain statements for Haken manifolds using an induction process. Let us next state these definitions more rigorously:

\begin{definition}\label{incompressible}
Let $M$ be a 3-manifold. A properly embedded surface $\Sigma \subset M$ is called \textit{incompressible}\index{incompressible} if the map between fundamental groups $\pi_{1}\left(\Sigma \right) \to \pi_{1}\left(M \right)$ is injective. Otherwise, the surface is called \textit{compressible}. A torus in an irreducible 3-manifold is compressible if and only if it bounds a solid torus.
\end{definition}

\begin{definition}\label{Haken_manifold}
A compact orientable 3-manifold $M$ is called a \textit{Haken manifold} if it is irreducible and contains an orientable incompressible surface $\Sigma \subset M$.
\end{definition}

In \cite{Haken2}, Haken associated a notion of complexity to a Haken manifold, which decreases when one cuts the Haken manifold along an incompressible surface; this can be iterated in order to reduce the complexity until we obtain 3-balls. This approach was a key ingredient in the proof of the Waldhausen theorem\index{Waldhausen theorem} showing that closed Haken manifolds are topologically characterized by their fundamental groups:

\begin{theorem}[Waldhausen, Corollary 6.5 in \cite{Wald}]\label{Waldhausen's}
Let $M$ and $M'$ be two Haken manifolds and let $\pi_{1}\left(M\right) \to \pi_{1}\left(M' \right)$ be an isomorphism between their fundamental groups. Then $M$ and $M'$ are homeomorphic. 
\end{theorem}

An algorithm to determine whether a 3-manifold is Haken was given by  Jaco and Oertel \cite{JaOe}. Thurston's studies of various examples of 3-manifolds admitting complete hyperbolic metrics lead to his proof of a ``uniformization theorem'' satisfied by this large class of Haken manifolds:

\begin{theorem}[Uniformization Theorem for Haken manifolds\index{uniformization theorem for Haken manifolds}, Thurston \cite{Thurston}]\label{Thurston's_Uniformization}
Any atoroidal Haken manifold $M$ admits a hyperbolic structure. By atoroidal here is meant that any embedded incompressible torus is boundary parallel, that is, it can be isotoped into a boundary component of $M$.
\end{theorem}

Thurston's proof uses the hierarchy property of Haken manifolds. By the Waldhausen theorem, (a Haken manifold) $M$ can be decomposed into a finite sum of closed balls $B^{3}$ by incompressible surfaces; in other words, there exists a sequence of manifolds with boundary 
\[M\mapsto {{M}_{1}}\mapsto \ldots \mapsto {{B}^{3}}\cup \ldots \cup {{B}^{3}}.\]
Then, starting with hyperbolic structures on the balls $B^{3}$ we may get a hyperbolic structure by gluing at each step in this sequence from these balls back to $M$. A full proof of this theorem was never published by Thurston; fairly detailed outlines of the proof can be found in the articles by Morgan \cite{Morgan} or Wall \cite{Wall}. It also follows from Perelman's proof of the more general geometrization conjecture of Thurston constructing the Ricci flow with surgeries on 3-manifolds \cite{Perelman}; see also \cite{Beetal}, \cite{MoTi}. 

The geometrization conjecture evolved from Thurston's considerations that a similar uniformization theorem as for Haken manifolds should hold for all closed 3-manifolds. An important fact considered was that non-Haken manifolds do not contain incompressible surfaces, thus it is impossible to decompose those into simpler pieces. One way by which Thurston proved that non-Haken atoroidal 3-manifolds can be equipped with a hyperbolic structure was by deforming the structure of a cone manifold by increasing its cone angle. 

Furthermore, using hyperbolic Dehn surgery it is possible to obtain \textit{non-Haken manifolds, whose hyperbolicity cannot be shown by the uniformization theorem}. Such examples are not easy to construct otherwise; see Reid \cite{Reid} for explicit examples of non-Haken hyperbolic 3-manifolds with a finite cover which fibers over the circle. Moreover, deformations of hyperbolic structures can be described more concretely using the framework of hyperbolic Dehn surgery. 

In \cite{HoKe} Hodgson and Kerckhoff established  a universal upper bound on the number of non-hyperbolic Dehn surgeries per boundary torus, thus  giving a quantitative version of Thurston's hyperbolic Dehn filling theorem; see also the later article of Lackenby and Meyerhoff \cite{LaMe} on the maximal number of exceptional Dehn surgeries\index{exceptional Dehn surgery}\index{Dehn surgery!exceptional}, providing a proof to Gordon's conjecture \cite{Gordon2} on the number of exceptional slopes. For example, Dehn surgeries on the figure-eight knot produce non-Haken, hyperbolic 3-manifolds except in ten cases. For the exterior of the figure-eight knot in $S^{3}$ the exceptional surgeries, that is, the ones which do not result in a hyperbolic structure, are 
\[\left\{ \left( 1,0 \right),\left( 0,1 \right),\pm \left( 1,1 \right),\pm \left( 2,1 \right),\pm \left( 3,1 \right),\pm \left( 4,1 \right) \right\}.\]

\section{Deformations of hyperbolic structures by bending}

A deformation method of hyperbolic structures on $n$-manifolds called \textit{bending}\index{bending} is suggested by the famous ``Mickey Mouse'' example of Thurston (Example 8.7.3 in \cite{Thurston}). Given a hyperbolic structure on a genus two surface, the structure can be considered to arise from the bending of the surface along a simple closed geodesic by an angle $\frac{\pi}{2}$. If the geodesic is short enough, this will give rise to a quasi-Fuchsian group. In order to extend this idea to $n$ dimensions, the manifold is required to contain a totally geodesic submanifold of codimension one along which the bending can take place, thus defining a deformation. That there are compact hyperbolic $n$-manifolds with arbitrarily many such submanifolds was shown by Millson in \cite{Millson}.

For $n=3$, hyperbolic structures of infinite volume are related to Kleinian groups which are discrete subgroups of $\text{PSL}(2, \mathbb{C})$ acting discontinuously on part of $S^{2}$. In turn, deformations of Kleinian groups can be studied by analyzing the conformal structures on the components of  the boundary of the quotient space; a similar phenomenon occurs in higher dimensions (see  the works of Apanasov and Tetenov \cite{Apanasov}, \cite{AT}).

Christos Kourouniotis introduced in \cite{Kourouniotis1} a deformation technique of hyperbolic structures on $n$-manifolds via the construction of a family of quasiconformal homeomorphisms of the hyperbolic $(n+1)$-space. His construction of the bending homeomorphism is similar to the construction by Wolpert in \cite{Wolpert} of a homeomorphism giving rise to the Fenchel--Nielsen deformation; cf. also the work of Johnson and Millson \cite{JM} for an algebraic version of the bending deformation.

The idea in \cite{Kourouniotis1} is to construct a quasi-conformal homeomorphism compatible with a subgroup $\Gamma$ of $G_{n}$ step by step, as the infinite product of a sequence of homeomorphisms; this product is required to converge and to be compatible with $\Gamma$.

In the case of a surface, the embedded totally geodesic hypersurfaces are simple closed curves along which bending is possible. One could also extend in this case the definition of bending to the case of a geodesic lamination, as for instance in the work of Epstein and Marden \cite{EM}. Still in this surface case, Kourouniotis has studied in \cite{Kourouniotis2} the possibility of bending quasi-Fuchsian structures. Namely, for a closed surface $\Sigma$, the space $\mathcal{QF}(\Sigma)$ of quasi-Fuchsian structures on $\Sigma$ is a quotient of the space of injective homomorphisms $\rho :{{\pi }_{1}}\left( \Sigma  \right)\to \text{PSL}\left( 2,\mathbb{C} \right)$ with $\text{Im}\rho =\Gamma $ and $\Sigma \times I\cong {{{\mathbb{H}}^{3}}}/{\Gamma }\;$; Fuchsian points are classes of homomorphisms with image in $\text{PSL}\left( 2,\mathbb{R} \right)$ and correspond to hyperbolic structures on $\Sigma $. For a simple closed geodesic $\gamma \subset \Sigma $, there is a 1-parameter family of pairs $\left( {{f}_{t}},{{\rho }_{t}} \right)$, where ${{f}_{t}}:\tilde{\Sigma }\to {{\mathbb{H}}^{3}}$ and ${{\rho }_{t}}:{{\pi }_{1}}\left( \Sigma  \right)\to \text{PSL}\left( 2,\mathbb{C} \right)$, such that ${{f}_{t}}$ is ${{\rho }_{t}}$-equivariant, for every $t\ge 0$. Note that for $t=0$, ${{\rho }_{0}}$ is Fuchsian and ${{f}_{0}}$ equivariantly embeds $\tilde{\Sigma }$ as a hyperbolic plane in the hyperbolic 3-space ${{\mathbb{H}}^{3}}$. This deformation is induced by a 1-parameter family of isometries from $\text{PSL}\left( 2,\mathbb{C} \right)$. When the bending parameter $t$ is small enough, then ${{f}_{t}}$ is an embedding and ${{\rho }_{t}}$ is an isomorphism of ${{\pi }_{1}}\left( \Sigma  \right)$ onto a quasi-Fuchsian subgroup of $\text{PSL}\left( 2,\mathbb{C} \right)$. 

In \cite{Kourouniotis3}, Kourouniotis studies some quantitative aspects of this bending construction, while universal bounds on the bending lamination of a quasi-Fuchsian group, hence of the bending deformation, were obtained by Bridgeman \cite{Br1}, \cite{Br2}.

\section{Higher Teichm\"{u}ller Theory}
The newly-emerged field of higher Teichm\"{u}ller theory concerns the study of connected components of character varieties for semisimple real Lie groups that entirely consist of discrete and faithful representations. We summarize here some of the very basic topological and geometric properties of these spaces, as well as a recent unified approach to the subject introduced by Olivier Guichard and Anna Wienhard, which seems to be identifying all the cases when such components emerge.

\subsection{The Teichm\"{u}ller space}
Let ${{\Sigma }}$ be a closed connected and oriented topological surface with negative Euler characteristic $\chi \left( {{\Sigma }} \right)=2-2g<0$, for $g$ the genus of ${{\Sigma }}$. The \textit{Teichm\"{u}ller space}\index{Teichm\"{u}ller space} $\mathsf{\mathcal{T}}\left( {{\Sigma }} \right)$ of the surface ${{\Sigma }}$ is defined as the space of marked conformal classes of Riemannian metrics on ${{\Sigma }}$. The Uniformization Theorem of Riemann--Poincar\'{e}--Koebe (see \cite{SGe} for a complete account) guarantees the existence of a unique hyperbolic metric with constant curvature -1 in each conformal class. The Teichm\"{u}ller space can be thus identified with the moduli space of marked hyperbolic structures. Moreover, the mapping class group $\text{Mod}\left( {{\Sigma }} \right)$, that is, the group of all orientation-preserving diffeomorphisms of ${{\Sigma }}$ modulo the ones which are isotopic to the identity, acts naturally on $\mathsf{\mathcal{T}}\left( {{\Sigma }} \right)$ by changing the marking; this action is properly discontinuous and the quotient is the moduli space $\mathsf{\mathcal{M}}\left( {{\Sigma }} \right)$ of Riemann surfaces of topological type given by ${{\Sigma }}$. 

A well-known fact about the Teichm\"{u}ller space is that it is homeomorphic to ${{\mathbb{R}}^{6g-6}}$. There are several ways to see this. One direct way is by parameterizing $\mathsf{\mathcal{T}}\left( {{\Sigma }} \right)$ by Fenchel--Nielsen coordinates --- a complete proof may be found in \cite{Ratcliffe}, Theorem 9.7.4. Another method is to use Teichm\"{u}ller's theorem to identify $\mathsf{\mathcal{T}}\left( {{\Sigma }} \right)$ with the unit ball in the vector space $Q\left( M \right)$ of holomorphic quadratic differentials on a Riemann surface $M$ homeomorphic to ${{\Sigma }}$ --- a detailed proof can be found in \cite{Hubbard}, Theorem 7.2.1. In fact,  $\mathsf{\mathcal{T}}\left( {{\Sigma }} \right)$ can be identified with the entire vector space $Q\left( M \right)$ using Hopf differentials of harmonic maps from $M$ to a Riemann surface of topological type given by ${{\Sigma }}$ --- see \cite{Wolf} for this approach. An application of the Riemann-Roch theorem finally provides that ${{\dim}_{\mathbb{R}}}Q\left( M \right)=6g-6$, for genus $g\ge 2$; we refer, for instance, to Corollary 5.4.2 in  \cite{Jost} for a proof. 

However, what opens the way from the classical  Teichm\"{u}ller theory to what is today called \textit{Higher Teichm\"{u}ller Theory}\index{higher Teichm\"{u}ller theory}\index{Teichm\"{u}ller theory!higher} is the algebraic realization of the space $\mathsf{\mathcal{T}}\left( {{\Sigma }} \right)$ as a subspace of the moduli space of representations of the fundamental group of ${{\Sigma }}$ into the isometry group of the hyperbolic plane. This algebraic realization is conceived through the holonomy representation of a hyperbolic structure. Indeed, for $\left( M,f \right)$ a hyperbolic structure over ${{\Sigma }}$, the orientation preserving homeomorphism $f:{{\Sigma }}\to M$ induces an isomorphism of fundamental groups  ${{f}_{*}}:{{\pi }_{1}}\left( {{\Sigma }} \right)\to {{\pi }_{1}}\left( M \right)$ and ${{\pi }_{1}}\left( M \right)$ acts as the group of deck transformations by isometries on $\tilde{M}\cong {{\mathbb{H}}^{2}}$. But, since $\text{PSL}\left( 2,\mathbb{R} \right)\cong \text{Iso}{{\text{m}}^{+}}\left( {{\mathbb{H}}^{2}} \right)$, the orientation preserving isometries, it follows that this action induces a homomorphism $\rho :{{\pi }_{1}}\left( {{\Sigma }} \right)\to \text{PSL}\left( 2,\mathbb{R} \right)$ which is well-defined up to conjugation by $\text{PSL}\left( 2,\mathbb{R} \right)$. This homomorphism is called the \textit{holonomy} of the hyperbolic structure $\left( M,f \right)$. The \textit{representation variety}\index{representation variety} 
$$\mathsf{\mathcal{R}}\left( \text{PSL}\left( 2,\mathbb{R} \right) \right) := {\text{Hom}\left( {{\pi }_{1}}\left( {{\Sigma }} \right),\text{PSL}\left( 2,\mathbb{R} \right) \right)}//{\text{PSL}\left( 2,\mathbb{R} \right)}\;$$ is the largest Hausdorff quotient of all group homomorphisms $\rho :{{\pi }_{1}}\left( {{\Sigma }} \right)\to \text{PSL}\left( 2,\mathbb{R} \right)$ modulo conjugation by $\text{PSL}\left( 2,\mathbb{R} \right)$. Furthermore, representations induced by equivalent hyperbolic structures using the above approach are conjugate by an element in $\text{PSL}\left( 2,\mathbb{R} \right)$ and the converse is true. 

On the other hand, Weil in \cite{Weil} (see also Theorem 6.19 in \cite{Raghunathan}) proved that the set of discrete such embeddings 
$\left\{ {{\pi }_{1}}\left( {{\Sigma }} \right)\hookrightarrow \text{PSL}\left( 2,\mathbb{R} \right) \right\} $ is open in the quotient space $\mathsf{\mathcal{R}}\left( \text{PSL}\left( 2,\mathbb{R} \right) \right)$. This open subset is called the \textit{Fricke space}\index{Fricke space} $\mathsf{\mathcal{F}}\left( {{\Sigma }} \right)$ of the topological surface ${{\Sigma }}$. Fricke spaces first appeared in the work of Fricke and Klein \cite{FrKl} defined in terms of Fuchsian groups (see \cite{BeGa} for an expository account). 

The connected components of the representation variety $\mathsf{\mathcal{R}}\left( \text{PSL}\left( 2,\mathbb{R} \right) \right)$ are distinguished in terms of the Euler class $e\left( \rho  \right)$ of a representation $\rho $; such a topological invariant for a representation $\rho $ can be considered in the realm of the Riemann--Hilbert correspondence and the associated flat $\text{PSL}\left( 2,\mathbb{R} \right)$-bundle. 

In \cite{Goldman3}, Goldman showed that this Euler class distinguishes the connected components and takes values in  $\mathbb{Z}\cap \left[ \chi \left( {{\Sigma }} \right),-\chi \left( {{\Sigma }} \right) \right]$. In particular, the Fricke space $\mathsf{\mathcal{F}}\left( {{\Sigma }} \right)$ is identified with the component maximizing this characteristic class (consisting of representations that correspond to holonomies of hyperbolic structures on ${{\Sigma }}$). 

To conclude this discussion about the Teichm\"{u}ller space, the Uniformization Theorem implies that $\mathsf{\mathcal{F}}\left( {{\Sigma }} \right)$ and $\mathsf{\mathcal{T}}\left( {{\Sigma }} \right)$ can be identified, therefore the Teichm\"{u}ller space is a \textit{connected component} of the representation variety $\mathsf{\mathcal{R}}\left( \text{PSL}\left( 2,\mathbb{R} \right) \right)$. In fact, it is one of the two connected components entirely consisting of discrete and faithful representations  $\rho :{{\pi }_{1}}\left( {{\Sigma }} \right)\to \text{PSL}\left( 2,\mathbb{R} \right)$; the other such component is $\mathsf{\mathcal{T}}\left( {{{\bar{\Sigma }}}} \right)$, that is, the Teichm\"{u}ller space of the surface  ${{\bar{\Sigma }}}$ with the opposite orientation. 

Since the representation variety can be considered for any reductive Lie group $G$, it is natural to ask whether there are special connected components of it for \textit{higher rank Lie groups} $G$ than $\text{PSL}\left( 2,\mathbb{R} \right)$, which consist entirely of representations related to significant geometric or dynamical structures on the fixed topological surface. This question leads to the introduction of \textit{higher Teichm\"{u}ller spaces} as we shall see next.

\subsection{Higher Teichm\"{u}ller spaces}\label{subsection_hTs}
Let $\Sigma $ be a closed oriented (topological) surface of genus $g$. The fundamental group of $\Sigma $ is described by
	\[{{\pi }_{1}}\left( \Sigma  \right)=\left\langle {{a}_{1}},{{b}_{1}},\ldots ,{{a}_{g}},{{b}_{g}}\left| \prod{\left[ {{a}_{i}},{{b}_{i}} \right]=1} \right. \right\rangle, \]
where $\left[ {{a}_{i}},{{b}_{i}} \right]={{a}_{i}}{{b}_{i}}a_{i}^{-1}b_{i}^{-1}$ is the commutator. The set of all representations of ${{\pi }_{1}}\left( \Sigma  \right)$ into a connected reductive real Lie group $G$, $\text{Hom}\left( {{\pi }_{1}}\left( \Sigma  \right),G \right)$, can be naturally identified with the subset of ${{G}^{2g}}$ consisting of $2g$-tuples $\left( {{A}_{1}},{{B}_{1}},\ldots ,{{A}_{g}},{{B}_{g}} \right)$ satisfying the algebraic equation $\prod{\left[ {{A}_{i}},{{B}_{i}} \right]}=1$. The group $G$ acts on the space $\text{Hom}\left( {{\pi }_{1}}\left( \Sigma  \right),G \right)$ by conjugation
	\[\left( g\cdot \rho  \right)=g\rho \left( \gamma  \right){{g}^{-1}},\]
where $g\in G$, $\rho \in \text{Hom}\left( {{\pi }_{1}}\left( \Sigma  \right),G \right)$ and $\gamma \in {{\pi }_{1}}\left( \Sigma  \right)$, and the restriction of this action to the subspace $\text{Ho}{{\text{m}}^{\text{red}}}\left( {{\pi }_{1}}\left( \Sigma  \right),G \right)$ of reductive representations provides that the orbit space is Hausdorff. Here, by a reductive representation we mean one that composed with the adjoint representation in the Lie algebra of $G$ can be decomposed as a sum of irreducible representations. When $G$ is algebraic, this is equivalent to the Zariski closure of the image of ${{\pi }_{1}}\left( \Sigma  \right)$ in $G$ being a reductive group. Define the \emph{moduli space of reductive representations of ${{\pi }_{1}}\left( \Sigma  \right)$ into $G$} to be the orbit space
	\[\mathsf{\mathcal{R}}\left( G \right)={\text{Ho}{{\text{m}}^{\text{red}}}\left( {{\pi }_{1}}\left( \Sigma  \right),G \right)}/{G}.\]
The following theorem of Goldman \cite{Goldman4} shows that this space is a real analytic variety and so $\mathsf{\mathcal{R}}\left( G \right)$ is usually called the \emph{character variety}\index{character variety}:
\begin{theorem}[Goldman \cite{Goldman4}]
The moduli space $\mathsf{\mathcal{R}}\left( G \right)$ has the structure of a real analytic variety, which is algebraic if $G$ is algebraic and is a complex variety if $G$ is complex.
\end{theorem}

Higher Teichm\"{u}ller Theory is concerned with the study of the properties of fundamental group representations lying in certain subsets of the character variety $\mathsf{\mathcal{R}}\left( G \right)$, for simple real groups $G$. An abundance of methods from geometry, gauge theory, algebraic geometry and dynamics is used to approach these subsets, many methods of which provided by the non-abelian Hodge theory for the moduli space $\mathsf{\mathcal{R}}\left( G \right)$. The term \textit{higher Teichm\"{u}ller space}\index{higher Teichm\"{u}ller space}\index{Teichm\"{u}ller space! higher} originates in the work of Vladimir Fock and Alexander Goncharov \cite{FG}, who developed a more algebro-geometric approach to Lusztig's notion of total positivity\index{total positivity}\index{positivity! total} in the context of general split real semisimple reductive Lie groups (see \cite{Lu}) and defined positive representations of the fundamental group $\pi_{1}(\Sigma)$ into these groups; among establishing significant geometric properties, Fock and Goncharov construct in \cite{FG} all positive representations and show that they are faithful, discrete and positive hyperbolic. Today, the term refers to connected components of the character variety in a broader sense: 

\begin{definition}
Let $\Sigma$ be a closed connected oriented topological surface of genus $g \ge 2$ and $G$ a semisimple real Lie group. A \textit{higher Teichm\"{u}ller space} is a connected component of the character variety  $\mathsf{\mathcal{R}}\left( G \right)$ that entirely consists of faithful representations with discrete image. 
\end{definition}

Several essential features of higher Teichm\"{u}ller spaces can be traced back to the ideas and work of Thurston. For instance, Thurston's shear coordinates have been extended in this setting by Fock and Goncharov \cite{FG}, and are sometimes called \textit{Fock--Goncharov coordinates}; noncommutative coordinates on the spaces of framed and decorated fundamental group representations for a surface with boundary into the group $\text{Sp}\left( 2n,\mathbb{R}\right)$ have been introduced by Alessandrini, Guichard, Rogozinnikov and Wienhard in \cite{AGRW}. Labourie and McShane \cite{LaMc} studied cross ratios and McShane--Mirzakhani identities in the case $G=\text{PSL}(n, \mathbb{R})$ and gave explicit expressions of these generalized identities in terms of a suitable choice of Fock--Goncharov coordinates; see also the work of Vlamis and Yarmola \cite{VY} for a generalization of Basmajian's identity for Hitchin representations into $\text{PSL}(n, \mathbb{R})$, as well as the article of Fanoni and Pozzetti \cite{FP} for Basmajian-type inequalities for maximal representations $\rho :{{\pi }_{1}}(\Sigma )\to \text{Sp}\left( 2n,\mathbb{R} \right)$. Hitchin and maximal representations, in particular, lie in higher Teichm\"{u}ller spaces and will be briefly reviewed below. Generalizations of the McShane identities for higher Teichm\"{u}ller spaces were obtained by Huang and Sun in \cite{HuSu}; these are expressed in terms of simple root lengths, triple ratios and edge functions. Le in \cite{Le} gave a definition of a higher lamination in the spirit of Thurston for the space of framed $G$-local systems over $\Sigma$ and showed that this coincides with the approach of Fock and Goncharov  \cite{FG} as the tropical points of a higher Teichm\"{u}ller space. Another example is the pressure metric for Anosov representations from \cite{BCLS}, \cite{BCS}, which can be viewed as a generalization of the Weil--Peterson metric on the Teichm\"{u}ller space as seen by Thurston. Moreover, generalizations of the Collar Lemma from hyperbolic geometry to Hitchin representations and to maximal representations have been also considered in \cite{LeeZh} and \cite{BuPo} respectively (see also \cite{BePo}).
 
Examples, however, of higher Teichm\"{u}ller spaces appeared long before the term was invented. For an adjoint split real semisimple Lie group $G$, there exists a unique embedding $\pi :\text{SL}\left( 2,\mathbb{R} \right)\to G$, which is the associated Lie group homomorphism to a principal 3-dimensional subalgebra of $\mathfrak{g}$, Kostant's principal subalgebra $\mathfrak{sl}\left( 2,\mathbb{R} \right)\subset \mathfrak{g}$ (see \cite{Kos}). For a fixed discrete embedding $\iota :{{\pi }_{1}}\left( \Sigma  \right)\to \text{SL}\left( 2,\mathbb{R} \right)$, Nigel Hitchin in \cite{Hit92} showed that the subspace containing $\pi \circ \iota :{{\pi }_{1}}\left( \Sigma  \right)\to G$ is a connected component and, in fact, topologically trivial of dimension $\left( 2g-2 \right)\dim G$. In the special case when the group is $G=\text{PSL}\left( 2,\mathbb{R} \right)$, this component is the Teichm\"{u}ller space. 

Following the work of Hitchin, it became apparent that the spaces identified, now called \textit{Hitchin components}\index{Hitchin component}, include representations with important geometric features. For instance, Labourie introduced in \cite{Labourie} the notion of an \textit{Anosov representation}\index{Anosov representation}\index{representation!Anosov} and used techniques from dynamical systems to prove (among other essential geometric properties) that representations lying inside the Hitchin component for $G=\text{PSL}\left( n,\mathbb{R} \right)$, $\text{PSp}\left( 2n,\mathbb{R} \right)$ or $\text{PO}\left( n,n+1 \right)$ are faithful with discrete image; we refer the reader to  \cite{BCLS}, \cite{Guichard}, \cite{GW2}, \cite{Labourie2}, \cite{LaMc}, \cite{LeeZh}, \cite{PoSa} for subsequent works on the geometric and dynamical properties of representations in the Hitchin components.

The second family of Lie groups $G$ where components of discrete and faithful representations have been detected, is the family of Hermitian Lie groups of non-compact type, that is, the symmetric space associated to $G$ is an irreducible Hermitian symmetric space of non-compact type. In this case, a characteristic number  called the \textit{Toledo invariant}\index{Toledo invariant} of a representation $\rho :{{\pi }_{1}}\left( \Sigma  \right)\to G$ can be defined as the integer
\[{{T}_{\rho }}:=\left\langle {{\rho }^{*}}\left( {{\kappa }_{G}} \right),\left[ \Sigma  \right] \right\rangle, \]
where ${{\rho }^{*}}( {{\kappa }_{G}})$ is the pullback of the K\"{a}hler class ${{\kappa }_{G}}\in H_{c}^{2}\left( G,\mathbb{R} \right)$ of $G$ and $\left[ \Sigma  \right]\in {{H}_{2}}\left( \Sigma ,\mathbb{R} \right)$ is the orientation class. The absolute value of the Toledo invariant has an upper bound of Milnor--Wood type
\begin{equation}\label{Milnor_Wood}
\left| {{T}_{\rho }} \right|\le \left( 2g-2 \right)\text{rk}\left( G \right)
\end{equation}
and a representation $\rho :{{\pi }_{1}}\left( \Sigma  \right)\to G$ is called \textit{maximal}\index{maximal representation}\index{representation!maximal} when this upper bound is achieved. Subspaces of maximal representations also have interesting geometric and dynamical properties and, in particular, consist entirely of discrete and faithful representations, as seen in \cite{BILW} and \cite{BIW}. 

It is also interesting to note at this point that in the case when the group  $G$ is the group $\text{PSL}\left( 2,\mathbb{R} \right)$, the Toledo invariant\index{Toledo invariant} is actually the Euler class, Inequality (\ref{Milnor_Wood}) is the Milnor--Wood inequality for the Euler class and the space of maximal representations in this case is identified with the Teichm\"{u}ller space, as in \cite{Goldman3}. 

We refer the reader to the survey articles of Wienhard \cite{Wienhard} and Pozzetti \cite{Pozzetti} for a broader presentation of the geometric properties of higher Teichm\"{u}ller spaces, as well as for an overview of the similarities and differences between these spaces and the classical  Teichm\"{u}ller space.

\subsection{$\Theta$-positive representations}

The special connected components introduced for the two families of Lie groups above, namely the adjoint split real semisimple Lie groups and the Hermitian Lie groups on non-compact type share (among many other fundamental properties) a common characterization that relates to the existence of a continuous equivariant map sending positive triples in $\mathbb{RP}^{1}$ to positive triples in certain flag varieties associated with the Lie group $G$. This property was identified by Labourie \cite{Labourie}, Guichard \cite{Guichard} and Fock--Goncharov \cite{FG} in the case of split semisimple real Lie groups, and by Burger--Iozzi--Wienhard \cite{BIW} for Hermitian Lie groups of non-compact type. 

This in turn provided the motivation to propose in \cite{GW} that the characterization above in terms of positivity can, in fact, distinguish \textit{all} higher Teichm\"{u}ller spaces\index{higher Teichm\"{u}ller space}\index{Teichm\"{u}ller space!higher}. We next include more details about this general conjectural picture; for complete references the reader is directed to the original article of Guichard and Wienhard \cite{GW}. 

The definition of a $\Theta$-positive structure\index{Theta positive structure}\index{positive structure!Theta} for a real semisimple Lie group $G$ is a generalization of Lusztig's total positivity condition in \cite{Lu} and is given in regards to properties of the Lie algebra of parabolic subgroups $P_{\Theta}<G$ defined by a subset of simple positive roots $\Theta \subset \Delta $. In these terms, let  ${{\mathfrak{u}}_{\Theta }}:=\sum\limits_{\alpha \in \Sigma _{\Theta }^{+}}{{{\mathfrak{g}}_{\alpha }}}$, for $\Sigma _{\Theta }^{+}={{\Sigma }^{+}}\backslash \text{Span}\left( \Delta -\Theta  \right)$, where ${{\Sigma }^{+}}$ denotes the set of positive roots, and then the standard parabolic subgroup  ${{P}_{\Theta }}$ associated to  $\Theta \subset \Delta $ is the normalizer in $G$ of ${{\mathfrak{u}}_{\Theta }}$. The group ${{P}_{\Theta }}$ is the semidirect product of its unipotent radical ${{U}_{\Theta }}:=\text{exp}\left( {{\mathfrak{u}}_{\Theta }} \right)$. Consider the Levi subgroup ${{L}_{\Theta }}:={{P}_{\Theta }}\cap P_{\Theta }^{\textit{opp}}$, where $P_{\Theta }^{\textit{opp}}$ is the normalizer in $G$ of $\mathfrak{u}_{\Theta }^{\textit{opp}}:=\sum\limits_{\alpha \in \Sigma _{\Theta }^{+}}{{{\mathfrak{g}}_{-\alpha }}}$. The Levi factor ${{L}_{\Theta }}$ acts on  ${{\mathfrak{u}}_{\Theta }}$ via the adjoint action. Denote by $L_{\Theta }^{0}$ the component of ${{L}_{\Theta }}$ containing the identity. 

For ${{\mathfrak{z}}_{\Theta }}$, the center of the Lie algebra ${{\mathfrak{l}}_{\Theta }}:=\text{Lie}\left( {{L}_{\Theta }} \right)$,  ${{\mathfrak{u}}_{\Theta }}$ can be decomposed into weight spaces
	\[{{\mathfrak{u}}_{\Theta }}=\sum\limits_{\beta \in \mathfrak{z}_{\Theta }^{*}}{{{\mathfrak{u}}_{\beta }}},\] 
where ${{\mathfrak{u}}_{\beta }}:=\left\{ N\in {{\mathfrak{u}}_{\Theta }}\left| \text{ad}\left( Z \right)N=\beta \left( Z \right)N,\text{ for every }Z\in {{\mathfrak{z}}_{\Theta }} \right. \right\}$. 

\begin{definition}[Guichard--Wienhard, Definition 4.2 in \cite{GW}]
Let $G$ be a semisimple Lie group with finite center and $\Theta \subset \Delta $ a subset of simple roots. The group $G$ admits a \emph{$\Theta $-positive structure} if for all  $\beta \in \Theta $, there exists an  $L_{\Theta }^{0}$-invariant sharp convex cone in  ${{\mathfrak{u}}_{\beta }}$.  
\end{definition}  

A central result in \cite{GW} provides that the semisimple Lie groups $G$ that can admit a $\Theta$-positive structure are classified as follows: 
\begin{theorem}[Guichard--Wienhard, Theorem 4.3 in \cite{GW}]\label{classification}
A semisimple Lie group $G$ admits a $\Theta$-positive structure if and only if the pair $\left( G, \Theta \right)$ belongs to one of the following four cases:
\begin{enumerate}
\item $G$ is a split real form and $\Theta = \Delta$.
\item $G$ is a Hermitian symmetric Lie group of tube type and $\Theta =\left\{ {{\alpha }_{r}} \right\}$.
\item $G$ is a Lie group locally isomorphic to a group $\text{SO}\left( p,q \right)$, for $p\ne q$, and $\Theta =\left\{ {{\alpha }_{1}},\ldots ,{{\alpha }_{p-1}} \right\}$.
\item $G$ is a real form of the groups ${{F}_{4}}$, ${{E}_{6}}$, ${{E}_{7}}$, ${{E}_{8}}$ with restricted root system of type ${{F}_{4}}$, and $\Theta =\left\{ {{\alpha }_{1}},{{\alpha }_{2}} \right\}$.
\end{enumerate}
\end{theorem}

In order to define the notion of a positive triple in the generalized flag variety ${G}/{{{P}_{\Theta }}}\;$ for a semisimple Lie group $G$ with a $\Theta $-positive structure, one needs to introduce the notion of a $\Theta $-positive semigroup. First, associate to $\Theta$ a subgroup $W(\Theta )$ of the Weyl group $W$ as follows: The group $W$ is generated by the reflections ${{s}_{\alpha }}$, for $\alpha \in \Delta $; set ${{\sigma }_{\beta }}={{s}_{\beta }}$ for all $\beta \in \Theta -\left\{ {{\beta }_{\Theta }} \right\}$ and define ${{\sigma }_{{{\beta }_{\Theta }}}}$ to be the longest element of the Weyl group ${{W}_{\left\{ {{\beta }_{\Theta }} \right\}\cup \left( \Delta -\Theta  \right)}}$ of the sub-root system generated by $\left\{ {{\beta }_{\Theta }} \right\}\cup \left( \Delta -\Theta  \right)$. Define now the subgroup of $W$, 
	\[W(\Theta )={{\left\langle {{\sigma }_{\beta }} \right\rangle }_{\beta \in \Theta }}.\]
The group $W(\Theta )$ acts on the weight spaces ${{\mathfrak{u}}_{\Theta }}$, for $\beta \in \text{span}(\Theta )$. Denote by $w_{\Theta }^{0}$, the longest element in $W(\Theta )$ and consider a reduced expression $w_{\Theta }^{0}={{\sigma }_{{{i}_{1}}}}\cdots {{\sigma }_{{{i}_{l}}}}$. Then, for $c_{\beta }^{0}\subset {{\mathfrak{u}}_{\Theta }}$, the interior of the $L_{\Theta }^{0}$-invariant closed convex cone, there is a map  for every $\beta \in \Theta $ defined by
\begin{align*}	
{{F}_{{{\sigma }_{{{i}_{1}}}}\cdots {{\sigma }_{{{i}_{l}}}}}}:c_{{{\beta }_{{{i}_{1}}}}}^{0}\times \cdots \times c_{{{\beta }_{{{i}_{l}}}}}^{0} & \to {{U}_{\Theta }}\\
\left( {{v}_{{{i}_{1}}}},\ldots ,{{v}_{{{i}_{l}}}} \right) & \mapsto {{\chi }_{{{\beta }_{{{i}_{1}}}}}}\left( {{v}_{{{i}_{1}}}} \right)\cdot \ldots \cdot {{\chi }_{{{\beta }_{{{i}_{l}}}}}}\left( {{v}_{{{i}_{l}}}} \right),
\end{align*}
where, for any $\beta \in \Theta $, the map ${{\chi }_{\beta }}:{{\mathfrak{u}}_{\Theta }}\to {{U}_{\beta }}\subset {{U}_{\Theta }}$ with $v\mapsto \text{exp}\left( v \right)$ is considered. The $\Theta $-positive semigroup of ${{U}_{\Theta }}$ is now defined as follows:

\begin{theorem}[Guichard--Wienhard, Theorem 4.5 in \cite{GW}] The image $U_{\Theta }^{>0}$ of the map ${{F}_{{{\sigma }_{{{i}_{1}}}}\cdots {{\sigma }_{{{i}_{l}}}}}}$ defined above is independent of the reduced expression of $w_{\Theta }^{0}$. 
\end{theorem}
One may now define positive triples in the generalized flag variety:

\begin{definition}
Fix ${{E}_{\Theta }}$ and  ${{F}_{\Theta }}$ to be the standard flags in ${G}/{{{P}_{\Theta }}}\;$ such that $\text{Sta}{{\text{b}}_{G}}\left( {{F}_{\Theta }} \right)={{P}_{\Theta }}$ and $\text{Sta}{{\text{b}}_{G}}\left( {{E}_{\Theta }} \right)=P_{\Theta }^\textit{opp}$. For any  ${{S}_{\Theta }}\in {G}/{{{P}_{\Theta }}}\;$ transverse to  ${{F}_{\Theta }}$, there exists ${{u}_{{{S}_{\Theta }}}}\subset {{U}_{\Theta }}$ such that  ${{S}_{\Theta }}={{u}_{{{S}_{\Theta }}}}{{E}_{\Theta }}$. The triple  $\left( {{E}_{\Theta }},{{S}_{\Theta }},{{F}_{\Theta }} \right)$ in the generalized flag variety  ${G}/{{{P}_{\Theta }}}\;$ will be called \emph{$\Theta $-positive}, if  ${{u}_{{{S}_{\Theta }}}}\in U_{\Theta }^{>0}$, for  $U_{\Theta }^{>0}$ the $\Theta $-positive semigroup of  ${{U}_{\Theta }}$. 
\end{definition}

The definition of a $\Theta$-positive fundamental group representation is now the following: 

\begin{definition}[Guichard--Wienhard, Definition 5.3 in \cite{GW}]
Let $\Sigma $ be a closed connected and oriented topological surface of genus $g\ge 2$ and let $G$ be a semisimple Lie group admitting a $\Theta $-positive structure. A representation of the fundamental group of  ${{\Sigma }}$ into $G$ will be called  \emph{$\Theta $-positive}\index{Theta positive representation}\index{representation!Theta positive}, if there exists a  $\rho $-equivariant positive map  $\xi :\partial {{\pi }_{1}}\left( {{\Sigma }} \right)=\mathbb{R}{{\mathbb{P}}^{1}}\to {G}/{{{P}_{\Theta }}}\;$ sending positive triples in  $\mathbb{R}{{\mathbb{P}}^{1}}$ to  $\Theta $-positive triples in  ${G}/{{{P}_{\Theta }}}\;$. 
\end{definition}

In their recent article \cite{GLW}, Guichard, Labourie and Wienhard show that $\Theta$-positive representations are $\Theta$-Anosov, thus discrete and faithful, and that, in fact, for the four families of semisimple Lie groups $G$ listed in Theorem \ref{classification} above, there are higher Teichm\"{u}ller spaces in the character variety:

\begin{theorem}[Guichard--Labourie--Wienhard, Theorem A in \cite{GLW}]\label{GLW_conjecture}
Let $G$ be a semisimple Lie group that admits a $\Theta$-positive structure. Then there exists a connected component of the representation variety $\mathcal{R}(G)$ that consists solely of discrete and faithful representations.
\end{theorem}

\section{Non-abelian Hodge theory}

A major contribution to the various methods available in order to study higher Teichm\"{u}ller spaces involves fixing a complex structure $J$ on the topological surface $\Sigma$, thus transforming $\Sigma$ into a Riemann surface $X=\left( \Sigma, J \right)$, therefore opening the way to holomorphic techniques and the theory of \textit{Higgs bundles}, as initiated by Nigel Hitchin in his article \textit{The self duality equations on a Riemann surface} published in 1987 \cite{Hit87}. The non-abelian Hodge theory correspondence provides a real-analytic isomorphism between the character variety $\mathsf{\mathcal{R}}\left( G \right)$ and the moduli space of polystable $G$-Higgs bundles, which we briefly introduce next. 

\subsection{Moduli spaces of $G$-Higgs bundles}

Let $X$ be a compact Riemann surface and let $G$ be a real reductive group. The latter involves considering \emph{Cartan data}\index{Cartan data} $\left( G,H,\theta ,B \right)$, where $H\subset G$ is a maximal compact subgroup, $\theta :\mathfrak{g}\to \mathfrak{g}$ is a Cartan involution and $B$ is a non-degenerate bilinear form on $\mathfrak{g}$ which is $\text{Ad}\left( G \right)$-invariant and $\theta $-invariant. The Cartan involution $\theta$ gives a decomposition (called the \emph{Cartan decomposition})
  \[\mathfrak{g}=\mathfrak{h}\oplus \mathfrak{m}\]
into its $\pm 1$-eigenspaces, where $\mathfrak{h}$ is the Lie algebra of $H$.

Let ${{H}^{\mathbb{C}}}$ be the complexification of $H$ and let ${{\mathfrak{g}}^{\mathbb{C}}}={{\mathfrak{h}}^{\mathbb{C}}}\oplus {{\mathfrak{m}}^{\mathbb{C}}}$ be the complexification of the Cartan decomposition. The adjoint action of $G$ on $\mathfrak{g}$ restricts to give a representation (the isotropy representation) of $H$ on $\mathfrak{m}$. This is independent of the choice of Cartan decomposition, since any two Cartan decompositions of $G$ are related by a conjugation using also that $\left[ \mathfrak{h},\mathfrak{m} \right]\subseteq \mathfrak{m}$. The action of $H$ extends to a linear holomorphic action of $H^{\mathbb{C}}$ on $\mathfrak{m}^{\mathbb{C}}$, thus providing the complexified isotropy representation $\iota :{{H}^{\mathbb{C}}}\to \text{GL(}{{\mathfrak{m}}^{\mathbb{C}}}\text{)}$. This introduces the following definition:

\begin{definition}
Let $K\cong {{T}^{*}}X$ be the canonical line bundle over a compact Riemann surface $X$. A \emph{$G$-Higgs bundle}\index{Higgs bundle} is a pair $\left( E,\varphi  \right)$ where
\begin{itemize}
  \item $E$ is a principal holomorphic ${{H}^{\mathbb{C}}}$-bundle over $X$ and
  \item $\varphi $ is a holomorphic section of the vector bundle $E\left( {{\mathfrak{m}}^{\mathbb{C}}} \right)\otimes K=\left( E{{\times }_{\iota }}{{\mathfrak{m}}^{\mathbb{C}}} \right)\otimes K$.
\end{itemize}
The section $\varphi $ is called the \emph{Higgs field}. Two $G$-Higgs bundles $\left( E,\varphi  \right)$ and $\left( {E}',{\varphi }' \right)$ are said to be \emph{isomorphic} if there is a principal bundle isomorphism $E\cong {E}'$ which takes $\varphi $ to ${\varphi }'$ under the induced isomorphism $E\left( {{\mathfrak{m}}^{\mathbb{C}}} \right)\cong {E}'\left( {{\mathfrak{m}}^{\mathbb{C}}} \right)$.
\end{definition}

To define a moduli space of $G$-Higgs bundles we need to consider a notion of semistability, stability and polystability. These notions are defined in terms of an antidominant character for a parabolic subgroup ${{P}}\subseteq {{H}^{\mathbb{C}}}$ and a holomorphic reduction $\sigma $ of the structure group of the bundle $E$ from ${{H}^{\mathbb{C}}}$ to ${{P}}$ (see \cite{GGMHitchin-Kob} for the precise definitions).

When the group $G$ is connected, principal ${{H}^{\mathbb{C}}}$-bundles $E$ are topologically classified by a characteristic class  $c\left( E \right)\in {{H}^{2}}\left( X,{{\pi }_{1}}\left( {{H}^{\mathbb{C}}} \right) \right)\cong {{\pi }_{1}}\left( {{H}^{\mathbb{C}}} \right) \cong {{\pi }_{1}}\left( H \right) \cong {{\pi }_{1}}\left( G \right)$.
\begin{definition}
For a fixed class $d\in {{\pi }_{1}}\left( G \right)$, the \emph{moduli space of polystable $G$-Higgs bundles} of fixed topological class $d$ with respect to the group of complex gauge transformations is defined as the set of isomorphism classes of polystable $G$-Higgs bundles $\left( E,\varphi  \right)$ such that $c\left( E \right)=d$. We will denote this set by ${{\mathsf{\mathcal{M}}}_{d}}\left( G \right)$.
\end{definition}

Using the general GIT constructions of Schmitt for decorated principal bundles in the case of a real form of a complex reductive algebraic Lie group, it is shown that the moduli space ${{\mathsf{\mathcal{M}}}_{d}}\left( G \right)$ is an algebraic variety. The expected dimension of the moduli space of $G$-Higgs bundles is $\left( g-1 \right)\dim{{G}^{\mathbb{C}}}$, in the case when $G$ is a connected semisimple real Lie group; see \cite{GGMHitchin-Kob}, \cite{Schmitt1}, \cite{Schmitt2} for details.

\subsection{$G$-Hitchin equations} Let $\left( E,\varphi  \right)$ be a $G$-Higgs bundle over a compact Riemann surface $X$. By a slight abuse of notation we shall denote the underlying smooth objects of $E$ and $\varphi $ by the same symbols. The Higgs field can be thus viewed as a $\left( 1,0 \right)$-form $\varphi \in {{\Omega }^{1,0}}\left( E\left( {{\mathfrak{m}}^{\mathbb{C}}} \right) \right)$. Given a reduction $h$ of structure group to $H$ in the smooth ${{H}^{\mathbb{C}}}$-bundle $E$, we denote by ${{F}_{h}}$ the curvature of the unique connection compatible with $h$ and the holomorphic structure on $E$. Let ${{\tau }_{h}}:{{\Omega }^{1,0}}\left( E\left( {{\mathfrak{g}}^{\mathbb{C}}} \right) \right)\to {{\Omega }^{0,1}}\left( E\left( {{\mathfrak{g}}^{\mathbb{C}}} \right) \right)$ be defined by the compact conjugation of ${{\mathfrak{g}}^{\mathbb{C}}}$ which is given fiberwise by the reduction $h$, combined with complex conjugation on complex 1-forms. The next theorem was proved in \cite{GGMHitchin-Kob} for an arbitrary reductive real Lie group $G$.

\begin{theorem}[Hitchin--Kobayashi correspondence, Theorem 3.21 in \cite{GGMHitchin-Kob}]\label{Hitchin_Kobayashi} There exists a reduction $h$ of the structure group of $E$ from ${{H}^{\mathbb{C}}}$ to $H$ satisfying the Hitchin equation\index{Hitchin equation}
	\[{{F}_{h}}-\left[ \varphi ,{{\tau }_{h}}\left( \varphi  \right) \right]=0\]
if and only if $\left( E,\varphi  \right)$ is polystable.
\end{theorem}
From the point of view of moduli spaces it is convenient to fix a ${{C}^{\infty }}$ principal $H$-bundle ${{\textbf{E}}_{H}}$ with fixed topological class $d\in {{\pi }_{1}}\left( H \right)$ and study the moduli space of solutions to Hitchin's equations for a pair $\left( A,\varphi  \right)$ consisting of an $H$-connection $A$ and $\varphi \in {{\Omega }^{1,0}}\left( X,\textbf{E}_{H}\left( {{\mathfrak{m}}^{\mathbb{C}}} \right) \right)$ with
\begin{align*}
{{F}_{A}}-\left[ \varphi ,\tau \left( \varphi  \right) \right] &= 0 \tag{*} \\
{{\bar{\partial }}_{A}}\varphi &= 0
\end{align*}
where ${{d}_{A}}$ is the covariant derivative associated with $A$ and ${{\bar{\partial }}_{A}}$ is the $\left( 0,1 \right)$-part of ${{d}_{A}}$, defining the holomorphic structure on ${{\textbf{E}}_{H}}$. Also, $\tau $ is defined by the fixed reduction of structure group ${\textbf{E}_{H}}\hookrightarrow {\textbf{E}_{H}}\left( {{H}^{\mathbb{C}}} \right)$. The gauge group ${{\mathsf{\mathcal{G}}}_{H}}$ of ${\textbf{E}_{H}}$ acts on the space of solutions by conjugation and the moduli space of solutions is defined by
	\[\mathsf{\mathcal{M}}_{d}^{\text{gauge}}\left( G \right):={\left\{ \left( A,\varphi  \right)\text{ satisfying  equations (*)} \right\}}/{{{\mathsf{\mathcal{G}}}_{H}}}\;.\]
Now, Theorem \ref{Hitchin_Kobayashi} implies that there is a homeomorphism
\[{{\mathsf{\mathcal{M}}}_{d}}\left( G \right)\cong \mathsf{\mathcal{M}}_{d}^{\text{gauge}}\left( G \right).\]

Using the one-to-one correspondence between $H$-connections on ${\textbf{E}_{H}}$ and $\bar{\partial }$-operators on ${\textbf{E}_{{{H}^{\mathbb{C}}}}}$, the homeomorphism in the above theorem can be interpreted as saying that in the  $\mathsf{\mathcal{G}}_{H}^{\mathbb{C}}$-orbit of a polystable $G$-Higgs bundle $\left( {{{\bar{\partial }}}_{{{E}_{0}}}},{{\varphi }_{0}} \right)$ we can find another Higgs bundle $\left( {{{\bar{\partial }}}_{E}},\varphi  \right)$ whose corresponding pair $\left( {{d}_{A}},\varphi  \right)$ satisfies the equation ${{F}_{A}}-\left[ \varphi ,\tau \left( \varphi  \right) \right]=0$, and this is unique up to $H$-gauge transformations.

\subsection{The non-abelian Hodge correspondence}
We can assign a topological invariant to a representation $\rho \in \mathsf{\mathcal{R}}\left( G \right)$ by considering its corresponding flat $G$-bundle on $\Sigma$ defined as ${{E}_{\rho }}=\tilde{\Sigma}{{\times }_{\rho }}G$. Here $\tilde{\Sigma}\to \Sigma$ is the universal cover and ${{\pi }_{1}}\left( \Sigma  \right)$ acts on $G$ via $\rho$. A topological invariant is then given by the characteristic class $c\left( \rho  \right):=c\left( {{E}_{\rho }} \right)\in {{\pi }_{1}}\left( G \right)\simeq {{\pi }_{1}}\left( H \right)$, for $H\subseteq G$ a maximal compact subgroup of $G$. For a fixed $d\in {{\pi }_{1}}\left( G \right)$ the moduli space of reductive representations with fixed topological invariant $d$ is now defined as the subvariety
	\[{{\mathsf{\mathcal{R}}}_{d}}\left( G \right):=\left\{ \left[ \rho  \right]\in \mathsf{\mathcal{R}}\left( G \right)\left| c\left( \rho  \right)=d \right. \right\}.\]
A reductive fundamental group representation corresponds to a solution to the Hitchin equations. This is seen using that any solution $h$ to Hitchin's equations defines a flat reductive $G$-connection
\begin{equation}\label{reductive}
  D={{D}_{h}}+\varphi -\tau \left( \varphi  \right),
\end{equation} where ${{D}_{h}}$ is the unique $H$-connection on $E$ compatible with its holomorphic structure. Conversely, given a flat reductive connection $D$ on a $G$-bundle $E_{G}$, there exists a harmonic metric, in other words, a reduction of structure group to $H\subset G$ corresponding to a harmonic section of ${{{E}_{G}}}/{H}\;\to X$. This reduction produces a solution to Hitchin's equations such that Equation (\ref{reductive}) holds. 

In summary, equipping the surface $\Sigma$ with a complex structure $J$, a reductive representation of ${{\pi }_{1}}\left( \Sigma  \right)$  into $G$ corresponds to a polystable $G$-Higgs bundle over the Riemann surface $X=\left(\Sigma, J\right)$; this is the content of \textit{non-abelian Hodge correspondence}\index{non-abelian Hodge correspondence}; its proof is based on combined work by Hitchin \cite{Hit87}, Simpson \cite{Simpson-variations}, \cite{Simpson-Higgs}, Donaldson \cite{Donaldson} and Corlette \cite{Corlette}:
\begin{theorem}[Non-abelian Hodge correspondence]\label{naHc}
Let $G$ be a connected semisimple real Lie group with maximal compact subgroup $H\subseteq G$ and let $d\in {{\pi }_{1}}\left( G \right)\simeq {{\pi }_{1}}\left( H \right)$. Then there exists a homeomorphism
	\[{{\mathsf{\mathcal{R}}}_{d}}\left( G \right)\cong {{\mathsf{\mathcal{M}}}_{d}}\left( G \right).\]
\end{theorem}

The introduction of holomorphic techniques via the non-abelian Hodge correspondence allows the description of a theory of higher Teichm\"{u}ller spaces from the Higgs bundle point of view. In \cite{BCGGP}, Bradlow, Collier, Garc\'{i}a-Prada, Gothen and Oliviera obtain a parameterization of special components of the moduli space of Higgs bundles on a compact Riemann surface using the decomposition data for a complex simple Lie algebra $\mathfrak{g}$. The possible decompositions of $\mathfrak{g}$ are defined by a newly introduced class of $\mathfrak{sl}\left( 2, \mathbb{R} \right)$-triples, and the classification of these triples is shown to be in bijection with the classification of the $\Theta$-positive structures of Guichard and Wienhard (Theorem \ref{classification}). We refer to \cite{BCGGP} for the precise statements; see also the survey article of Garc\'{i}a-Prada \cite{GPsurvey} for a broader description of the results for higher Teichm\"{u}ller spaces that can be obtained using the theory of Higgs bundles.

\section{Surgeries in representation varieties-General theory}\label{general_gluing}

We next describe a gluing construction for points of the moduli spaces appearing in the non-abelian Hodge correspondence. In particular, this technique can be used to obtain specific model objects of the moduli spaces which are hard to be constructed otherwise and can be used to improve our understanding of the geometric properties of the subsets of the character variety they live in.

\subsection{Topological gluing construction}

For a closed oriented surface $\Sigma$ of genus $g$, let $\Sigma ={{\Sigma }_{l}}{{\cup }_{\gamma }}{{\Sigma }_{r}}$ be a decomposition of $\Sigma $ along one simple closed oriented separating geodesic $\gamma$ into two subsurfaces, say ${{\Sigma }_{l}}$ and ${{\Sigma }_{r}}$. Let now ${{\rho }_{l}}:{{\pi }_{1}}\left( {{\Sigma }} \right)\to G$ and ${{\rho }_{r}}:{{\pi }_{1}}\left( {{\Sigma }} \right)\to G$ be two representations into a semisimple Lie group $G$. 

One could amalgamate the restriction of ${{\rho }_{l}}$ to ${{\Sigma }_{l}}$ with the restriction of ${{\rho }_{r }}$ to ${{\Sigma }_{r}}$, however the holonomies of those along $\gamma $ do not have to agree a priori. If the holonomies do agree (possibly after applying a deformation of at least one of the two representations for the holonomies to match up), then one can introduce new representations by gluing with a use of the van Kampen theorem at the level of topological surfaces, as follows.

\begin{definition}  A \emph{hybrid representation}\index{hybrid representation}\index{representation!hybrid} is defined as the amalgamated representation
\[\rho :={{\rho }_{l}}\left| _{{{\pi }_{1}}\left( {{\Sigma }_{l}} \right)} \right.*{{\rho }_{r}}\left| _{{{\pi }_{1}}\left( {{\Sigma }_{r}} \right)} \right.:{{\pi }_{1}}\left( \Sigma  \right)\simeq {{\pi }_{1}}\left( {{\Sigma }_{l}} \right){{*}_{\left\langle \gamma  \right\rangle }}{{\pi }_{1}}\left( {{\Sigma }_{r}} \right)\to G.\]
\end{definition}

\begin{remark}
The assumption that the holonomies agree over the boundary is crucial. In \S 3.3.1 of \cite{GW3}, Guichard and Wienhard provide an explicit example of hybrid representations in the case when the group is the symplectic group $\text{Sp}\left( 4,\mathbb{R} \right)$. Special attention is paid there in order to establish this assumption via an appropriate deformation argument.
\end{remark}

The above construction/definition can be generalized to the case when the subsurfaces ${{\Sigma }_{l}}$ and ${{\Sigma }_{r}}$ are not necessarily connected. For $\Sigma$ as earlier, let ${{\Sigma }_{1}}\subset \Sigma $ denote a subsurface with Euler characteristic $\chi ({{\Sigma }_{1}})\le -1$. The (nonempty) boundary of ${{\Sigma }_{1}}$ is a union of disjoint circles 
\[\partial {{\Sigma }_{1}}=\coprod\limits_{d\in {{\pi }_{0}}\left( \partial {{\Sigma }_{1}} \right)}{{{\gamma }_{d}}}.\]
The circles $\gamma_{d}$ are oriented so that for each $d$, the surface ${{\Sigma }_{1}}$ lies on the left of ${{\gamma }_{d}}$. Now, write
	\[\Sigma \backslash \partial {{\Sigma }_{1}}=\bigcup\limits_{c\in {{\pi }_{0}}\left( \Sigma \backslash \partial {{\Sigma }_{1}} \right)}{{{\Sigma }_{c}}}.\]
Then, for any $d \in {{\pi }_{0}}( \partial {{\Sigma }_{1}})$, the curve ${{\gamma }_{d}}$ bounds exactly two connected components of $\Sigma \backslash \partial {{\Sigma }_{1}}$, namely, one is included in ${{\Sigma }_{1}}$ and denoted by ${{\Sigma }_{l( d )}}$ with $l( d )\in {{\pi }_{0}}( {{\Sigma }_{1}})$, while the other is included in the complement of $\Sigma_{1}$ and is denoted by ${{\Sigma }_{r( d )}}$ with $r( d)\in {{\pi }_{0}}( \Sigma \backslash {{\Sigma }_{1}} )$. In this way, we have  $l( d ),r( d)\in {{\pi }_{0}}( \Sigma \backslash \partial {{\Sigma }_{1}})$, but it can be that $l( d )=l( {{d}'} )$ or that $r( d )=r( {{d}'} )$,  for $d\ne {d}'$. 

Assume now that the graph with vertex set ${{\pi }_{0}}( \Sigma \backslash {{\Sigma }_{1}} )$ and edges given by the pairs ${{\left\{ l( d),r( d) \right\}}_{d\in {{\pi }_{0}}( \partial {{\Sigma }_{1}} )}}$ is a tree. This allows us to apply a generalized van Kampen theorem argument and write the fundamental group ${{\pi }_{1}}( \Sigma  )$ as the amalgamated product of the groups ${{\pi }_{1}}( {{\Sigma }_{c}} )$, for all $c\in {{\pi }_{0}}( \Sigma \backslash \partial {{\Sigma }_{1}} )$ over the groups ${{\pi }_{1}}( {{\gamma }_{d}} )$, for all $d\in {{\pi }_{0}}( \partial {{\Sigma }_{1}} )$.

Pick a family of representations $\left\{{{\rho }_{c}}:{{\pi }_{1}}( \Sigma_{c} )\to G \right\}_{c \in {{\pi }_{0}}( \Sigma \backslash \partial {{\Sigma }_{1}})}$ subordinate to the following condition: there exist elements ${{g}_{c}}\in G$ for each $c\in {{\pi }_{0}}( \Sigma \backslash \partial {{\Sigma }_{1}})$, such that for any $d\in {{\pi }_{0}}( \partial {{\Sigma }_{1}})$ it holds that 
	\[{{g}_{l( d)}}{{\rho }_{l( d )}}( {{\gamma }_{d}} )g_{l( d )}^{-1}={{g}_{r( d )}}{{\rho }_{r( d )}}( {{\gamma }_{d}} )g_{r( d )}^{-1}.\]
Then one may construct a hybrid representation $\rho :{{\pi }_{1}}( \Sigma )\to G$ by amalgamating the representations ${{g}_{c}}{{\rho }_{c}}g_{c}^{-1}$, for each $c\in {{\pi }_{0}}( \Sigma \backslash \partial {{\Sigma }_{1}})$. An explicit example of amalgamation of a family of representations that satisfy the condition above is provided in \S 3.3.2 of \cite{GW3} in the case when $G= \text{Sp} (4, \mathbb{R})$.

\subsection{Gluing in exceptional components of the moduli space}

Motivated by the amalgamation construction for representations and in the realm of the non-abelian Hodge correspondence, one may seek for an analogous gluing construction from a holomorphic point of view. The benefit from establishing this method in the Higgs bundle moduli space is that it is easier to compute the Higgs bundle invariants for any models constructed in order to identify in which connected component these new objects lie. Indeed, for the cases when the Lie group is the group $\text{Sp}\left( 4,\mathbb{R} \right)$ or $\text{SO}\left( p, p+1 \right)$ the moduli space has a number of exceptional components\index{exceptional components} in terms of their topological and geometric properties; these exceptional components do, in fact, fall in the class of higher Teichm\"{u}ller spaces. It is for such components that a gluing construction for Higgs bundles can provide good models that are not easily obtained otherwise, thus allowing us to study more closely the components themselves. Examples of models in the case of the group $\text{Sp}\left( 4,\mathbb{R} \right)$ were obtained in \cite{Kydon-article} (see also \cite{Kydonakis}), while for $G = \text{SO}\left( p, p+1 \right)$ we will demonstrate some examples in \S \ref{Examples-orthogonal} later on. 

\subsubsection{Parabolic $\text{GL}\left( n,\mathbb{C} \right)$-Higgs bundles}\label{parabolic_data}
Remember that the amalgamation method involved fundamental group representations defined over a surface with boundary. The appropriate analog to a surface group representation into a reductive Lie group $G$ for a surface with boundary is a \textit{parabolic $G$-Higgs bundle}\index{parabolic Higgs bundle}\index{Higgs bundle!parabolic} over a Riemann surface with a divisor. This involves an extra layer of structure encoded by a weighted filtration on each fiber of the bundle over a collection of finitely many distinct points of the surface. We include next basic definitions for a parabolic $\text{GL}\left( n, \mathbb{C} \right)$-Higgs bundle; concrete examples of such pairs will be studied later on in \S \ref{Examples-orthogonal}.

Parabolic vector bundles over Riemann surfaces with marked points were introduced by Conjeeveram S. Seshadri in \cite{Seshadri} and similar to the Narasimhan--Seshadri correspondence, there is an analogous correspondence between stable parabolic bundles and unitary representations of the fundamental group of the punctured surface with fixed holonomy class around each puncture \cite{MeSe}. Later on, Carlos Simpson in \cite{Simpson-noncompact} proved a non-abelian Hodge correspondence over a \emph{non-compact} curve.

\begin{definition} Let $X$ be a closed, connected, smooth Riemann surface of genus $g\ge 2$ with $s$-many marked points ${{x}_{1}},\ldots ,{{x}_{s}}$ and let a divisor $D=\left\{ {{x}_{1}},\ldots ,{{x}_{s}} \right\}$.  A \emph{parabolic vector bundle} $E$ over $X$ is a holomorphic vector bundle $E\to X$ of rank $n$ with \emph{parabolic structure}\index{parabolic structure}\index{structure!parabolic} at each $x\in D$ (\emph{weighted flag} on each fiber ${{E}_{x}}$ ): \[\begin{matrix}
   {{E}_{x}}={{E}_{x,1}}\supset {{E}_{x,2}}\supset \ldots \supset {{E}_{x,r\left( x \right)+1}}=\left\{ 0 \right\}  \\
   0\le {{\alpha }_{1}}\left( x \right)<\ldots <{{\alpha }_{r\left( x \right)}}\left( x \right)<1.  \\
\end{matrix}\]
\end{definition}
The real numbers ${{\alpha }_{i}}\left( x \right)\in \left[ 0,1 \right)$ for $1\le i\le r\left( x \right)$ are called the \emph{weights} of the subspaces ${{E}_{x}}$ and we usually write $\left( E, \alpha \right)$ to denote a parabolic vector bundle equipped with a parabolic structure determined by a system of weights $\alpha \left( x \right)=\left( {{\alpha }_{1}}\left( x \right),\ldots ,{{\alpha }_{r(x)}}\left( x \right) \right)$ at each $x\in D$; whenever the system of weights is not discussed in the context, we will be omitting the notation  $\alpha$ to ease exposition. Moreover, let ${{k}_{i}}\left( x \right)=\dim\left( {{{E}_{x,i}}}/{{{E}_{x,i+1}}}\; \right)$ denote the \emph{multiplicity} of the weight ${{\alpha }_{i}}\left( x \right)$ and notice that $\sum\limits_{i}{{{k}_{i}}}\left( x \right)=n$. A weighted flag shall be called \emph{full}, if ${{k}_{i}}\left( x \right)=1$ for every $1\le i\le r\left( x \right)$  and every $x\in D$.\\
The \textit{parabolic degree} and \textit{parabolic slope} of a vector bundle equipped with a parabolic structure are the real numbers \[\text{par}\deg \left( E \right)=\deg E+\sum\limits_{x\in D}{\sum\limits_{i=1}^{r\left( x \right)}{{{k}_{i}}\left( x \right){{\alpha }_{i}}\left( x \right)}},\]
\[\text{par}\mu \left( \text{E} \right)=\frac{\text{pardeg}\left( E \right)}{\text{rk}\left( E \right)}.\]

\begin{definition}
Let $K$ be the canonical bundle over $X$ and $E$ a parabolic vector bundle. The bundle morphism $\Phi :E\to E\otimes K\left( D \right)$ will be called a \textit{parabolic Higgs field} if it preserves the parabolic structure at each point $x\in D$:
	\[\Phi \left| _{x}\left( {{E}_{x,i}} \right) \right.\subset {{E}_{x,i}}\otimes K\left( D \right)\left| _{x} \right. .\]
In particular, we call $\Phi$ \emph{strongly parabolic} if
\[\Phi \left| _{x}\left( {{E}_{x,i}} \right) \right.\subset {{E}_{x,i+1}}\otimes K\left( D \right)\left| _{x} \right. ,\]
that is, $\Phi \in H^0 (X, \text{End}(E) \otimes K(D))$ is an element with simple poles along the divisor $D$, whose residue at $x\in D$ is nilpotent with respect to the filtration.
\end{definition}
After these considerations we define parabolic Higgs bundles as follows.

\begin{definition}
Let $K$ be the canonical bundle over $X$ and $E$ a parabolic vector bundle over $X$. 
A \textit{parabolic Higgs bundle} is a pair $\left( E,\Phi  \right)$, where $E$ is a parabolic vector bundle and $\Phi :E\to E\otimes K\left( D \right)$ is a strongly parabolic Higgs field.
\end{definition}

Analogously to the non-parabolic case, we may define a notion of stability as follows:

\begin{definition}
A parabolic Higgs bundle will be called \textit{stable}\index{parabolic Higgs bundle!stable} (resp. \textit{semistable}) if for every $\Phi $-invariant parabolic subbundle $F\le E$ we have $\text{par}\mu \left( F \right)<\text{par}\mu \left( E \right)$ (resp. $\le $).
Furthermore, it will be called \emph{polystable} if it is the direct sum of stable parabolic Higgs bundles of the same parabolic slope.
\end{definition}

\subsection{Complex connected sum of Riemann surfaces}
In order to describe how two parabolic Higgs bundles can be glued to a (non-parabolic) Higgs bundle, the first step is to glue their underlying surfaces with boundary as follows.

Take annuli ${{\mathbb{A}}_{1}}=\left\{ z\in \mathbb{C}\left| {{r}_{1}}<\left| z \right|<{{R}_{1}} \right. \right\}$  and ${{\mathbb{A}}_{2}}=\left\{ z\in \mathbb{C}\left| {{r}_{2}}<\left| z \right|<{{R}_{2}} \right. \right\}$  on two copies of the complex plane, and consider the M{\"o}bius transformation ${{f}_{\lambda }}:{{\mathbb{A}}_{1}}\to {{\mathbb{A}}_{2}}$ with ${{f}_{\lambda }}\left( z \right)=\frac{\lambda }{z}$, where $\lambda \in \mathbb{C}$ with $\left| \lambda  \right|={{r}_{2}}{{R}_{1}}={{r}_{1}}{{R}_{2}}$. This is a conformal biholomorphism (equivalently bijective, angle-preserving and orientation-preserving) between the two annuli and such that the continuous extension of the function $z\mapsto \left| {{f}_{\lambda }}\left( z \right) \right|$ to the closure of ${{\mathbb{A}}_{1}}$ reverses the order of the boundary components. 

Consider two compact Riemann surfaces $X_{1}, X_{2}$ of respective genera $g_{1}, g_{2}$. Choose points $p\in {{X}_{1}}$, $q\in {{X}_{2}}$ and local charts around these points ${{\psi }_{i}}:{{U}_{i}}\to \Delta \left( 0,{{\varepsilon }_{i}} \right)$ on ${{X}_{i}}$, for $i=1,2$. Now fix positive real numbers ${{r}_{i}}<{{R}_{i}}<{{\varepsilon }_{i}}$ such that the following two conditions are satisfied:
\begin{itemize}
  \item $\psi _{i}^{-1}\left( \overline{\Delta \left( 0,{{R}_{i}} \right)} \right)\cap {{U}_{j}}\ne \varnothing $, for every ${{U}_{j}}\ne {{U}_{i}}$ from the complex atlas of ${{X}_{i}}$. In other words, we are considering an annulus around each of the $p$ and $q$ contained entirely in the neighborhood of a single chart, and
  \item $\frac{{{R}_{2}}}{{{r}_{2}}}=\frac{{{R}_{1}}}{{{r}_{1}}}.$
\end{itemize}
Set now \[X_{i}^{*}={{X}_{i}}\backslash \psi _{i}^{-1}\left( \overline{\Delta \left( 0,{{r}_{i}} \right)} \right).\]

Choosing the biholomorphism ${{f}_{\lambda }}:{{\mathbb{A}}_{1}}\to {{\mathbb{A}}_{2}}$ as above,  ${{f}_{\lambda }}$ is used to glue the two Riemann surfaces $X_{1}, X_{2}$ along the inverse image of the annuli ${{\mathbb{A}}_{1}},{{\mathbb{A}}_{2}}$ on the surfaces, via the biholomorphism
\[{{g}_{\lambda }}:{{\Omega }_{1}}=\psi _{1}^{-1}\left( {{\mathbb{A}}_{1}} \right)\to {{\Omega }_{2}}=\psi _{2}^{-1}\left( {{\mathbb{A}}_{2}} \right)\] with ${{g}_{\lambda }}=\psi _{2}^{-1}\circ {{f}_{\lambda }}\circ {{\psi }_{1}}$.

Define ${{X}_{\lambda }}={{X}_{1}}{{\#}_{\lambda }}{{X}_{2}}=X_{1}^{*}\coprod{X_{2}^{*}}/\sim$, where the gluing of ${{\Omega }_{1}}$ and ${{\Omega }_{2}}$ is performed through the equivalence relation which identifies $y\in {{\Omega }_{1}}$ with $w\in {{\Omega }_{2}}$ iff $w={{g}_{\lambda }}\left( y \right)$. For collections of $s$-many distinct points $D_{1}$ on $X_{1}$ and $D_{2}$ on $X_{2}$, this procedure is assumed to be taking place for annuli around each pair of points $\left( p,q \right)$ for $p\in {{D}_{1}}$ and $q\in {{D}_{2}}$.

If ${{X}_{1}},{{X}_{2}}$ are orientable and orientations are chosen for both, since ${{f}_{\lambda }}$ is orientation preserving we obtain a natural orientation on the connected sum ${{X}_{1}}\#{{X}_{2}}$ which coincides with the given ones on $X_{1}^{*}$ and $X_{2}^{*}$.

Therefore, ${{X}_{\#}}={{X}_{1}}\#{{X}_{2}}$ is a Riemann surface of genus ${{g}_{1}}+{{g}_{2}}+s-1$, the \emph{complex connected sum}\index{complex connected sum}, where $g_{i}$ is the genus of the $X_{i}$ and $s$ is the number of points in $D_{1}$ and $D_{2}$. Its complex structure however is heavily dependent on the parameters ${{p}_{i}},{{q}_{i}},\lambda $.

\subsection{Gluing at the level of solutions to Hitchin's equations}\label{gluing_solutions}
For gluing two parabolic $G$-Higgs bundles over a complex connected sum ${{X}_{\#}}$ of Riemann surfaces, we choose to switch to the language of solutions to Hitchin's equations and make use of the analytic techniques of Clifford Taubes for gluing instantons over 4-manifolds \cite{Taubes} in order to control the stability condition. These techniques have been applied to establish similar gluing constructions for solutions to gauge-theoretic equations, as for instance in \cite{DonKron}, \cite{Foscolo}, \cite{He}, \cite{Safari}, and they pertain first to finding good \textit{local model solutions} of the gauge-theoretic equations. Then one has to put, using appropriate gauge transformations, the initial data into these model forms, which are identified locally over annuli around the marked points, thus allowing a construction of a new pair over ${{X}_{\#}}$ that combines the original data from $X_{1}$ and $X_{2}$. This produces, however, an \textit{approximate solution} of the equations, which then has to be corrected to an exact solution via a gauge transformation. The argument providing the existence of such a gauge is translated into a Banach fixed point theorem argument and involves the study of the linearization of a relevant elliptic operator. We briefly describe these steps in the sequel; for complete proofs we refer to \cite{Kydonakis} and \cite{Kydon-article}.

\subsubsection{The local model}\label{local_model} Local $\text{SL(2}\text{,}\mathbb{R}\text{)}$-model solutions to the Hitchin equations can be obtained by studying the behavior of the harmonic map between a surface $X$ with a given complex structure and the surface $X$ with the corresponding Riemannian metric of constant curvature -4, under degeneration of the domain Riemann surface $X$ to a nodal surface; a Riemann surface with nodes arises from an unnoded surface by pinching off one or more simple closed curves (see \cite{Swoboda}, \cite{Wolf2} for a detailed description).

Let $(E, \Phi)$ be an $\text{SL(2}\text{,}\mathbb{R}\text{)}$-Higgs bundle over $X$ with $E=L\oplus {{L}^{-1}}$ for $L$ a holomorphic square root of the canonical line bundle over $X$ endowed with an auxiliary Hermitian metric ${{h}_{0}}$, and $\Phi =\left( \begin{matrix}
   0 & q  \\
   1 & 0  \\
\end{matrix} \right)\in {{H}^{0}}\left( X ,\mathfrak{sl}\left( E \right) \right)$ for $q$ a holomorphic quadratic differential. If $(E, \Phi)$ is stable, then there is an induced Hermitian metric $H_{0}={{h}_{0}}\oplus h_{0}^{-1}$ on $E$ and an associated Chern connection $A$ with respect to $h$, such that $A={{A}_{L}}\oplus A_{L}^{-1}$, where $A_{L}$ denotes the restriction of the connection $A$ to the line bundle $L$. The stability condition implies that there exists a complex gauge transformation $g$ unique up to unitary gauge transformations, such that $\left( {{A}_{1,s}},{{\Phi }_{1,s}} \right):={{g}^{*}}\left( A,\Phi  \right)$ is a solution to the Hitchin equations.  Calculations in \cite{Swoboda} considering the Hermitian metric on $L$ and a complex gauge giving rise to an exact solution $\left({{A}_{1,s}}, {\Phi_{1,s}} \right)$ of the self-duality equations imply that 
\[{{A}_{1,s}}=O\left( {{\left| \zeta  \right|}^{s}} \right)\left( \begin{matrix}
   1 & 0  \\
   0 & -1  \\
\end{matrix} \right)\left( \frac{d\zeta }{\zeta }-\frac{d\bar{\zeta }}{{\bar{\zeta }}} \right),\text{   }{{\Phi }_{1,s}}=\left( 1+O\left( {{\left| \zeta  \right|}^{s}} \right) \right)\left( \begin{matrix}
   0 & \frac{s}{2}  \\
   \frac{s}{2} & 0  \\
\end{matrix} \right)\frac{d\zeta }{i\zeta }\]
for local coordinates $\zeta$. Therefore, after a unitary change of frame, the Higgs field ${{\Phi }_{1,s}}$ is asymptotic to the model Higgs field $\Phi _{s}^{\bmod }=\left( \begin{matrix}
   \frac{s}{2} & 0  \\
   0 & -\frac{s}{2}  \\
\end{matrix} \right)\frac{d\zeta }{i\zeta }$, while the connection ${{A}_{1,s}}$ is asymptotic to the trivial flat connection.

In conclusion, the \emph{model solution}\index{Hitchin equation!model solution} to the $\text{SL(2}\text{,}\mathbb{R}\text{)}$-Hitchin equations we will be considering is described by
	\[{{A}^{\bmod }}=0,\text{   }{{\Phi }^{\bmod }}=\left( \begin{matrix}
   C & 0  \\
   0 & -C  \\
\end{matrix} \right)\frac{dz}{z}\]
over a punctured disk with $z$-coordinates around the puncture with the condition that $C\in \mathbb{R}$ with $C\ne 0$ and that the meromorphic quadratic differential $q:=\det {{\Phi }^{\bmod }}$ has at least one simple zero. That this is indeed the generic case, is discussed in \cite{MSWW}.

\subsubsection{Approximate solutions of the $\text{SL(2}\text{,}\mathbb{R}\text{)}$-Hitchin equations}\label{approximate_solution} 
Let $X$ be a compact Riemann surface and $D:=\left\{ {{p}_{1}},\ldots ,{{p}_{s}} \right\}$ a collection of $s$ distinct points on $X$. Moreover, let $\left( E,h \right)$ be a Hermitian vector bundle on $E$. Choose an initial pair $\left( {{A}^{\bmod }},{{\Phi }^{\bmod }} \right)$ on $E$, such that in some unitary trivialization of $E$ around each point $p\in D$, the pair coincides with the local model from \S \ref{local_model}; of course, on the interior of each region $X\backslash \left\{ {{p}} \right\}$ the pair $\left( {{A}^{\bmod }},{{\Phi }^{\bmod }} \right)$ need not satisfy the Hitchin equations.

One can then define \emph{global Sobolev spaces on $X$}\index{global Sobolev space} as the spaces of admissible deformations of the model unitary connection and the model Higgs field $\left( {{A}^{\bmod }},{{\Phi }^{\bmod }} \right)$ and introduce the moduli space $\mathsf{\mathcal{M}}\left( {{X}^{\times }} \right)$ of solutions to the Hitchin equations modulo unitary gauge transformation, which are close to the model solution over a punctured Riemann surface ${{X}^{\times }}:=X-D$ for some fixed parameter $C\in \mathbb{R}$; this moduli space was explicitly constructed by Konno in \cite{Konno} as a hyperk{\"a}hler quotient.

In fact, as was shown by Biquard and Boalch (Lemma 5.3 in \cite{BiBo}) and later improved by Swoboda (Lemma 3.2 in \cite{Swoboda}), a pair $\left( A,\Phi  \right)\in \mathsf{\mathcal{M}}\left( {{X}^{\times }} \right)$ is asymptotically close to the model $\left( {{A}^{\text{mod}}},{{\Phi }^{\text{mod}}} \right)$ near each puncture in $D$. In particular, there exists a complex gauge transformation $g=\exp \left( \gamma  \right)$ such that ${{g}^{*}}\left( A,\Phi  \right)$ coincides with $\left( {{A}_{p}^{\bmod }},{{\Phi }_{p}^{\bmod }} \right)$ on a sufficiently small neighborhood of the point $p$, for each $p\in D$.

We shall now use this complex gauge transformation as well as a smooth cut-off function to obtain an approximate solution to the $\text{SL(2}\text{,}\mathbb{R}\text{)}$-Hitchin equations. For fixed local coordinates $z$ around each puncture $p$ and given the positive function $r = \left| z \right|$ around the puncture, fix a constant $0<R<1$ and choose a smooth cut-off function ${{\chi }_{R}}:\left[ 0,\infty  \right)\to \left[ 0,1 \right]$ with $\text{supp}\chi \subseteq \left[ 0,R \right]$ and ${{\chi }_{R}}\left( r \right)=1$ for $r\le \frac{3R}{4}$. We impose the further requirement on the growth rate of this cut-off function:
\begin{equation}\label{growth_rate}
  \left| r{{\partial }_{r}}{{\chi }_{R}} \right|+\left| {{\left( r\partial r \right)}^{2}}{{\chi }_{R}} \right|\le k
\end{equation}
for some constant $k$ not depending on $R$.

The map $x\mapsto {{\chi }_{R}}\left( r\left( x \right) \right):{{X}^{\times }}\to \mathbb{R}$ gives rise to a smooth cut-off function on the punctured surface ${{X}^{\times }}$ which by a slight abuse of notation we shall still denote by ${{\chi }_{R}}$. We may use this function ${{\chi }_{R}}$ to glue the two pairs $\left( A,\Phi  \right)$ and $\left( A_{p}^{\bmod },\Phi _{p}^{\bmod } \right)$ into an \emph{approximate solution}\index{Hitchin equation!approximate solution}
	\[\left( A_{R}^\textit{app},\Phi _{R}^\textit{app} \right):=\exp {{\left( {{\chi }_{R}}\gamma  \right)}^{*}}\left( A,\Phi  \right).\]
The pair $\left( A_{R}^\textit{app},\Phi _{R}^\textit{app} \right)$ is a smooth pair and is by construction an exact solution of the Hitchin equations away from each punctured neighborhood ${{\mathsf{\mathcal{U}}}_{p}}$, while it coincides with the model pair $\left( A_{p}^{\bmod },\Phi _{p}^{\bmod } \right)$ near each puncture. More precisely, we have:
\[\left( A_{R}^\textit{app},\Phi _{R}^\textit{app} \right)=\left\{ \begin{matrix}
   \left( A,\Phi  \right),\text{ over }X\backslash \bigcup\limits_{p\in D}{\left\{ z\in {{\mathsf{\mathcal{U}}}_{p}}\left| \frac{3R}{4}\le \left| z \right|\le R \right. \right\}}  \\
   \left( A_{p}^{\bmod },\Phi _{p}^{\bmod } \right),\text{ over }\left\{ z\in {{\mathsf{\mathcal{U}}}_{p}}\left| 0<\left| z \right|\le \frac{3R}{4} \right. \right\}\text{, for each }p\in D.  \\
\end{matrix} \right.\]

\vspace{5mm}
\begin{figure}
 \centering
\includegraphics[width=0.6\linewidth,height=0.6\textheight,keepaspectratio]{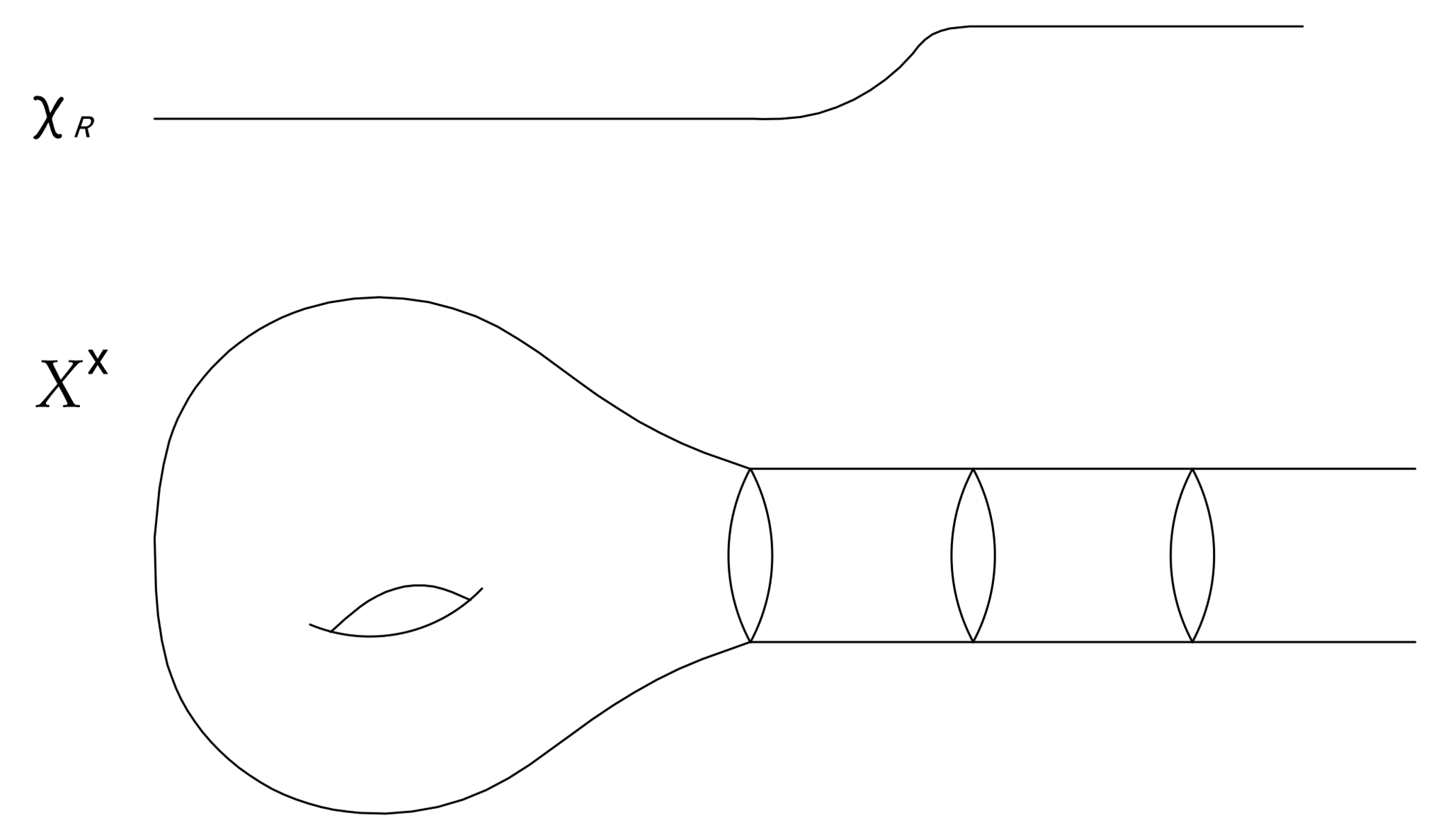}
 \caption{Constructing an approximate solution over the punctured surface ${{X}^{\times }}$.}
\end{figure}
\vspace{5mm}

Since $\left( A_{R}^\textit{app},\Phi _{R}^\textit{app} \right)$ is complex gauge equivalent to an exact solution $\left( A,\Phi  \right)$ of the Hitchin equations, the Higgs field $\Phi _{R}^\textit{app}$ is holomorphic with respect to the holomorphic structure ${{\bar{\partial }}_{A_{R}^\textit{app}}}$, in other words, one has  ${{\bar{\partial }}_{A_{R}^\textit{app}}}\Phi _{R}^\textit{app}=0$.
Moreover, assumption (\ref{growth_rate}) on the growth rate of the bump function ${{\chi }_{R}}$ provides us with a good estimate of the error up to which $\left( A_{R}^\textit{app},\Phi _{R}^\textit{app} \right)$ satisfies the first among the Hitchin equations, 
$F\left( A \right)+\left[ \Phi ,{{\Phi }^{*}} \right]=0$.

\subsection{Approximate solutions to the $G$-Hitchin equations}
We now wish to obtain an approximate $G$-Higgs pair by extending the $\text{SL(2}\text{,}\mathbb{C}\text{)}$-data via an embedding
\[\phi :\text{SL(2}\text{,}\mathbb{R}\text{)}\hookrightarrow G,\]
for a reductive Lie group $G$. It is important that copies of a maximal compact subgroup of $\text{SL(2}\text{,}\mathbb{R}\text{)}$ are mapped via $\phi$ into copies of a maximal compact subgroup of $G$ and that the norm of the infinitesimal deformation ${\phi }_{*}$ on the complexified Lie algebra ${{\mathfrak{g}}^{\mathbb{C}}}$ satisfies a Lipschitz condition. Assuming that this is indeed the case for an embedding $\phi $ (examples can be found in \cite{Kydon-article} and will be demonstrated in \S \ref{Examples-orthogonal}), one gets by extension via the embedding $\phi $ a ${{G}^{\mathbb{C}}}$-pair satisfying the $G$-Hitchin equations up to an error, which we have good control of. 

For $i=1,2$, let $X_{i}$ be a closed Riemann surface of genus $g_{i}$ and let ${{D}_{1}}=\left\{ {{p}_{1}},\ldots ,{{p}_{s}} \right\}$, ${{D}_{2}}=\left\{ {{q}_{1}},\ldots ,{{q}_{s}} \right\}$ a divisor of $s$-many distinct points on  ${{X}_{1}}$, ${{X}_{2}}$ respectively. Choose local coordinates $z$ near the points in ${{D}_{1}}$ and local coordinates $w$ near the points in ${{D}_{2}}$. Assume that we get via an embedding as was described above approximate solutions $\left( {{A}_{1}},{{\Phi }_{1}} \right)$, $\left( {{A}_{2}},{{\Phi }_{2}} \right)$, which agree over neighborhoods around the points in the divisors ${{D}_{1}}$ and ${{D}_{2}}$, with ${{A}_{1}}={{A}_{2}}=0$ and with ${{\Phi }_{1}}\left( z \right)= -{{\Phi }_{2}}\left( w \right)$. Then, there is a suitable frame for the connections over which the Hermitian metrics are both described by the identity matrix and so they are constant in particular. Set $\left( A_{p,q}^{\bmod },\Phi _{p,q}^{\bmod } \right):=\left( A_{1,p}^{\bmod },\Phi _{1,p}^{\bmod } \right)=-\left( A_{2,q}^{\bmod },\Phi _{2,q}^{\bmod } \right)$. We can glue the pairs $\left( {{A}_{1}},{{\Phi }_{1}} \right),\left( {{A}_{2}},{{\Phi }_{2}} \right)$ together to get an \emph{approximate solution}\index{Hitchin equation!approximate solution} of the $G$-Hitchin equations over the complex connected sum ${{X}_{\#}}:=X_{{1}}\#X_{{2}}$:

\[\left( A_{R}^\textit{app},\Phi _{R}^\textit{app} \right):=\left\{ \begin{matrix}
   \left( {{A}_{1}},{{\Phi }_{1}} \right), & {} & \text{over }{{X}_{1}}\backslash {{X}_{2}};  \\
   \left( A_{p,q}^{\bmod },\Phi _{p,q}^{\bmod } \right), & {} & \text{         over } \Omega \text{ around each pair of points } \left( p,q \right);  \\
   \left( {{A}_{2}},{{\Phi }_{2}} \right), & {} & \text{over }{{X}_{2}}\backslash {{X}_{1}}.  \\
\end{matrix} \right.\]

\vspace{5mm}
\begin{figure}
\centering
  \includegraphics[width=0.6\linewidth,height=0.6\textheight,keepaspectratio]{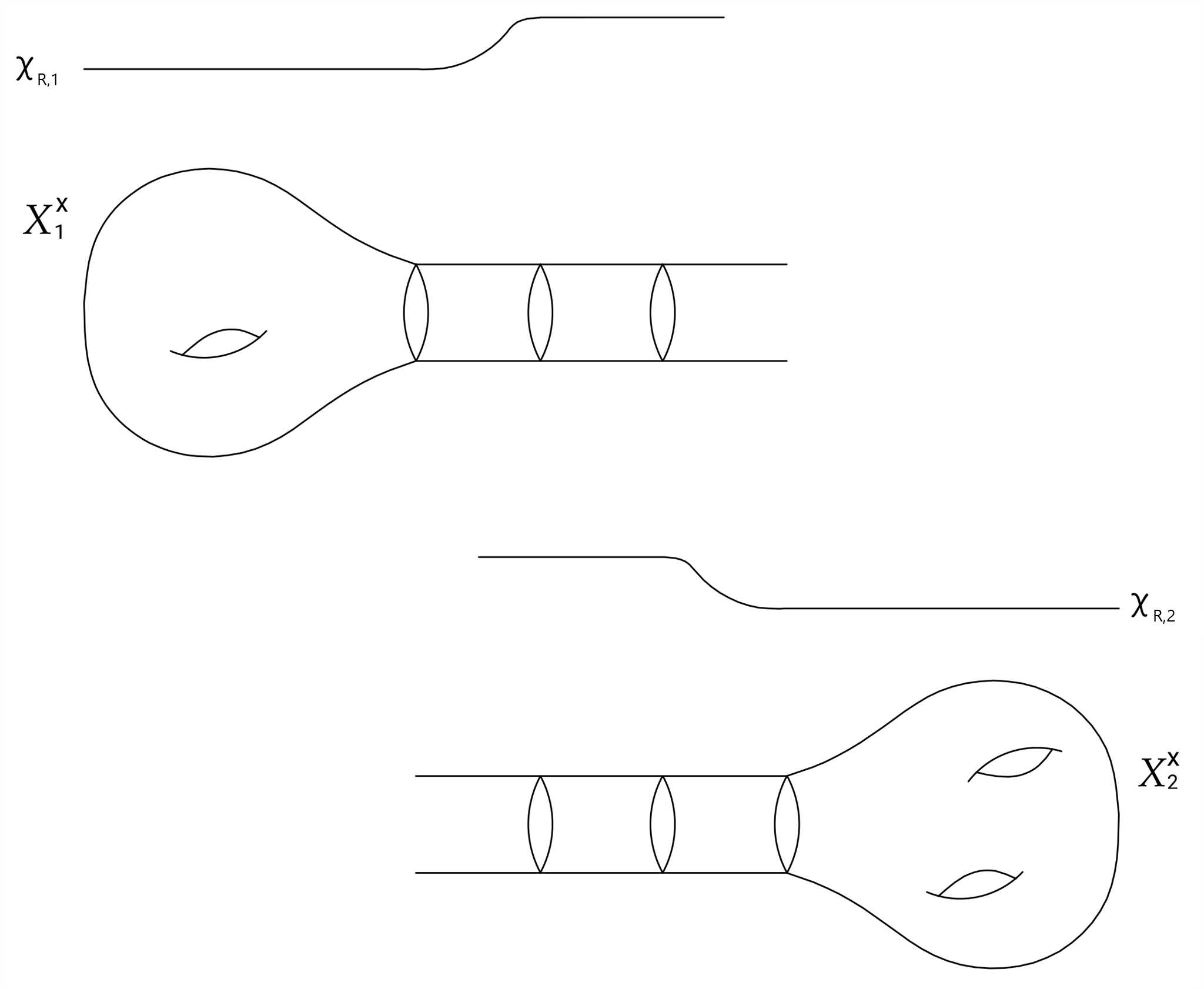}
  \caption{Constructing approximate solutions over $X_{1}^{\times }$ and $X_{2}^{\times }$.}
\end{figure}
\vspace{5mm}

\begin{figure}
\centering
  \includegraphics[width=0.6\linewidth,height=0.6\textheight,keepaspectratio]{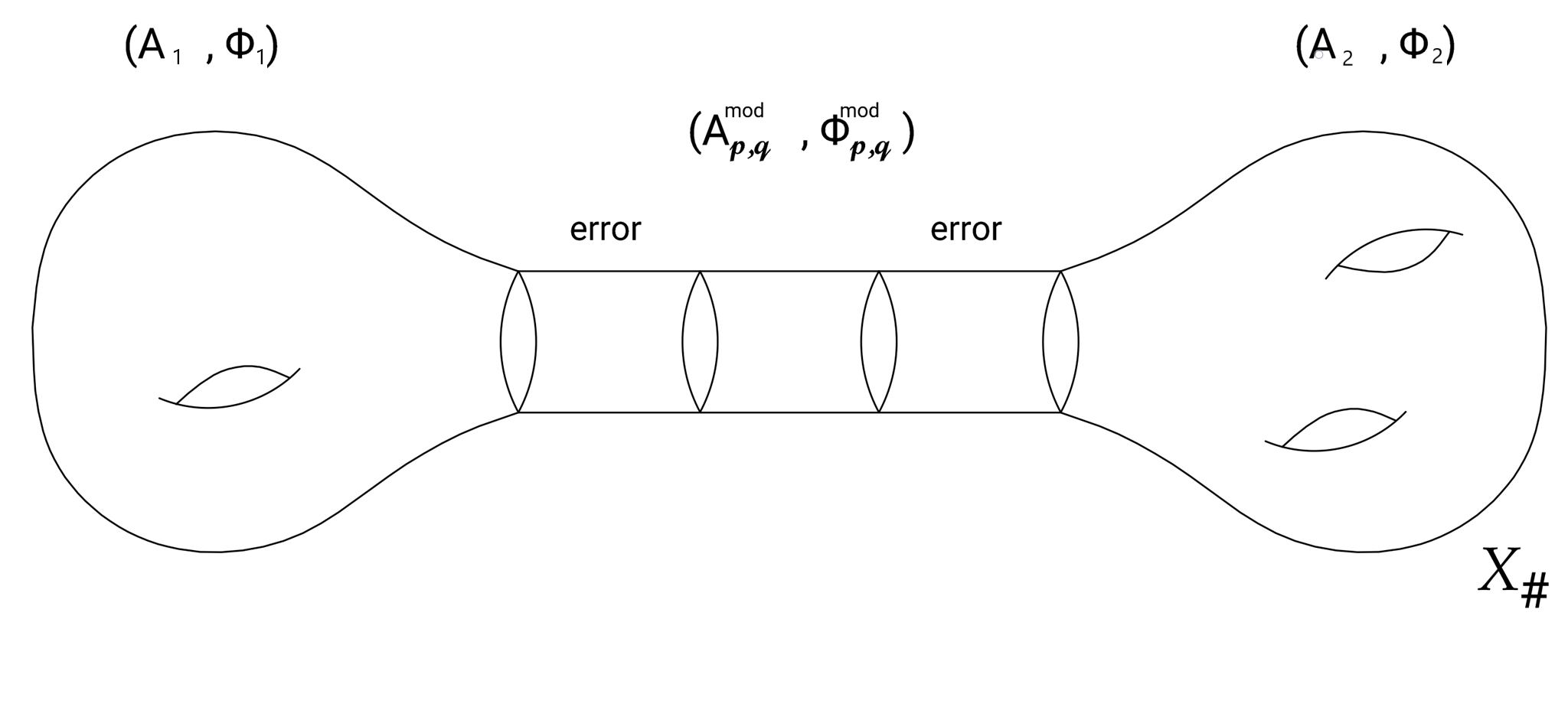}
   \caption{The approximate solution $\left( A_{R}^\textit{app},\Phi _{R}^\textit{app} \right)$ over the complex connected sum ${{X}_{\#}}$.}
\end{figure}
\vspace{5mm}

By construction, $\left( A_{R}^\textit{app},\Phi _{R}^\textit{app} \right)$ is a smooth pair on ${{X}_{\#}}$, complex gauge equivalent to an exact solution of the Hitchin equations by a smooth gauge transformation defined over all of ${{X}_{\#}}$. It satisfies the second Hitchin equation (holomorphicity), while the first equation is satisfied up to an error which we have good control of.

\subsection{The contraction mapping argument}\label{contraction_mapping}
A standard strategy, due largely to Taubes \cite{Taubes}, for correcting an approximate solution to an exact solution of gauge-theoretic equations involves studying the linearization of a relevant elliptic operator. In the Higgs bundle setting, the linearization of the Hitchin operator was first described in \cite{MSWW} and furthermore in \cite{Swoboda} for solutions to the $\text{SL(2}\text{,}\mathbb{C}\text{)}$-self duality equations over a nodal surface. We are going to use this analytic machinery to correct our approximate solution to an exact solution over the complex connected sum of Riemann surfaces. We summarize this strategy below.

Let $G$ be a connected, semisimple Lie group. For the complex connected sum ${{X}_{\#}}$, consider the nonlinear $G$-Hitchin operator\index{Hitchin operator} at a pair $\left( A,\Phi  \right)\in {{\Omega }^{1}}\left( {{X}_{\#}},{{E}_{H}}\left( {{\mathfrak{h}}^{\mathbb{C}}} \right) \right)\oplus {{\Omega }^{1,0}}\left( {{X}_{\#}},{{E}_{H}}\left( {{\mathfrak{m}}^{\mathbb{C}}} \right) \right)$:

\begin{equation}\label{Hitchin_operator_def}
\mathsf{\mathcal{H}}\left( A,\Phi  \right)=\left( F\left( A \right)-\left[ \Phi ,\tau \left( \Phi  \right) \right],{{{\bar{\partial }}}_{A}}\Phi  \right).
\end{equation}
Moreover, consider the orbit map
\[\gamma \mapsto {{\mathsf{\mathcal{O}}}_{\left( A,\Phi  \right)}}\left( \gamma  \right)={{g}^{*}}\left( A,\Phi  \right)=\left( {{g}^{*}}A,{{g}^{-1}}\Phi g \right),\]
for $g=\exp \left( \gamma  \right)$ and $\gamma \in {{\Omega }^{0}}\left( {{X}_{\#}},{{E}_{H}}\left( \mathfrak{h}^{\mathbb{C}} \right) \right)$, where $H\subset G$ is a maximal compact subgroup.

Therefore, correcting the approximate solution $\left( A_{R}^\textit{app},\Phi _{R}^\textit{app} \right)$ to an exact solution of the $G$-Hitchin equations accounts to finding a point $\gamma$ in the complex gauge orbit of $\left( A_{R}^\textit{app},\Phi _{R}^\textit{app} \right)$, for which $\mathsf{\mathcal{H}}\left( {{g}^{*}}\left( A_{R}^\textit{app},\Phi _{R}^\textit{app} \right) \right)=0$. However, since we have seen that the second equation is satisfied by the pair $\left( A_{R}^\textit{app},\Phi _{R}^\textit{app} \right)$ and since the condition ${{\bar{\partial }}_{A}}\Phi =0$ is preserved under the action of the complex gauge group ${{\mathsf{\mathcal{G}}}_{H}^{\mathbb{C}}}$, we actually seek a solution $\gamma $ to the following equation \[{{\mathsf{\mathcal{F}}}_{R}}\left( \gamma  \right):=p{{r}_{1}}\circ \mathsf{\mathcal{H}}\circ {{\mathsf{\mathcal{O}}}_{\left( A_{R}^\textit{app},\Phi _{R}^\textit{app} \right)}}\left( \exp (\gamma ) \right)=0.\]
For a Taylor series expansion of this operator
	\[{{\mathsf{\mathcal{F}}}_{R}}\left( \gamma  \right)=p{{r}_{1}}\mathsf{\mathcal{H}}\left( A_{R}^\textit{app},\Phi _{R}^\textit{app} \right)+{{L}_{\left( A_{R}^\textit{app},\Phi _{R}^\textit{app} \right)}}\left( \gamma  \right)+{{Q}_{R}}\left( \gamma  \right),\]
where ${{Q}_{R}}$ includes the quadratic and higher order terms in $\gamma $, we can then see that ${{\mathsf{\mathcal{F}}}_{R}}\left( \gamma  \right)=0$ if and only if $\gamma $ is a fixed point of the map
\begin{align*}
   T: H_{B}^{2}\left( {{X}_{\#}} \right) & \to H_{B}^{2}\left( {{X}_{\#}} \right) \\
   \gamma & \mapsto -{{G}_{R}}\left( \mathsf{\mathcal{H}}\left( A_{R}^\textit{app},\Phi _{R}^\textit{app} \right)+{{Q}_{R}}(\gamma ) \right),
\end{align*}
where we denoted ${{G}_{R}}:= L_{\left( A_{R}^\textit{app},\Phi _{R}^\textit{app} \right)}^{-1}$ and $H_{B}^{2}\left( {{X}_{\#}} \right)$ is the Hilbert space defined by
\[H_{B}^{2}\left( {{X}_{\#}} \right):=\left\{ \gamma \in {{L}^{2}}\left( {{X}_{\#}} \right)\left| {{\nabla }_{B}}\gamma ,\nabla _{B}^{2}\gamma  \right.\in {{L}^{2}}\left( {{X}_{\#}} \right) \right\},\]
for a fixed background connection $\nabla_{B}$ defined as a smooth extension to ${X}_{\#}$ of the model connection $A^{\text{mod}}_{p,q}$ over the cylinder for each pair of points $(p,q)$. 

The problem then reduces to showing that the mapping $T$ is a contraction\index{contraction} of the open ball ${{B}_{{{\rho }_{R}}}}$ of radius ${{\rho }_{R}}$ in $H_{B}^{2}\left( {{X}_{\#}} \right)$, since then from Banach's fixed point theorem there will exist a unique $\gamma $ such that $T\left( \gamma  \right)=\gamma $, in other words, such that ${{\mathsf{\mathcal{F}}}_{R}}\left( \gamma  \right)=0$. In particular, one needs to show that:
\begin{enumerate}
  \item $T$ is a contraction defined on ${{B}_{{{\rho }_{R}}}}$ for some ${{\rho }_{R}}$, and
  \item $T$ maps ${{B}_{{{\rho }_{R}}}}$ to ${{B}_{{{\rho }_{R}}}}$.
\end{enumerate}

In order to complete the above described contraction mapping argument, we need to show the following:

\begin{description}
  \item[i] The linearized operator\index{Hitchin equation!linearized operator} at the approximate solution ${{L}_{\left( A_{R}^\textit{app},\Phi _{R}^\textit{app} \right)}}$ is invertible.
  \item[ii] There is an upper bound for the inverse operator ${{G}_{R}}=L_{\left( A_{R}^\textit{app},\Phi _{R}^\textit{app} \right)}^{-1}$ as an operator ${{L}^{2}}\left( {{r}}drd\theta \right)\to {{L}^{2}}\left( {{r}}drd\theta \right)$.
  \item[iii] There is an upper bound for the inverse operator ${{G}_{R}}=L_{\left( A_{R}^\textit{app},\Phi _{R}^\textit{app} \right)}^{-1}$ also when viewed as an operator ${{L}^{2}}\left( {{r}}drd\theta \right)\to H_{B}^{2}\left( {{X}_{\#}},{{r}}drd\theta \right)$.
  \item[iv] We can control a Lipschitz constant for ${{Q}_{R}}$, that means there exists a constant $C>0$ such that
	\[{{\left\| {{Q}_{R}}\left( {{\gamma }_{1}} \right)-{{Q}_{R}}\left( {{\gamma }_{0}} \right) \right\|}_{{{L}^{2}}}}\le C\rho {{\left\| {{\gamma }_{1}}-{{\gamma }_{0}} \right\|}_{H_{B}^{2}}}\]
for all $0<\rho \le 1$ and ${{\gamma }_{0}},{{\gamma }_{1}}\in {{B}_{\rho }}$, the closed ball of radius $\rho$ around $0$ in $H_{B}^{2}\left( {{X}_{\#}} \right)$.
\end{description}

\subsection{Correcting an approximate solution to an exact solution}\label{section_linearization}

We shall focus on the linear term in the Taylor series expansion. The linearization operator  ${{L}_{\left( A,\Phi  \right)}}$ at a pair $\left( A,\Phi  \right)\in {{\Omega }^{1}}\left( {{X}_{\#}},{{E}_{H}}\left( {{\mathfrak{h}}^{\mathbb{C}}} \right) \right)\oplus {{\Omega }^{1,0}}\left( {{X}_{\#}},{{E}_{H}}\left( {{\mathfrak{m}}^{\mathbb{C}}} \right) \right)$ is defined by
 \[{{L}_{\left( A,\Phi  \right)}}:=-i*D\mathsf{\mathcal{F}}\left( \gamma  \right):{{\Omega }^{0}}\left( {{X}_{\#}},{{E}_{H}}\left( {{\mathfrak{h}}^{\mathbb{C}}} \right) \right)\to {{\Omega }^{0}}\left( {{X}_{\#}},{{E}_{H}}\left( {{\mathfrak{h}}^{\mathbb{C}}} \right) \right),\]
where $D\mathsf{\mathcal{F}}\left( \gamma  \right)$ denotes the differential
\[D\mathsf{\mathcal{F}}\left( \gamma  \right) =  {{\partial }_{A}}{{{\bar{\partial }}}_{A}} \gamma - {{{\bar{\partial }}}_{A}}{{\partial }_{A}} \gamma^{*} +\left[ \Phi ,-\tau \left( \left[ \Phi ,\gamma  \right] \right) \right]+\left[ \left[ \Phi ,\gamma  \right],-\tau \left( \Phi  \right) \right],
\]
for $H \subset G$ a maximal compact subgroup and $\tau$ the compact real form of $\mathfrak{g}^{\mathbb{C}}$. It satisfies the following lemma; for a proof, see Lemma 5.1 in \cite{Kydon-article}:

\begin{lemma}\label{linearization_lemma}
For $\gamma \in {{\Omega }^{0}}\left( {{X}_{\#}},{{E}_{H}}\left( \mathfrak{h} \right) \right)$, the linearization operator satisfies
\[{{\left\langle {{L}_{\left( A,\Phi  \right)}}\gamma ,\gamma  \right\rangle }_{{{L}^{2}}}}=\left\| {{d}_{A}}\gamma  \right\|_{{{L}^{2}}}^{2}+2\left\| \left[ \Phi ,\gamma  \right] \right\|_{{{L}^{2}}}^{2}\ge 0.\]
In particular, ${{L}_{\left( A,\Phi  \right)}}\gamma =0$ if and only if ${{d}_{A}}\gamma =\left[ \Phi ,\gamma  \right]=0$.
\end{lemma}

In order to prove the existence of the inverse operator ${{G}_{R}}:= L_{\left( A_{R}^\textit{app},\Phi _{R}^\textit{app} \right)}^{-1}$ and obtain an upper bound for its $L^2$-norm, we apply a version of the Cappell--Lee--Miller gluing theorem for a pair of cylindrical $\mathbb{Z}_{2}$-graded Dirac-type operators (see \cite{CLM} and \cite[\S 5.B]{Nicolaescuarticle}). 

For our approximate solution $\left( A_{R}^\textit{app},\Phi _{R}^\textit{app} \right)$ constructed over ${{X}_{\#}}$ with $0<R<1$ and $T=-\log R$, consider the elliptic complex
\begin{align*}
0\xrightarrow{{}}{{\Omega }^{0}}\left( {{X}_{\#}},{{E}_{H}}\left( {{\mathfrak{h}}^{\mathbb{C}}} \right) \right) & \xrightarrow{{{L}_{1,T}}}{{\Omega }^{1}}\left( {{X}_{\#}},{{E}_{H}}\left( {{\mathfrak{h}}^{\mathbb{C}}} \right) \right)\oplus {{\Omega }^{1,0}}\left( {{X}_{\#}},{{E}_{H}}\left( {{\mathfrak{g}}^{\mathbb{C}}} \right) \right)\\
& \xrightarrow{{{L}_{2,T}}}{{\Omega }^{2}}\left( {{X}_{\#}},{{E}_{H}}\left( {{\mathfrak{h}}^{\mathbb{C}}} \right) \right)\oplus {{\Omega }^{2}}\left( {{X}_{\#}},{{E}_{H}}\left( {{\mathfrak{g}}^{\mathbb{C}}} \right) \right)\xrightarrow{{}}0,
\end{align*}
where \[{{L}_{1,T}}\gamma =\left( {{d}_{A_{R}^\textit{app}}}\gamma ,\left[ \Phi _{R}^\textit{app},\gamma  \right] \right)\]
is the linearization of the complex gauge group action and \[{{L}_{2,T}}\left( \alpha ,\varphi  \right)=D\mathsf{\mathcal{H}}\left( \alpha ,\varphi  \right)=\left( \begin{matrix}
   {{d}_{A_{R}^\textit{app}}}\alpha +\left[ \Phi _{R}^\textit{app},-\tau \left( \varphi  \right)\right]+\left[ \varphi ,-\tau \left( \Phi _{R}^\textit{app} \right) \right]   \\
   {{{\bar{\partial }}}_{A_{R}^\textit{app}}}\varphi +\left[ \alpha ,\Phi _{R}^\textit{app} \right]  \\
\end{matrix} \right)\]
is the differential of the Hitchin operator from (\ref{Hitchin_operator_def}). Note that, in general, it does not hold that $${{L}_{2,T}}{{L}_{1,T}}=\left[ {{F}_{A_{R}^\textit{app}}},\gamma  \right]+\left[ \left[ \Phi _{R}^\textit{app},-\tau \left( \Phi _{R}^\textit{app} \right) \right],\gamma  \right]=0,$$ since $\left( A_{R}^\textit{app},\Phi _{R}^\textit{app} \right)$ need not be an exact solution. Decomposing ${{\Omega }^{*}}\left( {{X}_{\#}},{{E}_{H}}\left( {{\mathfrak{g}}^{\mathbb{C}}} \right) \right)$ into forms of even, respectively odd total degree, we may introduce the  ${{\mathbb{Z}}_{2}}$-graded Dirac-type operator
	\[{{\mathfrak{D}}_{T}}:=\left( \begin{matrix}
   0 & L_{1,T}^{*}+{{L}_{2,T}}  \\
   {{L}_{1,T}}+L_{2,T}^{*} & 0  \\
\end{matrix} \right)\]
on the closed surface ${{X}_{\#}}$.

For applying the Cappell--Lee--Miller theorem, one has to study the kernel $\ker \left( {{L}_{1}}+L_{2}^{*} \right)$ on the extended $L^2$-space $L_{\text{ext}}^{2}\left( X_{\#}^{\times } \right)$ for the nodal surface $X_{\#}^{\times }$ obtained by extending the cylindrical neck of $X_{\#}$ infinitely (see Definition 6.2 and \S 6.2 in \cite{Kydon-article} for the precise definitions).  

As $R\searrow 0$, the curve ${{X}_{\#}}$ degenerates to a nodal surface $X_{\#}^{\times }$ (equivalently, the cylindrical neck of ${{X}_{\#}}$ extends infinitely). For the cut-off functions ${{\chi }_{R}}$ that we considered in obtaining the approximate pair $\left( A_{R}^\textit{app},\Phi _{R}^\textit{app} \right)$, their support will tend to be empty as $R\searrow 0$, therefore the ``error regions'' disappear along with the neck $\Omega$, thus $\left( A_{R}^\textit{app},\Phi _{R}^\textit{app} \right)\to \left( {{A}_{0}},{{\Phi }_{0}} \right)$ uniformly on compact subsets with
	\[\left( A_{0}^\textit{app},\Phi _{0}^\textit{app} \right)=\left\{ \begin{matrix}
   \left( {{A}_{1}},{{\Phi }_{1}} \right),\,\,\,{{X}_{1}}\backslash \Omega  \\
   \left( {{A}_{2}},{{\Phi }_{2}} \right),\,\,\,{{X}_{2}}\backslash \Omega  \\
\end{matrix} \right.\]
an exact solution with the holonomy of the associated flat connection in $G$.

For trivial kernel $\ker \left( {{L}_{1}}+L_{2}^{*} \right)$, and computing the upper bound for the inverse operator and a Lipschitz constant for the quadratic or higher order terms in the Taylor series expansion, one can correct the approximate solution constructed into an exact solution of the $G$-Hitchin equations. The contraction mapping argument described above then provides the following:

\begin{theorem}\label{correction}
There exists a constant $0<{{R}_{0}}<1$, and for every $0<R<{{R}_{0}}$ there exist a constant ${{\sigma }_{R}}>0$ and a unique section $\gamma \in H_{B}^{2}\left( {{X}_{\#}},E_{H} \left( \mathfrak{h}^{\mathbb{C}} \right) \right)$ satisfying ${{\left\| \gamma  \right\|}_{H_{B}^{2}\left( {{X}_{\#}} \right)}}\le {{\sigma }_{R}}$, so that, for $g=\exp \left( \gamma  \right)$,
	\[\left( {{A}_{\#}},{{\Phi }_{\#}} \right)={{g}^{*}}\left( A_{R}^\textit{app},\Phi _{R}^\textit{app} \right)\]
is an exact solution of the $G$-Hitchin equations over the closed surface ${{X}_{\#}}$.
\end{theorem}

Theorem \ref{correction} now implies that for $\bar{\partial }:=A_{\#}^{0,1}$, the Higgs bundle $\left( {{E}_{\#}}:=\left( {{\mathbb{E}}_{\#}},\bar{\partial } \right),{{\Phi }_{\#}} \right)$ is a polystable $G$-Higgs bundle over the complex connected sum ${{X}_{\#}}$. Collecting the steps from the previous subsections one has the following:

\begin{theorem}\label{main_theorem}
Let $X_{1}$ be a closed Riemann surface of genus $g_{1}$ and ${{D}_{1}}=\left\{ {{p}_{1}},\ldots ,{{p}_{s}} \right\}$ be a collection of $s$ distinct points on $X_{1}$. Let also $G$ be a subgroup of $\rm{GL}\left(n, \mathbb{C} \right)$. Consider respectively a closed Riemann surface $X_{2}$ of genus $g_{2}$ and a collection of also $s$ distinct points ${{D}_{2}}=\left\{ {{q}_{1}},\ldots ,{{q}_{s}} \right\}$ on $X_{2}$. Let $\left( {{E}_{1}},{{\Phi }_{1}} \right)\to {{X}_{1}}$ and $\left( {{E}_{2}},{{\Phi }_{2}} \right)\to {{X}_{2}}$ be parabolic polystable $G$-Higgs bundles with corresponding solutions to the Hitchin equations $\left( {{A}_{1}},{{\Phi }_{1}} \right)$ and $\left( {{A}_{2}},{{\Phi }_{2}} \right)$. Assume that these solutions agree with model solutions $\left( A_{1,{{p}_{i}}}^{\bmod },\Phi _{1,{{p}_{i}}}^{\bmod } \right)$ and $\left( A_{2,{{q}_{j}}}^{\bmod },\Phi _{2,{{q}_{j}}}^{\bmod } \right)$ near the points ${{p}_{i}}\in {{D}_{1}}$ and ${{q}_{j}}\in {{D}_{2}}$, and that the model solutions satisfy $\left( A_{1,{{p}_{i}}}^{\bmod },\Phi _{1,{{p}_{i}}}^{\bmod } \right)=-\left( A_{2,{{q}_{j}}}^{\bmod },\Phi _{2,{{q}_{j}}}^{\bmod } \right)$, for $s$ pairs of points $\left( {{p}_{i}},{{q}_{j}} \right)$. Then there is a polystable $G$-Higgs bundle $\left( E_{\#},\Phi_{\#}  \right)\to {{X}_{\#}}$, constructed over the complex connected sum of Riemann surfaces ${{X}_{\#}}={{X}_{1}}\#{{X}_{2}}$, which agrees with the initial data over ${{X}_{\#}}\backslash {{X}_{1}}$ and ${{X}_{\#}}\backslash {{X}_{2}}$.
\end{theorem}

\begin{definition}
We call a $G$-Higgs bundle constructed by the procedure developed above a \textit{hybrid $G$-Higgs bundle}\index{hybrid Higgs bundle}\index{Higgs bundle!hybrid}.
\end{definition}

\subsection{Topological invariants}
The connected component of the moduli space $\mathcal{M}\left(G \right)$ that a hybrid Higgs bundle lies, can be determined by Higgs bundle topological invariants, and one needs to understand how these invariants behave under the complex connected sum operation. The next two propositions show that there is an additivity property for topological invariants over the connected sum operation, both from the Higgs bundle and the surface group representation point of view.

When the group $G$ is a subgroup of $\text{GL}\left( n,\mathbb{C} \right)$, the data of a parabolic $G$-Higgs bundle (defined in full generality in \cite{BGM}) reduce to the data of a parabolic Higgs bundle as seen in \S \ref{parabolic_data}. Moreover, the basic topological invariant of a parabolic (resp. non-parabolic) pair is the parabolic degree (resp. degree) of some underlying parabolic (resp. non-parabolic) bundle in the Higgs bundle data. We refer to \cite{KSZ} for a detailed description of this data and the corresponding topological invariants for a number of cases of parabolic $G$-Higgs bundles.  

The following proposition now describes an additivity property for the degrees:  
\begin{proposition}[Proposition 8.1 in \cite{Kydon-article}]\label{additivity_property}
Let ${{X}_{\#}}=X_{1}\#X_{2}$ be the complex connected sum of two closed Riemann surfaces $X_{1}$ and $X_{2}$ with divisors $D_{1}$ and $D_{2}$ of $s$ distinct points on each surface, and let ${{V}_{1}},{{V}_{2}}$ be parabolic vector bundles over $X_{1}$ and $X_{2}$ respectively. Then, if the parabolic bundles $V_{1}$, $V_{2}$ glue to a bundle $V_{1}\#V_{2}$ over ${{X}_{\#}}$, the following identity holds
\[\deg \left( {{V}_{1}}\#{{V}_{2}} \right) = par{\deg }\left( {{V}_{1}} \right)+par{\deg }\left( {{V}_{2}} \right).\]
\end{proposition}

Considering the connected sum of the underlying topological surfaces $\Sigma ={{\Sigma }_{1}}{{\cup }_{\gamma}}{{\Sigma }_{2}}$ along a loop $\gamma$, a notion of Toledo invariant is defined for representations over these subsurfaces with boundary; see \cite{BIW} for a detailed definition in this context. Moreover, the authors in \cite{BIW} have established an additivity property for the Toledo invariant over a connected sum of surfaces. In particular:

\begin{proposition}
[Proposition 3.2 in \cite{BIW}] If $\Sigma ={{\Sigma }_{1}}{{\cup }_{\gamma}}{{\Sigma }_{2}}$ is the connected sum of two subsurfaces ${{\Sigma }_{i}}$ along a simple closed separating loop $\gamma$, then
\[{{\text{T}}_{\rho }}={{\text{T}}_{\rho_{1} }} \text{+} {{\text{T}}_{\rho_{2}}},\]
where ${{\rho }_{i}}=\rho \left| _{{{\pi }_{1}}\left( {{\Sigma }_{i}} \right)} \right.$, for  $i=1,2$.
\end{proposition}
The two propositions above allow one to determine the topological invariants of the hybrid Higgs bundles, respectively fundamental group representations, from the topological invariants of the underlying objects that were deformed and glued together. Note, in particular, that this property implies that the amalgamated product of two maximal representations is again a maximal representation defined over the compact surface $\Sigma$.

\section{Examples: Model Higgs bundles in exceptional components of orthogonal groups}\label{Examples-orthogonal}

We now exhibit specific examples where the previous gluing construction can provide model objects lying inside higher Teichm\"{u}ller spaces of particular geometric importance. 

When the Lie group is $G=\text{Sp}\left(4, \mathbb{R} \right)$, hybrid Higgs bundles in the exceptional connected components of the maximal  $G=\text{Sp}\left(4, \mathbb{R} \right)$-Higgs bundles identified by Gothen in \cite{Gothen} were obtained in \cite{Kydon-article}. We next provide such examples in the case of the group $G=\text{SO}\left(p, p+1 \right)$, which involves an extra parameter compared to the $\text{Sp}\left(4, \mathbb{R} \right)$-case. Note, however, that a maximality property is not apparent in this case apart from when $p=2$, since the group $G=\text{SO}\left(p, p+1 \right)$ for $p \ne 2$ is not Hermitian of noncompact type; cf. the discussion on maximality in \S \ref{subsection_hTs}. 

\subsection{$\text{SO}\left( p,q \right)$-Higgs bundle data}

The connected components of the $\text{SO}\left( p,q \right)$-character variety $\mathsf{\mathcal{R}}\left( \text{SO}\left( p,q \right) \right)$ can be more explicitly described using the theory of Higgs bundles. Let $X$ be a compact Riemann surface with underlying topological surface ${{\Sigma }}$. Under the non-abelian Hodge correspondence, fundamental group representations into the group  $\text{SO}\left( p,q \right)$ correspond to holomorphic tuples  $\left( V,{{Q}_{V}},W,{{Q}_{W}},\eta  \right)$ over $X$, where:
\begin{itemize}
\item  $\left( V,{{Q}_{V}} \right)$ and $\left( W,{{Q}_{W}} \right)$ are holomorphic orthogonal bundles of rank $p$ and $q$ respectively with the additional condition that ${{\wedge }^{p}}\left( V \right)\cong {{\wedge }^{q}}\left( W \right)$.
\item $\eta :W\to V\otimes K$ is a holomorphic section of $\text{Hom}\left( W,V \right)\otimes K$.
\end{itemize}

Using Higgs bundle methods, in particular a real valued proper function defined by the $L^{2}$-norm of the Higgs field and a natural holomorphic ${{\mathbb{C}}^{*}}$-action, the authors in \cite{Aparicio et al} classify \emph{all} polystable local minima of the Hitchin function in $\mathsf{\mathcal{M}}\left( \text{SO}\left( p,q \right) \right)$, for $2<p\le q$. For these moduli spaces, not all local minima occur at fixed points of the ${{\mathbb{C}}^{*}}$-action and additional connected components of $\mathsf{\mathcal{M}}\left( \text{SO}\left( p,q \right) \right)$ emerge by constructing a map 
\[\Psi :{{\mathsf{\mathcal{M}}}_{{{K}^{p}}}}\left( \text{SO}\left( 1,q-p+1 \right) \right)\times \underset{j=1}{\overset{p-1}{\mathop \bigoplus }}\,{{H}^{0}}\left( X,{{K}^{2j}} \right)\to \mathsf{\mathcal{M}}\left( \text{SO}\left( p,q \right) \right),\]
which is an isomorphism onto its image, open and closed. In the description above,  ${{\mathsf{\mathcal{M}}}_{{{K}^{p}}}}\left( \text{SO}\left( 1,q-p+1 \right) \right)$ denotes the moduli space of ${{K}^{p}}$-twisted  $\text{SO}\left( 1,q-p+1 \right)$-Higgs bundles on the Riemann surface $X$, where $K$ is the canonical line bundle over $X$, and $\underset{j=1}{\overset{p-1}{\mathop \bigoplus }}\,{{H}^{0}}\left( X,{{K}^{2j}} \right)$ denotes the vector space of holomorphic differentials of degree $2j$. Note that a ${{K}^{p}}$-twisted $\text{SO}\left( 1,n \right)$-Higgs bundle is defined by a triple $\left( I,\hat{W},\hat{\eta } \right)$, where  $\left( \hat{W},{{Q}_{{\hat{W}}}} \right)$ is a rank $n$ orthogonal bundle, $I={{\wedge }^{n}}\hat{W}$  and $\hat{\eta }\in {{H}^{0}}\left( \text{Hom}\left( \hat{W},I \right)\otimes {{K}^{p}} \right)$. A point in the image of the map $\Psi $ is then described by
\begin{equation}\label{201}
\Psi \left( \left( I,\hat{W},\hat{\eta } \right),{{q}_{2}},\ldots ,{{q}_{2p-2}} \right)=\left( V,W,\eta  \right),
\end{equation}
where\\
$$V:=I\otimes \left( {{K}^{p-1}}\oplus {{K}^{p-3}}\oplus \ldots \oplus {{K}^{3-p}}\oplus {{K}^{1-p}} \right);$$
$$W:=\hat{W}\oplus I\otimes \left( {{K}^{p-2}}\oplus {{K}^{p-4}}\oplus \ldots \oplus {{K}^{4-p}}\oplus {{K}^{2-p}} \right);$$
\begin{equation}\label{202}
\eta :=\left( \begin{matrix}
   {\hat{\eta }} & {{q}_{2}} & {{q}_{4}} & \ldots  & {{q}_{2p-2}}  \\
   0 & 1 & {{q}_{2}} & \cdots  & {{q}_{2p-4}}  \\
   \vdots  & {} & \ddots  & {} & \vdots   \\
   \vdots  & {} & {} & 1 & {{q}_{2}}  \\
   0 & 0 & \cdots  & 0 & 1  \\
\end{matrix} \right).
\end{equation}
Moreover, an $\text{SO}\left( p,q \right)$-Higgs bundle  $\left( V,W,\eta  \right)$ is (poly)stable if and only if the ${{K}^{p}}$-twisted $\text{SO}\left( 1,n \right)$-Higgs bundle $\left( I,\hat{W},\hat{\eta } \right)$ is (poly)stable (see Lemma 4.4 in \cite{Aparicio et al}). 

The case when  $q=p+1$ is even more special, because the relevant ${{K}^{p}}$-twisted  $\text{O}\left( q-p+1 \right)$-Higgs bundles in the pre-image of  $\Psi $ are now rank 2 orthogonal bundles. In this case, when the first Stiefel-Whitney class ${{w}_{1}}\left( \hat{W},{{Q}_{{\hat{W}}}} \right)$ vanishes, then the structure group of $\hat{W}$ reduces to $\text{SO}\left( 2,\mathbb{C} \right)\cong {{\mathbb{C}}^{*}}$ and thus
	\[\left( \hat{W},{{Q}_{{\hat{W}}}} \right)\cong \left( M\oplus {{M}^{-1}},\left( \begin{matrix}
   0 & 1  \\
   1 & 0  \\
\end{matrix} \right) \right),\] 
for a degree $d$ holomorphic line bundle $M\in \text{Pi}{{\text{c}}^{d}}\left( X \right)$, while for stability reasons $d$ is an integer in the interval $\left[ 0,p\left( 2g-2 \right) \right]$. This degree is a new topological invariant, which distinguishes extra components of the moduli space $\mathsf{\mathcal{M}}\left( \text{SO}\left( p,p+1 \right) \right)$, and in \cite{Collier} is proven the following:

\begin{theorem}[Theorem 4.1 in \cite{Collier}] For each integer $d\in \left( 0,p\left( 2g-2 \right)-1 \right]$ there is a smooth connected component ${{\mathsf{\mathcal{R}}}_{d}}\left( \text{SO}\left( p,p+1 \right) \right)$ of the moduli space $\mathsf{\mathcal{R}}\left( \text{SO}\left( p,p+1 \right) \right)$, which does not contain representations with compact Zariski closure.
\end{theorem}

Since all points in these $p\left( 2g-2 \right)-1$ many components are smooth, all corresponding fundamental group representations are \emph{irreducible representations}. In fact, these representations are conjectured in \cite{Collier} to have Zariski dense image. For this reason we shall call these components \emph{exceptional} to distinguish them among the rest of the components of the character varieties  $\mathsf{\mathcal{R}}\left( \text{SO}\left( p,q \right) \right)$ that are not detected by the fixed points of the ${{\mathbb{C}}^{*}}$-action. 

\begin{definition}\label{exceptional_comp_def}
The connected components of the moduli space $\mathsf{\mathcal{M}}\left( \text{SO}\left( p,p+1 \right) \right)$, which are smooth, will be called the \textit{exceptional components} of the moduli space $\mathsf{\mathcal{M}}\left( \text{SO}\left( p,p+1 \right) \right)$.
\end{definition}

For each integer $0<d\le p\left( 2g-2 \right)-1$, the Higgs bundles $\left( V,W,\eta  \right)$ in the exceptional components are described by the map $\Psi $ from (\ref{201}) as follows:
	\[\left( V,{{Q}_{V}} \right)=\left( {{K}^{p-1}}\oplus {{K}^{p-3}}\oplus \ldots \oplus {{K}^{3-p}}\oplus {{K}^{1-p}},\left( \begin{matrix}
   {} & {} & 1  \\
   {} &  & {}  \\
   1 & {} & {}  \\
\end{matrix} \right) \right),\] 
\begin{equation}\label{203}
\left( W,{{Q}_{W}} \right)=\left( M\oplus {{K}^{p-2}}\oplus {{K}^{p-4}}\oplus \ldots \oplus {{K}^{4-p}}\oplus {{K}^{2-p}}\oplus {{M}^{-1}},\left( \begin{matrix}
   {} & {} & 1  \\
   {} &  & {}  \\
   1 & {} & {}  \\
\end{matrix} \right) \right),
\end{equation} 
\[\eta =\left( \begin{matrix}
   0 & 0 & \ldots  & 0 & \nu   \\
   1 & {{q}_{2}} & {{q}_{4}} & \cdots  & {{q}_{2p-2}}  \\
   0 & 1 & {{q}_{2}} & {} & {{q}_{2p-4}}  \\
   0 & 0 & \ddots  & {} & \vdots   \\
   \vdots & {} & \ddots  & 1 & q_{2}  \\
    0 & 0 & \ldots  & 0 & \mu \\
\end{matrix} \right):V\to W\otimes K,\]
for $M\in \text{Pi}{{\text{c}}^{d}}\left( X \right)$, and sections $\mu \in {{H}^{0}}\left( {{M}^{-1}}{{K}^{p}} \right)\backslash \left\{ 0 \right\}$  and $\nu \in {{H}^{0}}\left( M{{K}^{p}} \right)$ with $0\ne \mu \ne \lambda \nu $. 

In the case when  $d=p\left( 2g-2 \right)$, then  $\left( V,W,\eta  \right)$ lies in the \emph{Hitchin component}\index{Hitchin component} of  $\mathsf{\mathcal{M}}\left( \text{SO}\left( p,p+1 \right) \right)$ with data

\[\left( V,{{Q}_{V}} \right)=\left( {{K}^{p-1}}\oplus {{K}^{p-3}}\oplus \ldots \oplus {{K}^{3-p}}\oplus {{K}^{1-p}},\left( \begin{matrix}
   {} & {} & 1  \\
   {} &  & {}  \\
   1 & {} & {}  \\
\end{matrix} \right) \right),\] 
\[\left( W,{{Q}_{W}} \right)=\left( {{K}^{p}}\oplus {{K}^{p-2}}\oplus {{K}^{p-4}}\oplus \ldots \oplus {{K}^{4-p}}\oplus {{K}^{2-p}}\oplus {{K}^{-p}},\left( \begin{matrix}
   {} & {} & 1  \\
   {} &  & {}  \\
   1 & {} & {}  \\
\end{matrix} \right) \right),\] 
\[\eta =\left( \begin{matrix}
   {{q}_{2}} & {{q}_{4}} & \ldots  & {{q}_{2p-2}} & {{q}_{2p}}  \\
   1 & {{q}_{2}} & {{q}_{4}} & \cdots  & {{q}_{2p-2}}  \\
   0 & 1 & {{q}_{2}} & {} & {{q}_{2p-4}}  \\
   0 & 0 & \ddots  & {} & \vdots   \\
   \vdots & {} & \ddots  & 1 & q_{2}  \\
    0 & 0 & \ldots  & 0 & 1 \\
\end{matrix} \right):V\to W\otimes K.\]

\subsection{Hitchin equations for orthogonal groups}
The moduli space $\mathsf{\mathcal{M}}\left( \text{SO}\left( p,q \right) \right)$ of polystable $\text{SO}\left( p,q \right)$-Higgs bundles is alternatively viewed as the moduli space of polystable pairs  $\left( {{{\bar{\partial }}}_{E}},\Phi  \right)$ modulo the gauge group  $\mathsf{\mathcal{G}}\left( \textbf{E} \right)$, where ${{\bar{\partial }}_{E}}$ is a Dolbeault operator on a principal  $\text{SO}\left( p,\mathbb{C} \right)\times \text{SO}\left( q,\mathbb{C} \right)$-bundle $\textbf{E}$ and  $\Phi \in {{\Omega }^{1,0}}\left( E\left( {{\mathfrak{m}}^{\mathbb{C}}} \right) \right)$ satisfying  ${{\bar{\partial }}_{E}}\left( \Phi  \right)=0$, for the $(-1)$-eigenspace  $\mathfrak{m}$ in the Cartan decomposition of the Lie algebra of the group  $\text{SO}\left( p,q \right)$. 

For the principal  $\text{SO}\left( p,\mathbb{C} \right)\times \text{SO}\left( q,\mathbb{C} \right)$-bundle $\textbf{E}$ equipped with a Dolbeault operator ${{\bar{\partial }}_{E}}$, the gauge group  
$$\mathsf{\mathcal{G}}\left( \textbf{E} \right)\cong {{\Omega }^{0}}\left( {{E}_{\text{SO}\left( p,\mathbb{C} \right)}}\left( \text{SO}\left( p,\mathbb{C} \right) \right) \right)\times {{\Omega }^{0}}\left( {{E}_{\text{SO}\left( q,\mathbb{C} \right)}}\left( \text{SO}\left( q,\mathbb{C} \right) \right) \right)$$
acts on the operators ${{\bar{\partial }}_{E}}$ by conjugation, where $\textbf{E}={{E}_{\text{SO}\left( p,\mathbb{C} \right)}}\times {{E}_{\text{SO}\left( q,\mathbb{C} \right)}}$. Now a Dolbeault operator on $\textbf{E}$ corresponds to a connection $A$ on the reduction $V$ of $\textbf{E}$ to $\text{SO}\left( p,\mathbb{C} \right)\times \text{SO}\left( q,\mathbb{C} \right)$ and consider a Higgs field $\Phi \in {{\Omega }^{1,0}}\left( V\left( {{\mathfrak{m}}^{\mathbb{C}}} \right) \right)$. 

The group $G=\text{SO}\left( p,q \right)$ is a real form of $\text{SO}\left( p+q,\mathbb{C} \right)$. It coincides with the compact real form when $p=q=0$ and with the split real form when $p=q$ for $p+q$ even, or when $q=p+1$ for $p+q$ odd. Matrix conjugation $\tau \left( X \right)=\bar{X}$ defines the compact real form; indeed, we check 
\begin{align*}
\mathfrak{so}\left( p+q \right) & = \left\{ X\in \mathfrak{so}\left( p+q,\mathbb{C} \right)\left| X=\bar{X} \right. \right\}\\
& = \left\{ X\in \mathfrak{so}\left( p+q,\mathbb{R} \right)\left| X+{{X}^{T}}=0 \right. \right\}.
\end{align*}
If we locally write $\Phi =\varphi dz$, then a calculation shows that  
$$\left[ \Phi ,\tau \left( \Phi  \right) \right]=\left( \begin{matrix}
   -\varphi {{\varphi }^{*}}-\bar{\varphi }{{\varphi }^{T}} & {}  \\
   {} & -{{\varphi }^{T}}\bar{\varphi }-{{\varphi }^{*}}\varphi   \\
\end{matrix} \right).$$
The Hitchin-Kobayashi correspondence for $G=\text{SO}\left( p,q \right)$ provides that if an $\text{SO}\left( p,q \right)$-Higgs bundle $\left( V, Q_{V}, W, Q_{W}, \eta \right)$ is polystable, then and only then the pair $\left( A, \Phi \right)$ as considered above satisfies the Hitchin equation\index{Hitchin equation} 
\[\left\{ \begin{matrix}
   {{F}_{A}}-\left[ \Phi ,\tau \left( \Phi  \right) \right]=0  \\
   {{{\bar{\partial }}}_{A}}\left( \Phi  \right)=0,  \\
\end{matrix} \right.\]
where $F_{A}$ denotes the curvature of the unique connection compatible with the structure group reduction and the holomorphic structure. For a local description of the connection $A=\left( {{A}_{1}},{{A}_{2}} \right)$ the equation ${{F}_{A}}-\left[ \Phi ,\tau \left( \Phi  \right) \right]=0$ becomes the pair 
\begin{align*}
& {{F}_{{{A}_{1}}}}+\varphi {{\varphi }^{*}}+\bar{\varphi }{{\varphi }^{T}}={{F}_{{{A}_{1}}}}+2\text{Re}\left( \varphi {{\varphi }^{*}} \right)=0\\ 
& {{F}_{{{A}_{2}}}}+{{\varphi }^{T}}\bar{\varphi }+{{\varphi }^{*}}\varphi ={{F}_{{{A}_{2}}}}+2\text{Re}\left( {{\varphi }^{T}}\bar{\varphi } \right)=0.
\end{align*}

\subsection{Model parabolic $\text{SL}\left( 2,\mathbb{R} \right)$-Higgs bundles} Parabolic $\text{SL}\left( 2,\mathbb{R} \right)$-Higgs bundles corresponding via the non-abelian Hodge correspondence to Fuchsian representations of the fundamental group of  a punctured surface into  the group $\text{PSL}\left( 2,\mathbb{R} \right)$ were first identified by Biswas, Ar\'{e}s-Gastesi and Govindarajan in \cite{BAG}; see also the article of Mondello \cite{Mondello} for a complete topological description of the relevant representation space. We next investigate these pairs more closely.

Let $D=\left\{ {{x}_{1}},\ldots ,{{x}_{s}} \right\}$ be a finite collection of $s$-many points on a closed genus $g$ Riemann surface $X$, such that $2g-2+s>0$. Let $K$ denote the canonical line bundle over the Riemann surface $X$. Consider the pair $\left( E,\Phi  \right)$, where:
\begin{enumerate}
  \item $E:={{\left( L\otimes \iota  \right)}^{*}}\oplus L$, \\
  where $L$ is a line bundle with ${{L}^{2}}={{K}}$ and $\iota :={{\mathsf{\mathcal{O}}}_{X}}\left( D \right)$ denotes the line bundle over the divisor $D$; we equip the bundle $E$ with a parabolic structure given by a trivial flag ${{E}_{{{x}_{i}}}}\supset \left\{ 0 \right\}$ and weight $\frac{1}{2}$ for every $1\le i\le s$.
  \item $\Phi :=\left( \begin{matrix}
   0 & 1  \\
   0 & 0  \\
\end{matrix} \right)\in {{H}^{0}}\left( X,\text{End}\left( E \right)\otimes K \otimes \iota \right)$.
\end{enumerate}
Then, the pair $\left( E,\Phi  \right)$ is a stable parabolic $\text{SL}\left( 2,\mathbb{R} \right)$-Higgs bundle with parabolic degree $\text{pardeg} (E)$ $= 0$. Therefore, from the non-abelian Hodge correspondence on non-compact curves \cite{Simpson-noncompact}, the vector bundle $E$ supports a tame harmonic metric; the local estimate for this Hermitian metric on $E$ restricted to the line bundle $L$ is
\[{{r}^{\frac{1}{2}}}{{\left| \log r \right|}^{\frac{1}{2}}},\] 
for $r=\left| z \right|$. Indeed, if $\beta \in \mathbb{R}$ denotes in general the weights in the filtration of the filtered local system $F$ corresponding to a parabolic Higgs bundle with weights $\alpha $, for $0\le \alpha <1$, then, if ${{W}_{k}}$ is the span of vectors of weights $\le k$, the weight filtration of  $\text{Re}{{\text{s}}_{x}}\left( F \right)$ describes the behavior of the tame harmonic map under the local estimate 
\[C{{r}^{\beta }}{{\left| \log r \right|}^{\frac{k}{2}}}.\]
In our case, the weight is  $\alpha =\frac{1}{2}=\beta $ and the residue at each point  ${{x}_{i}}\in D$ is $N=\left( \begin{matrix}
   0 & 1  \\
   0 & 0  \\
\end{matrix} \right)$, an upper triangular $2\times 2$ nilpotent matrix. Thus, its weight filtration is ${{W}_{-2}}=0$, ${{W}_{-1}}={{W}_{0}}=\operatorname{Im}\left( N \right)=\ker \left( N \right)$, and  ${{W}_{1}}=$ the whole space. Therefore, in the notation of Simpson from \cite{Simpson-noncompact} we have $L\subset {{W}_{1}}$ and $L\not\subset {{W}_{0}}={{W}_{-1}}$, while the Hermitian metric on the line bundle $L$ is locally 
\[{{r}^{\alpha }}{{\left| \log r \right|}^{\frac{k}{2}}}={{r}^{\frac{1}{2}}}{{\left| \log r \right|}^{\frac{1}{2}}}.\] 
For the parabolic dual ${{\left( L\otimes \iota  \right)}^{*}}$, the weight is by construction equal to $1-\frac{1}{2}$ and in the weight filtration for the residue it is  ${{\left( L\otimes \iota  \right)}^{*}}\subset {{W}_{-1}}$ and $L\not\subset {{W}_{1}}$. Thus, the Hermitian metric on ${{\left( L\otimes \iota  \right)}^{*}}$ is locally 
\[{{r}^{\alpha }}{{\left| \log r \right|}^{\frac{k}{2}}}={{r}^{1-\frac{1}{2}}}{{\left| \log r \right|}^{-\frac{1}{2}}}={{r}^{\frac{1}{2}}}{{\left| \log r \right|}^{-\frac{1}{2}}}.\] 
In conclusion, the metric on $\text{Hom}\left( L,{{\left( L\otimes \iota  \right)}^{*}} \right)$ is induced by the restricted tame harmonic metric of $E$ on the line bundles $L$ and ${{\left( L\otimes \iota  \right)}^{*}}$, as a section of ${{L}^{*}}\otimes {{\left( L\otimes \iota  \right)}^{*}}$ and is locally described by  
\[{{r}^{-\frac{1}{2}}}{{\left| \log r \right|}^{-\frac{1}{2}}}\cdot {{r}^{\frac{1}{2}}}{{\left| \log r \right|}^{-\frac{1}{2}}}={{\left| \log r \right|}^{-1}},\]
for $r=\left| z \right|$. Subsequently, the metric on the tangent bundle ${{L}^{-2}}$ is locally
\[{{r}^{-\frac{1}{2}}}{{\left| \log r \right|}^{-\frac{1}{2}}}\cdot {{r}^{-\frac{1}{2}}}{{\left| \log r \right|}^{-\frac{1}{2}}}={{r}^{-1}}{{\left| \log r \right|}^{-1}}\] 
and is therefore the Poincar\'{e} metric of the punctured disk on $\mathbb{C}$; we refer the interested reader to \cite{BAG} and \cite{Simpson-noncompact} for further information.

\subsection{Parabolic $\text{SO}\left( p,p+1 \right)$-models}

In this subsection we construct model parabolic $\text{SO}\left( p,p+1 \right)$-Higgs bundles which shall be later on used in providing the desired (non-parabolic) $\text{SO}\left( p,p+1 \right)$-models in the exceptional components over the complex connected sum of Riemann surfaces. Of critical importance to this construction are the parabolic  $\text{SL}\left( 2,\mathbb{R} \right)$-Higgs bundles $\left( E,\Phi  \right)$ of Biswas, Ar\'{e}s-Gastesi and Govindarajan from \cite{BAG} described earlier. As we have seen in \S \ref{local_model}, from the gauge theoretic viewpoint, a model solution to the $\text{SL}\left( 2,\mathbb{C} \right)$-Hitchin equations that corresponds to the polystable pair  $\left( E,\Phi  \right)$ is given by a pair $\left( {{A}^{\bmod }},{{\Phi }^{\bmod }} \right)$, where
\[{{A}^{\bmod }}=0,\text{   }{{\Phi }^{\bmod }}=\left( \begin{matrix}
   C & 0  \\
   0 & -C  \\
\end{matrix} \right)\frac{dz}{z}\]
over a punctured disk with $z$-coordinates around the puncture with the condition that $C\in \mathbb{R}$ with $C\ne 0$, and that the meromorphic quadratic differential $q:=\det {{\Phi }^{\bmod }}$ has at least one simple zero.

\subsubsection{Models via the irreducible representation $\text{SL}\left( 2,\mathbb{R} \right)\hookrightarrow \text{SO}\left( p,p+1 \right)$}\label{first_models}\label{models_via_irreducible_repr}
We next construct model parabolic $\text{SO}\left( p,p+1 \right)$-Higgs bundles lying inside the parabolic Teichm\"{u}ller component\index{parabolic Teichm\"{u}ller component}\index{Hitchin component!parabolic} for $\text{SO}\left( p,p+1 \right)$. The general construction of this component was carried out in \cite{KSZ}, while in the non-parabolic case, a detailed construction of models can be found in \cite{Aparicio}.

The connected component of the special orthogonal group containing the identity $\text{SO}_{0}\left( p,p+1 \right)$ is a split real form of $\text{SO}\left( 2p+1,\mathbb{C} \right)$. The Lie algebra of $\text{SO}\left( p,p+1 \right)$ is 
\begin{align*}
\mathfrak{so}\left( p,p+1 \right) & = \left\{ X\in \mathfrak{sl}\left( 2p+1,\mathbb{R} \right)\left| {{X}^{t}}{{I}_{p,p+1}}+{{I}_{p,p+1}}X \right.=0 \right\}\\
& = \left\{ \left( \begin{matrix}
   {{X}_{1}} & {{X}_{2}}  \\
   X_{2}^{t} & {{X}_{3}}  \\
\end{matrix} \right) \right.\left| {{X}_{1}},{{X}_{3}}\text{ real skew-sym}\text{. of rank }p,p+1\text{ resp}\text{.;} \right.\\
& \,\,\,\,\,\,\,\,\,\, \left. {{X}_{2}}\text{ real }\left( p\times \left( p+1 \right) \right)\text{-matrix} \right\}.
\end{align*}
The Lie algebra $\mathfrak{so}\left( p,p+1 \right)$ admits a Cartan decomposition $\mathfrak{so}\left( p,p+1 \right)=\mathfrak{h}\oplus \mathfrak{m}$ into its $\left( \pm 1 \right)$-eigenspaces, where 
\[\mathfrak{h}=\mathfrak{so}\left( p \right)\times \mathfrak{so}\left( p+1 \right)=\left\{ \left( \begin{matrix}
   {{X}_{1}} & 0  \\
   0 & {{X}_{3}}  \\
\end{matrix} \right)\left| {{X}_{1}}\in \mathfrak{so}\left( p \right),{{X}_{3}}\in \mathfrak{so}\left( p+1 \right) \right. \right\},\]
\[\mathfrak{m}=\left\{ \left( \begin{matrix}
   0 & {{X}_{2}}  \\
   X_{2}^{t} & 0  \\
\end{matrix} \right)\left| {{X}_{2}}\text{ real }\left( p\times \left( p+1 \right) \right)\text{-matrix} \right. \right\}.\]
The Cartan decomposition of the complex Lie algebra is
\[\mathfrak{so}( 2p+1, \mathbb{C} )= \left(\mathfrak{so}\left( p, \mathbb{C} \right)\times \mathfrak{so}\left( p+1, \mathbb{C} \right) \right) \oplus {\mathfrak{m}^{\mathbb{C}}}, \]
where 
\[{{\mathfrak{m}}^{\mathbb{C}}}=\left\{ \left( \begin{matrix}
   0 & {{X}_{2}}  \\
   -X_{2}^{t} & 0  \\
\end{matrix} \right)\left| {{X}_{2}}\text{ complex }\left( p\times \left( p+1 \right) \right)\text{-matrix} \right. \right\}.\] 
If $\mathfrak{c}$ is a Cartan subalgebra of $\mathfrak{so}\left( p,p+1 \right)$ and  $\Delta $ is the set of the corresponding roots, then the element  
	\[\sum\limits_{\alpha \in \Delta }{{{c}_{\alpha }}{{X}_{\alpha }}}\in \mathfrak{so}\left( 2p+1,\mathbb{C} \right),\] 
is regular nilpotent, for ${{c}_{\alpha }}\ne 0$, $\alpha \in \Pi $ and ${{X}_{\alpha }}$ a root vector for  $\alpha $, where
\[{{\Delta }^{+}}=\left\{ {{e}_{i}}\pm {{e}_{j}},\text{ with }1\le i<j\le p \right\}\cup \left\{ {{e}_{i}},1\le i\le p \right\},\] 
\[\Pi =\left\{ {{a}_{i}}={{e}_{i}}-{{e}_{i+1}},\,\,1\le i\le p-1 \right\}\cup \left\{ {{a}_{p}}={{e}_{p}} \right\}.\] 
The corresponding root vectors are 

${{X}_{{{e}_{i}}-{{e}_{j}}}} = {{E}_{i,j}}-{{E}_{p+j,p+i}}$

${{X}_{{{e}_{i}}+{{e}_{j}}}} = {{E}_{i,p+j}}-{{E}_{j,p+i}}$

${{X}_{{{e}_{i}}}} = {{E}_{i,2p+1}}-{{E}_{2p+1,p+i}}$

${{X}_{-{{e}_{i}}}} = {{E}_{p+i,2p+1}}-{{E}_{2p+1,i}}$.

Now, let  $x:=\sum\limits_{i=1}^{p}{2\left( p+1-i \right)\left( {{E}_{i,i}}-{{E}_{p+i,p+i}} \right)}$ and take $e:=\sum\limits_{a\in \Pi }{{{X}_{a}}}$. From this choice it is then satisfied that  $\left[ x,e \right]=2e$, for the semisimple element $x$ and the regular nilpotent element $e$. Moreover, the conditions  $\left[ x,\tilde{e} \right]=-2\tilde{e}$ and $\left[ e,\tilde{e} \right]=x$ determine another nilpotent element  $\tilde{e}$, thus the triple  $\left\langle x,e,\tilde{e} \right\rangle \cong \mathfrak{sl}\left( 2,\mathbb{C} \right)$ defines a principal 3-dimensional Lie subalgebra of $\mathfrak{so}\left( p,p+1 \right)$. 

The adjoint action $\left\langle x,e,\tilde{e} \right\rangle \cong \mathfrak{so}\left( 2,\mathbb{C} \right)\to \text{End}\left( \mathfrak{so}\left( 2p+1,\mathbb{C} \right) \right)$ of this subalgebra decomposes $\mathfrak{so}\left( p,p+1 \right)$ as a direct sum of irreducible representations
\[\mathrm{}\left( 2p+1,\mathbb{C} \right)=\underset{i=1}{\overset{p}{\mathop{\oplus }}}\,{{V}_{i}},\] 
with $\dim{{V}_{i}}=4i-1$, for $1\le i\le p$. Therefore,  ${{V}_{i}}={{S}^{4i-2}}{{\mathbb{C}}^{2}}$,   $1\le i\le p$ with eigenvalues $4i-2$, $4i-4$,..., $-4i+4$, $-4i+2$ for the action of $\text{ad}x$, and the highest weight vectors are  ${{e}_{1}},\ldots ,{{e}_{p}}$, where  ${{e}_{i}}$ has eigenvalue  $4i-2$, for $1\le i\le p$. 

Considering the representation
\[\mathfrak{sl}\left( 2,\mathbb{C} \right)\to \mathfrak{so}\left( 2p+1,\mathbb{C} \right),\] 
for $\mathfrak{so}\left( 2p+1,\mathbb{C} \right)={{S}^{2}}{{\mathbb{C}}^{2}}+{{S}^{6}}{{\mathbb{C}}^{2}}+\ldots +{{S}^{4p-2}}{{\mathbb{C}}^{2}}={{\Lambda }^{2}}\left( {{S}^{2p}}{{\mathbb{C}}^{2}} \right)$, we may next deduce the defining data  $\left( {{E}_{1}},{{\Phi }_{1}} \right)$ for a parabolic  $\text{SO}\left( p,p+1 \right)$-Higgs bundle inside the parabolic Teichm\"{u}ller component for the split real form  ${{G}^{r}}=\text{S}{{\text{O}}_{0}}\left( p,p+1 \right)$. The parabolic vector bundle is obtained from the  $\left( 2p \right)$-th symmetric power of the parabolic  $\text{SL}\left( 2,\mathbb{R} \right)$-bundle in the Teichm\"{u}ller component, as follows.

Let ${{X}_{1}}$ be a compact Riemann surface of genus  ${{g}_{1}}$,   ${{D}_{1}}=\left\{ {{p}_{1}},\ldots ,{{p}_{s}} \right\}$ a collection of $s$ distinct points on  ${{X}_{1}}$ and let ${{L}_{1}}\to {{X}_{1}}$ with  $L_{1}^{2}\cong {{K}_{{{X}_{1}}}}$  and ${{\iota }_{1}}={{\mathsf{\mathcal{O}}}_{{{X}_{1}}}}\left( {{D}_{1}} \right)$. Consider the parabolic vector bundle ${{\left( {{L}_{1}}\otimes {{\iota }_{1}} \right)}^{*}}\oplus {{L}_{1}}$ over $\left( {{X}_{1}},{{D}_{1}} \right)$, equipped with a trivial flag and weight $\frac{1}{2}$. Then, the vector bundle  ${{E}_{1}}$ of a model parabolic  $\text{SO}\left( p,p+1 \right)$-Higgs bundle in the parabolic Teichm\"{u}ller component is 
\begin{align*}
{{E}_{1}}: & = {{S}^{2p}}\left( {{\left( {{L}_{1}}\otimes {{\iota }_{1}} \right)}^{*}}\oplus {{L}_{1}} \right)\\
& = L_{1}^{-2p}\otimes \mathsf{\mathcal{O}}\left( -p{{D}_{1}} \right)\oplus L_{1}^{-2p+2}\otimes \mathsf{\mathcal{O}}\left( \left( 1-p \right){{D}_{1}} \right)\oplus \ldots \\
& \ldots \oplus L_{1}^{2p-2}\otimes \mathsf{\mathcal{O}}\left( \left( p-1 \right){{D}_{1}} \right)\oplus L_{1}^{2p}\otimes \mathsf{\mathcal{O}}\left( p{{D}_{1}} \right)\\
& = K_{1}^{-p}\otimes \mathsf{\mathcal{O}}\left( -p{{D}_{1}} \right)\oplus K_{1}^{-\left( p-1 \right)}\otimes \mathsf{\mathcal{O}}\left( \left( 1-p \right){{D}_{1}} \right)\oplus \ldots \\
& \ldots \oplus K_{1}^{p-1}\otimes \mathsf{\mathcal{O}}\left( \left( p-1 \right){{D}_{1}} \right)\oplus K_{1}^{p}\otimes \mathsf{\mathcal{O}}\left( p{{D}_{1}} \right),
\end{align*}
equipped with a trivial parabolic flag and weight 0. 

\begin{remark}
Note that in the above description we have included the consideration for the parabolic structure in a symmetric power of a parabolic bundle. In fact, restricting attention on the first original term ${{\left( {{L}_{1}}\otimes {{\iota }_{1}} \right)}^{*}}$ with weight $\frac{1}{2}$, the symmetric power ${{S}^{2p}}\left( {{\left( {{L}_{1}}\otimes {{\iota }_{1}} \right)}^{*}} \right)$ is the line bundle $L_{1}^{-2p}\otimes \mathsf{\mathcal{O}}\left( -2p{{D}_{1}} \right)$ with weight  $2p\cdot \frac{1}{2}=p$. However, we obtain a well-defined parabolic bundle by reducing the weight to a number within the interval  $\left[ 0,1 \right)$, this means, by tensoring $L_{1}^{-2p}\otimes \mathsf{\mathcal{O}}\left( -2p{{D}_{1}} \right)$ by $\mathsf{\mathcal{O}}\left( p{{D}_{1}} \right)$. We thus get  $K_{1}^{-p}\otimes \mathsf{\mathcal{O}}\left( -p{{D}_{1}} \right)$ with weight 0, as appears in the first term of the parabolic bundle ${{E}_{1}}$ above. 
\end{remark}

The Higgs field in the parabolic  $\text{SO}\left( p,p+1 \right)$-Teichm\"{u}ller component is given by  
\[\tilde{e}+{{q}_{1}}{{e}_{1}}+\ldots +{{q}_{p}}{{e}_{p}},\] 
for $\left( {{q}_{1}},\ldots ,{{q}_{p}} \right)\in \underset{i=1}{\overset{p}{\mathop{\oplus }}}\,{{H}^{0}}\left( K_{1}^{2i}\otimes \iota _{1}^{2i-1} \right)$ and  ${{e}_{1}},\ldots ,{{e}_{p}}$ are the highest weight vectors. From the set of simple roots of  $\mathfrak{so}\left( p,p+1 \right)$,   
\[\Pi =\left\{ {{e}_{i}}-{{e}_{i+1}},\,1\le i\le p-1 \right\}\cup \left\{ {{e}_{p}} \right\},\] 
we obtain the 3-dimensional subalgebra $\left\langle x,e,\tilde{e} \right\rangle \cong \mathfrak{sl}\left( 2,\mathbb{C} \right)\hookrightarrow \mathfrak{so}\left( p,p+1 \right)$, with

\begin{equation}\label{302}
x=
  \left(\begin{array}{@{}cccc|cccc@{}|c@{}}
    2p &  &  &  &  &  &  &  & \\
      & 2(p-1) &  &  &  &  &  &  &  \\
     &  & \ddots  &  &  &  &  &  &  \\
       &  &  & 2 &  &  &  &  &  \\\hline
       &  &  &  & -2p &  &  &  &  \\
   &  &  &  &  & -2(p-1) &  &  &  \\
     &  &  &  &  &  & \ddots &  &  \\
       &  &  &  &  &  &  & -2 &  \\\hline
       &  &  &  &  &  &  &  &  0
  \end{array}\right), 
\end{equation}

\begin{equation}\label{regular nilpotent}
e=
  \left(\begin{array}{@{}cccc|cccc@{}|c@{}}
   0  & 1 &  &  &  &  &  &  & \\
      & \ddots & \ddots &  &  &  &  &  &  \\
     &  &  & 1 &  &  &  &  &  \\
       &  &  & 0 &  &  &  &  & 1 \\\hline
       &  &  &  & 0  &  &  &  &  \\
   &  &  &  & -1 & \ddots &  &  &  \\
     &  &  &  &  & \ddots &  &  &  \\
       &  &  &  &  &  & -1 & 0 &  \\\hline
       &  &  &  &  &  &  & -1 & 0
  \end{array}\right),
\end{equation}
the semisimple and regular nilpotent element respectively; from these we may also determine the third element in the principal 3-dimensional subalgebra of $\mathfrak{so}(p, p+1)$:
\begin{equation}\label{e-tilda}
\tilde{e}=
  \left(\begin{array}{@{}cc|cc@{}|c@{}}

    & A &  &  &  \\

     \hline
     &  & B &  & D \\
 
  \hline
    & C & &  & 
  \end{array}\right),
\end{equation}
where
\[A = \left( \begin{matrix}
   0  &  &  &   \\
    2p  & \ddots &  &   \\
     & 2p+2(p-1) &  &   \\
     & \ddots &  &   \\
       &  & 2p+2(p-1) + \dots + 2 \cdot 2 & 0 
\end{matrix} \right)\]   is a $p \times p$ block with zeros on the main diagonal,
\[B = \left( \begin{matrix}
  0  & -2p &  &  &  \\
    & \ddots &-2p-2(p-1)  &  &  \\
      &  &  & \ddots &\\
 &  & &   & -2p-2(p-1) - \dots - 2 \cdot 2   \\
    &  &  &  & 0
\end{matrix} \right)\]  is a $p \times p$ block with zeros on the main diagonal, and
\[C = \left( \begin{matrix}
   0  & \cdots & 0 &   2p+2(p-1) + \dots + 2  
\end{matrix} \right) \text{ is a } 1 \times p \text{ block}, \]
\[D = \left( \begin{matrix}
   0  \\
   \vdots \\
   0 \\
   -2p-2(p-1) - \dots - 2  
\end{matrix} \right) \text{ is a } p \times 1 \text{ block}. \]

From the analysis above we deduce that a model parabolic Higgs pair lying inside the parabolic $\text{S}{{\text{O}}_{0}}\left( p,p+1 \right)$-Hitchin component which is a local minimum of the Hitchin functional, when viewed as an $\text{SL}\left( 2p+1,\mathbb{C} \right)$-pair, is a pair $\left( {{E}_{1}},{{\Phi }_{1}} \right)$ with 
\begin{itemize}
\item ${{E}_{1}}=K_{1}^{-p}\otimes \mathsf{\mathcal{O}}\left( -p{{D}_{1}} \right)\oplus K_{1}^{-\left( p-1 \right)}\otimes \mathsf{\mathcal{O}}\left( \left( 1-p \right){{D}_{1}} \right)\oplus \ldots \oplus K_{1}^{\left( p-1 \right)}\otimes \mathsf{\mathcal{O}}\left( \left( p-1 \right){{D}_{1}} \right)\oplus K_{1}^{p}\otimes \mathsf{\mathcal{O}}\left( p{{D}_{1}} \right)$ \\
a parabolic vector bundle of rank $2p+1$ over  $\left( {{X}_{1}},{{D}_{1}} \right)$ equipped with a parabolic structure given by a trivial flag and weight 0, 
\item ${{\Phi }_{1}}=\left( \begin{matrix}
   0 & 1 & 0 & \cdots & & 0 \\
   0 & 0 & 1 & 0 & \cdots &  0 \\
   \vdots  & {} & {} & {} & {} & {}  \\
   0 & \cdots  &  & & 0 & 1  \\
\end{matrix} \right):{{E}_{1}}\to {{E}_{1}}\otimes {{K}_{1}}\otimes {{\iota }_{1}}$ \\
as a $p\times \left( p+1 \right)$-matrix.
\end{itemize}

The next lemma is analogous to Lemma 2.1 in \cite{BAG}.

\begin{lemma}
 The parabolic Higgs bundle $\left( {{E}_{1}},{{\Phi }_{1}} \right)$ above is a parabolic stable Higgs bundle of parabolic degree zero.
\end{lemma}

\begin{proof}
The proof that $\text{pardeg}\left( {{E}_{1}} \right) =0$ is immediate, following the properties of the parabolic degree on a direct sum and the dual of a parabolic bundle. The ${{\Phi }_{1}}$-invariant proper subbundles of  ${{E}_{1}}$ are of the form
	\[K_{1}^{-p}\otimes \mathsf{\mathcal{O}}\left( -p{{D}_{1}} \right)\oplus K_{1}^{-\left( p-1 \right)}\otimes \mathsf{\mathcal{O}}\left( -\left( p-1 \right){{D}_{1}} \right)\oplus \ldots \oplus K_{1}^{m-p}\otimes \mathsf{\mathcal{O}}\left( \left( m-p \right){{D}_{1}} \right),\] 
for $0\le m\le 2p-1$. One now checks that these all have negative parabolic degree, that is, strictly less than  $\text{pardeg}\left( {{E}_{1}} \right)$.
\end{proof}
Therefore, from the punctured-surface version of the non-abelian Hodge correspondence \cite{Simpson-noncompact}, there is a tame harmonic metric on the vector bundle ${{E}_{1}}$. Let ${{A}_{1}}$ denote the associated Chern connection. Parabolic stability implies the existence of a complex gauge transformation, unique up to modification by a unitary gauge, such that  $\left( {{A}_{1}},{{\Phi }_{1}} \right)$ solves the Hitchin equations. 

In a suitably chosen local holomorphic trivialization of  ${{E}_{1}}$, the pair  $\left( {{A}_{1}},{{\Phi }_{1}} \right)$ is asymptotic to a model solution, which after a unitary change of frame can be written locally over a punctured neighborhood around a point ${{p}_{i}}\in {{D}_{1}}$ as  
	\[A_{1}^{\bmod }=0,\,\,\,\Phi _{1}^{\bmod }=Cx\frac{dz}{z},\] 
where $x$ denotes the semisimple element from (\ref{302}) and $z$ the local coordinate around the point ${{p}_{i}}\in {{D}_{1}}$. 

\subsubsection{Models via the general map $\Psi$}\label{second_models}

Let ${{X}_{2}}$ be a compact Riemann surface of genus ${{g}_{2}}$ and ${{D}_{2}}=\left\{ {{q}_{1}},\ldots ,{{q}_{s}} \right\}$ a collection of $s$ points on  ${{X}_{2}}$. Let ${{\iota }_{2}}={{\mathsf{\mathcal{O}}}_{{{X}_{2}}}}\left( {{D}_{2}} \right)$. 
The second family of model parabolic $\text{SO}\left( p,p+1 \right)$-Higgs bundles is obtained via the more general map 
	\[{{\Psi }^\textit{par}}:\mathsf{\mathcal{M}}_{K_{2}^{p}\otimes \iota _{2}^{p-1}}^\textit{par}\left( \text{SO}\left( 1,2 \right) \right)\times \underset{j=1}{\overset{p-1}{\mathop{\oplus }}}\,{{H}^{0}}\left( {{X}_{2}},K_{2}^{2j}\otimes \iota _{2}^{2j-1} \right)\to {{\mathsf{\mathcal{M}}}^\textit{par}}\left( \text{SO}\left( p,p+1 \right) \right)\]
defined as in (\ref{201}), but considering also the relevant parabolic structures. Take $\left( I,\hat{W},\hat{\eta } \right)\in \mathsf{\mathcal{M}}_{K_{2}^{p}\otimes \iota _{2}^{p-1}}^\textit{par}\left( \text{SO}\left( 1,2 \right) \right)$, the moduli space of  $K_{2}^{p}$-twisted parabolic  $\text{SO}\left( 1,2 \right)$-Higgs bundles, for
\begin{itemize}
\item  $\hat{W}:=\tilde{M}\oplus {{\tilde{M}}^{\vee }}$, for  $\tilde{M}\cong \mathsf{\mathcal{O}}\left( \left( 2k-1-p \right){{D}_{2}} \right)$ with $k=1,\ldots ,p$ an integer; 
\item $I:=\wedge _\textit{par}^{2}\hat{W}\cong \wedge \tilde{M}\otimes \wedge {{\tilde{M}}^{\vee }}\cong \tilde{M}\otimes {{\tilde{M}}^{\vee }}\cong \mathsf{\mathcal{O}}$;
\item $\hat{\eta }=0$.
\end{itemize} 
Then, one gets by the definition of the map ${{\Psi }^\textit{par}}$ the triple ${{\Psi }^\textit{par}}\left( \left( I,\hat{W},\hat{\eta } \right),\left( 0,\ldots ,0 \right) \right)=:\left( V,W,\eta  \right)$, where
\begin{itemize}
\item $V=K_{2}^{p-1}\otimes \mathsf{\mathcal{O}}\left( \left( p-1 \right){{D}_{2}} \right)\oplus \ldots \oplus K_{2}^{1-p}\otimes \mathsf{\mathcal{O}}\left( \left( 1-p \right){{D}_{2}} \right)$; 
\item $W=\tilde{M}\oplus {{\tilde{M}}^{\vee }}\oplus K_{2}^{p-2}\otimes \mathsf{\mathcal{O}}\left( \left( p-2 \right){{D}_{2}} \right)\oplus \ldots \oplus K_{2}^{2-p}\otimes \mathsf{\mathcal{O}}\left( \left( 2-p \right){{D}_{2}} \right)$; 
\item $\eta =\left( \begin{matrix}
   \hat{\eta }=0 & 0 & \cdots & & 0  \\
   0 & 1 & 0 & \cdots & 0  \\
   \vdots  & {} & {} & {}  & \\
   0 & \cdots  & {} & 0 & 1 \\
\end{matrix} \right)$.
\end{itemize}
From the description of the Higgs bundle data we see that since  $\hat{\eta }=0$, the triple  $\left( V,W,\eta  \right)$ reduces to an  $\text{SO}\left( p,p-1 \right)\times \text{SO}\left( 2 \right)$-Higgs bundle whose  $\text{SO}\left( p,p-1 \right)$-factor lies in the parabolic Hitchin component. We rather define this as an  $\text{SL}\left( 2p+1,\mathbb{C} \right)$-pair $\left( {{E}_{2}},{{\Phi }_{2}} \right)$, where
\begin{itemize}
\item ${{E}_{2}}=V\oplus W=\tilde{M}\oplus {{\tilde{M}}^{\vee }}\oplus K_{2}^{-\left( p-1 \right)}\otimes \mathsf{\mathcal{O}}\left( \left( 1-p \right){{D}_{2}} \right)\oplus \ldots \oplus K_{2}^{p-1}\otimes \mathsf{\mathcal{O}}\left( \left( p-1 \right){{D}_{2}} \right)$;
\item ${{\Phi }_{2}}=\left( \begin{matrix}
   0 & 0 & 0 & \cdots & & &  & 0   \\
   0 & 0 & 0 & \cdots & & & & 0  \\
   0 & 0 & 0 & 1 & 0 & \cdots & & 0 \\
   \vdots  & \vdots  & \vdots  & 0 & 1 & 0 & \cdots & 0  \\
&  &  &  & & \ddots &  & \\
    &  &  &  & & &  & 1\\
   0 & 0 & 0 & \cdots & & &  & 0
\end{matrix} \right):{{E}_{2}}\to {{E}_{2}}\otimes {{K}_{2}}\otimes {{\iota }_{2}}$.
\end{itemize}
The  ${{\Phi }_{2}}$-invariant proper subbundles of ${{E}_{2}}$ are 
\begin{align*}
 &  \tilde{M}\oplus {{{\tilde{M}}}^{\vee }}\oplus K_{2}^{-\left( p-1 \right)}\otimes \mathsf{\mathcal{O}}\left( \left( 1-p \right){{D}_{2}} \right)  \\
  & \tilde{M}\oplus {{{\tilde{M}}}^{\vee }}\oplus K_{2}^{-\left( p-1 \right)}\otimes \mathsf{\mathcal{O}}\left( \left( 1-p \right){{D}_{2}} \right)\oplus K_{2}^{-\left( p-2 \right)}\otimes \mathsf{\mathcal{O}}\left( \left( 2-p \right){{D}_{2}} \right)  \\
   & \vdots   \\
   & \tilde{M}\oplus {{{\tilde{M}}}^{\vee }}\oplus K_{2}^{-\left( p-1 \right)}\otimes \mathsf{\mathcal{O}}\left( \left( 1-p \right){{D}_{2}} \right)\oplus \ldots \oplus K_{2}^{\left( p-2 \right)}\otimes \mathsf{\mathcal{O}}\left( \left( p-2 \right){{D}_{2}} \right),  \\
\end{align*} or, in general, these are of the form 
$$\tilde{M}\oplus {{\tilde{M}}^{\vee }}\oplus K_{2}^{-\left( p-1 \right)}\otimes \mathsf{\mathcal{O}}\left( \left( 1-p \right){{D}_{2}} \right)\oplus \ldots \oplus K_{2}^{\left( l-p \right)}\otimes \mathsf{\mathcal{O}}\left( \left( l-p \right){{D}_{2}} \right),$$ for each $1\le l\le 2p-2$. As in the previous lemma, one sees that all proper  ${{\Phi }_{2}}$-invariant subbundles of  ${{E}_{2}}$ have negative parabolic degree, while $\text{pardeg}\left( {{E}_{2}} \right)=0$. Therefore, the models  $\left( {{E}_{2}},{{\Phi }_{2}} \right)$ for every  $k=1,\ldots ,p$ are all parabolic stable. For  ${{A}_{2}}$ be the Chern connection with respect to a tame harmonic metric on  ${{E}_{2}}$, in a suitably chosen local holomorphic trivialization  of ${{E}_{2}}$, the pair $\left( {{A}_{2}},{{\Phi }_{2}} \right)$ is, after conjugation by a unitary gauge, asymptotic to a model solution which locally over a punctured neighborhood around a point  ${{q}_{j}}\in {{D}_{2}}$ is written as  
\[A_{2}^{\bmod }=0,\,\,\Phi _{2}^{\bmod }=\left( \begin{matrix}
   0 & {} & {} & {} & {} & {} & {} & {} & {}  \\
   {} & 0 & {} & {} & {} & {} & {} & {} & {} \\
   {} & {} & 2\left( p-1 \right)C & {} & {} & {} & {} & {} & {} \\
   {} & {} & {} & \ddots  & {} & {} & {} & {} & {}  \\
       & {} & {} & {} & 2C & {} & {} & {} & {}  \\
       & {} & {} & {} & {} & -2(p-1)C & {} & {} & {}  \\
       & {} & {} & {} & {} & {} & \ddots & {} & {}  \\
       & {} & {} & {} & {} & {} & {} & -2C & {}  \\
       & {} & {} & {} & {} & {} & {} & {} & 0  
\end{matrix} \right)\frac{dw}{w},\] 
for coordinates $w$ around each puncture ${{q}_{j}}\in {{D}_{2}}$.

\subsection{Gauge-theoretic gluing of parabolic $\text{SO}\left( p,p+1 \right)$-Higgs bundles}

We have described parabolic  $\text{SO}\left( p,p+1 \right)$-models $\left( E_{i},\Phi_{i}  \right)$, $i=1,2$, which are parabolic stable. Model solutions to the Hitchin equations over punctured disks corresponding to the pairs $\left( E_{i},\Phi_{i}  \right)$ are respectively of the form
\begin{align*}
& A_{1}^{\bmod }=0,\,\,\,\Phi _{1}^{\bmod }=\left( \begin{matrix}
   2pC & {} & {} & {} & {} & {} & {} & {} & {}  \\
   {} & 2(p-1)C & {} & {} & {} & {} & {} & {} & {} \\
   {} & {} & \ddots & & {} & {} & {} & {} & {} \\
   {} & {} & {} &  2C & {} & {} & {} & {} & {}  \\
       & {} & {} & {} & -2pC & {} & {} & {} & {}  \\
       & {} & {} & {} & {} & \ddots & {} & {} & {}  \\
       & {} & {} & {} & {} & {} &  & {} & {}  \\
       & {} & {} & {} & {} & {} & {} & -2C & {}  \\
       & {} & {} & {} & {} & {} & {} & {} & 0  
\end{matrix} \right)\frac{dz}{z},\\
& A_{2}^{\bmod }=0,\,\,\Phi _{2}^{\bmod }=\left( \begin{matrix}
   0 & {} & {} & {} & {} & {} & {} & {} & {}  \\
   {} & 0 & {} & {} & {} & {} & {} & {} & {} \\
   {} & {} & 2\left( p-1 \right)C & {} & {} & {} & {} & {} & {} \\
   {} & {} & {} & \ddots  & {} & {} & {} & {} & {}  \\
       & {} & {} & {} & 2C & {} & {} & {} & {}  \\
       & {} & {} & {} & {} & -2(p-1)C & {} & {} & {}  \\
       & {} & {} & {} & {} & {} & \ddots & {} & {}  \\
       & {} & {} & {} & {} & {} & {} & -2C & {}  \\
       & {} & {} & {} & {} & {} & {} & {} & 0  
\end{matrix} \right)\frac{dw}{w}.
\end{align*}

In order to glue the above parabolic $\text{SO}\left( p,p+1 \right)$-Higgs bundles over the complex connected sum of Riemann surfaces ${{X}_{\#}}:={{X}_{1}}\#{{X}_{2}}$ of genus $g={{g}_{1}}+{{g}_{2}}+s-1$ we shall use the gauge-theoretic gluing construction summarized in \S \ref{general_gluing}. To this end, the initial model data  $\left( A_{i}^{\bmod },\Phi _{i}^{\bmod } \right)$ should be identified locally over the annuli around the points in the divisors of $s$-many points ${{D}_{i}}$, for $i=1,2$. This is achieved using the perturbation argument described next. 

Consider the embedding  
\[{{\Psi }^\textit{par}_{i}}:\mathsf{\mathcal{M}}_{K_{i}^{p}\otimes \iota _{i}^{p-1}}^\textit{par}\left( \text{SO}\left( 1,2 \right) \right)\times \underset{j=1}{\overset{p-1}{\mathop{\oplus }}}\,{{H}^{0}}\left( {{X}_{i}},K_{i}^{2j}\otimes \iota _{i}^{2j-1} \right)\to {{\mathsf{\mathcal{M}}}^\textit{par}}\left( \text{SO}\left( p,p+1 \right) \right),\] for $i=1, 2$.
Over the pair $\left( {{X}_{1}},{{D}_{1}} \right)$, take a parabolic $\text{SO}\left( 1,2 \right)$-Higgs bundle defined by the triple  $\left(\hat{W}_{1},  I_{1}, \hat{\eta }_{1} \right)$ with 
\begin{align*}
& \hat{W}_{1}:=K_{1}^{\vee }\oplus {{K}_{1}};\\
& I_{1}\cong \mathsf{\mathcal{O}};\\
& \hat{\eta }_{1}=\left( \begin{matrix}
   0 & 1 & 0  \\
   0 & 0 & 1  \\
   0 & 0 & 0  \\
\end{matrix} \right).
\end{align*}
Let  $\left( {{{\tilde{A}}}_{1}},{{{\tilde{\Phi }}}_{1}} \right)$ be the corresponding solution to the Hitchin equations. There is a complex gauge transformation which locally puts  $\left( {{{\tilde{A}}}_{1}},{{{\tilde{\Phi }}}_{1}} \right)$ into the model form 
\[\tilde{A}_{1}^{\bmod }=0,\,\,\tilde{\Phi }_{1}^{\bmod }=\left( \begin{matrix}
   2C & 0 & 0  \\
   0 & -2C & 0  \\
   0 & 0 & 0  \\
\end{matrix} \right)\frac{dz}{z}\] 
over a disk centered at the points in ${{D}_{1}}$ with coordinates $z$. 

\begin{remark}
The existence of this complex gauge transformation is provided for the local $\text{SL}\left( 2,\mathbb{R} \right)$-model solution  $\left( {{A}^{\bmod }}=0,\,\,{{\Phi }^{\bmod }}=\left( \begin{matrix}
   C & 0  \\
   0 & -C  \\
\end{matrix} \right)\frac{dz}{z} \right)$; then we embed into $\text{SO}\left( 1,2 \right)$. 
\end{remark} 
 
Using the map $\Psi _{1}^\textit{par}$ above, the parabolic stable $\text{SO}\left( p,p+1 \right)$-pair $\left( {{E}_{1}},{{\Phi }_{1}} \right)$ over $\left( {{X}_{1}},{{D}_{1}} \right)$ corresponds to an \emph{approximate solution}\index{Hitchin equation!approximate solution} $\left( {{A}_{1}},{{\Phi }_{1}} \right)$ of the Hitchin equations, which near each point of   ${{D}_{1}}$ has the form $\left( A_{1}^{\bmod },\Phi _{1}^{\bmod } \right)$ with 
\[ A_{1}^{\bmod }=0,\,\,\,\Phi _{1}^{\bmod }=\left( \begin{matrix}
   2pC & {} & {} & {} & {} & {} & {} & {} & {}  \\
   {} & 2(p-1)C & {} & {} & {} & {} & {} & {} & {} \\
   {} & {} & \ddots & & {} & {} & {} & {} & {} \\
   {} & {} & {} &  2C & {} & {} & {} & {} & {}  \\
       & {} & {} & {} & -2pC & {} & {} & {} & {}  \\
       & {} & {} & {} & {} & \ddots & {} & {} & {}  \\
       & {} & {} & {} & {} & {} &  & {} & {}  \\
       & {} & {} & {} & {} & {} & {} & -2C & {}  \\
       & {} & {} & {} & {} & {} & {} & {} & 0  
\end{matrix} \right)\frac{dz}{z},\]
for $p>2$ and $C \in \mathbb{R}$ nonzero.

Over the pair $\left( {{X}_{2}},{{D}_{2}} \right)$, take the triple  $\left(\hat{W}_{2},  I_{2}, \hat{\eta }_{2} \right)$ with 
\begin{align*}
& \hat{W}_{2}:=\tilde{M}\oplus {{\tilde{M}}^{\vee }}, \text{ where } \tilde{M}\cong \mathsf{\mathcal{O}}\left( \left( 2k-1-p \right){{D}_{2}} \right),\text{ for } k=1,\ldots ,p\\
& I_{2}\cong \mathsf{\mathcal{O}}\\
& \hat{\eta }_{2} \in {{H}^{0}}\left( \text{Hom}\left( {{{\hat{W}}}_{2}},{{I}_{2}} \right)\otimes K_{2}^{p}\otimes \iota _{2}^{p-1} \right).
\end{align*}
Applying a similar argument as above, we may perturb the relevant  $\text{SL}\left( 2,\mathbb{R} \right)$-pair and extend our data to $\text{SO}\left( p,p+1 \right)$ to finally get an \emph{approximate} solution  $\left( {{A}_{2}},{{\Phi }_{2}} \right)$, which near each point of ${{D}_{2}}$ has the form  $\left( A_{2}^{\bmod },\Phi _{2}^{\bmod } \right)$ with 
\[ A_{2}^{\bmod }=0,\,\,\,\Phi _{2}^{\bmod }=\left( \begin{matrix}
   -2pC & {} & {} & {} & {} & {} & {} & {} & {}  \\
   {} & -2(p-1)C & {} & {} & {} & {} & {} & {} & {} \\
   {} & {} & \ddots & & {} & {} & {} & {} & {} \\
   {} & {} & {} & -2C & {} & {} & {} & {} & {}  \\
       & {} & {} & {} & 2pC & {} & {} & {} & {}  \\
       & {} & {} & {} & {} & \ddots & {} & {} & {}  \\
       & {} & {} & {} & {} & {} &  & {} & {}  \\
       & {} & {} & {} & {} & {} & {} & 2C & {}  \\
       & {} & {} & {} & {} & {} & {} & {} & 0  
\end{matrix} \right)\frac{dw}{w},\]
for $p>2$ and $C \in \mathbb{R}$ nonzero, as above.

The complex connected sum of Riemann surfaces  ${{X}_{\#}}={{X}_{1}}\#{{X}_{2}}$ is realized along the curve $zw=\lambda $ for a parameter $\lambda \in \mathbb{C}$, and so $\frac{dz}{z}=-\frac{dw}{w}$ for coordinates on annuli around each puncture which are glued using a biholomorphism for each pair of points $\left( p_{i},q_{j} \right)$ from the divisors  ${{D}_{1}}$ and ${{D}_{2}}$. Let $\Omega \subset {{X}_{\#}}$ denote the result of gluing these pairs of annuli  and set $\left( A_{p_{i},q_{j}}^{\bmod },\Phi _{p_{i},q_{j}}^{\bmod } \right):=\left( A_{1}^{\bmod },\Phi _{1}^{\bmod } \right)=-\left( A_{2}^{\bmod },\Phi _{2}^{\bmod } \right)$. We can glue the pairs $\left( {{A}_{1}},{{\Phi }_{1}} \right)$, $\left( {{A}_{2}},{{\Phi }_{2}} \right)$ together to get an \emph{approximate solution} of the $\text{SO}\left( p,p+1 \right)$-Hitchin equations:

\[\left( A^\textit{app},\Phi ^\textit{app} \right):=\left\{ \begin{matrix}
   \left( {{A}_{1}},{{\Phi }_{1}} \right), & {} & \text{over }{{X}_{1}}\backslash {{X}_{2}}  \\
   \left( A_{p_{i},q_{j}}^{\bmod },\Phi _{p_{i},q_{j}}^{\bmod } \right), & {} & \text{         over } \Omega \text{ around each pair of points } \left( p_{i},q_{j} \right)  \\
   \left( {{A}_{2}},{{\Phi }_{2}} \right) & {} & \text{over }{{X}_{2}}\backslash {{X}_{1}},  \\
\end{matrix} \right.\]
over the connected sum bundle over ${{X}_{\#}}$. 

By construction, $\left( A^\textit{app},\Phi^\textit{app} \right)$ is a smooth pair on ${{X}_{\#}}$, complex gauge equivalent to an exact solution of the Hitchin equations by a smooth gauge transformation defined over all of ${{X}_{\#}}$. The next step is to correct the approximate solution $\left( A^\textit{app},\Phi^\textit{app} \right)$ to an exact solution of the $\text{SO}\left( p,p+1 \right)$-Hitchin equations. We follow the contraction mapping argument for the nonlinear $G$-Hitchin operator from \S \ref{contraction_mapping}-\ref{section_linearization} developed for a general connected semisimple Lie group $G$. We next describe how the general theory for showing that the linearization operator is invertible adapts to the case when $G = \text{SO}\left( p,p+1 \right)$; the computation of the necessary analytic estimates for an approximate solution does not depend on the the semisimple Lie group $G$ and can be found in \S 6 of \cite{Kydon-article}.

For the group $G = \text{SO}(p, p+1)$, a maximal compact subgroup is $H = \text{SO}(p, \mathbb{C}) \times \text{SO}(p+1, \mathbb{C})$ with Lie algebra $\mathfrak{h} = \mathfrak{so}(p) \times \mathfrak{so}(p+1)$. Moreover, for a Higgs field $\Phi = \varphi dz$, the compact real form $\tau : \mathfrak{g}^{\mathbb{C}} \to \mathfrak{g}^{\mathbb{C}}$ is giving $\tau (\Phi) = \bar{\varphi} dz$. For the notation introduced in \S \ref{section_linearization}, we have the following:

\begin{lemma}\label{lemma on D}
 For an element $\left( {{\psi }_{1}},{{\psi }_{2}} \right)\in \ker \left( {{L}_{1}}+L_{2}^{*} \right)\cap L_{\text{ext}}^{2}\left( X_{\#}^{\times } \right)$, we have 
\[{{d}_{A_{0}^\textit{app}}}{{\psi }_{i}}=\left[ {{\psi }_{i}},\Phi _{0}^\textit{app} \right]=\left[ {{\psi }_{i}},{{\left( \Phi _{0}^\textit{app} \right)}^{*}} \right]=0,\] 
for $i=1,2$.
\end{lemma}
\begin{proof}
The proof follows exactly the same steps as that of Lemma 6.5 in \cite{Kydon-article}. There are no nontrivial off-diagonal elements in the kernel of the operator 
\[D\left( {{\psi }_{1}},{{\psi }_{2}} \right):=2\left( \begin{matrix}
   {\frac{i}{2}{{\partial }_{\theta }}{{\psi }_{1}}+\left[ {{\psi }_{2}},\tau \left( \varphi  \right) \right]}  \\
   {-\frac{i}{2}{{\partial }_{\theta }}{{\psi }_{2}}-\left[ {{\psi }_{1}},\varphi  \right]}  \\
\end{matrix} \right),
\]
since we have taken a diagonal and traceless model Higgs field  $\Phi^{\bmod }=\varphi^{\bmod }\frac{dz}{z}$, where
\[\varphi^{\bmod }=\left( \begin{matrix}
   2pC & {} & {} & {} & {} & {} & {} & {} & {}  \\
   {} & 2(p-1)C & {} & {} & {} & {} & {} & {} & {} \\
   {} & {} & \ddots & & {} & {} & {} & {} & {} \\
   {} & {} & {} &  2C & {} & {} & {} & {} & {}  \\
       & {} & {} & {} & -2pC & {} & {} & {} & {}  \\
       & {} & {} & {} & {} & \ddots & {} & {} & {}  \\
       & {} & {} & {} & {} & {} &  & {} & {}  \\
       & {} & {} & {} & {} & {} & {} & -2C & {}  \\
       & {} & {} & {} & {} & {} & {} & {} & 0  
\end{matrix} \right),\]
with $p>2$ and $C \in \mathbb{R}$ nonzero.
\end{proof}
The next proposition is now providing invertibility:
\begin{proposition}
The operator ${{L}_{1}}+L_{2}^{*}$ considered as a densely defined operator on $L_{\text{ext}}^{2}\left( X_{\#}^{\times } \right)$ has trivial kernel.
\end{proposition}

\begin{proof}
Let $\left( {{\psi }_{1}},{{\psi }_{2}} \right)\in \ker \left( {{L}_{1}}+L_{2}^{*} \right)\cap L_{\text{ext}}^{2}\left( X_{\#}^{\times } \right)$. From Lemma \ref{lemma on D} we have
\[{{d}_{A_{0}^\textit{app}}}{{\psi }_{i}}=\left[ {{\psi }_{i}},\Phi _{0}^\textit{app} \right]=\left[ {{\psi }_{i}},{{\left( \Phi _{0}^\textit{app} \right)}^{*}} \right]=0,\] for $i=1,2$. We show that ${{\psi }_{1}}=0$ by showing separately that $\gamma :={{\psi }_{1}}+\psi _{1}^{*}\in {{\Omega }^{0}}\left( X_{\#}^{\times }, E \left( \mathfrak{h} \right) \right)$ and $\delta :=i\left( {{\psi }_{1}}-\psi _{1}^{*} \right)\in {{\Omega }^{0}}\left( X_{\#}^{\times }, E \left( \mathfrak{h} \right) \right)$ both vanish.\\
For holomorphic coordinate $z$ centered at the node of $X_{\#}^{\times }$, the Higgs field $\Phi _{0}^\textit{app}$ in our exact solution is written $\Phi _{0}^\textit{app}=\varphi \frac{dz}{z}$
with 
$$\varphi \in {{\mathfrak{m}}^{\mathbb{C}}}\left( \text{SO}\left( p,p+1 \right) \right)=\left\{ \left( \begin{matrix}
   0 & Q  \\
   -{{Q}^{T}} & 0  \\
\end{matrix} \right), \text{ for } M \text{ a } p\times \left( p+1 \right)\text{ complex matrix} \right\}.$$
We get that $d{{\left| \gamma  \right|}^{2}}=2\left\langle {{d}_{A_{0}^{app}}}\gamma ,\gamma  \right\rangle =0$, that is, $\left| \gamma  \right|$ is constant on $X_{\#}^{\times }$, as well as that $\gamma( x)$ lies in the kernel of the linearization operator (see Lemma \ref{linearization_lemma}). Moreover, 
\[\gamma \left( x \right)\in \mathfrak{h}=\left\{ \left( \begin{matrix}
   M & 0  \\
   0 & N  \\
\end{matrix} \right)\left| M\in \mathfrak{so}\left( p \right),N\in \mathfrak{so}\left( p+1 \right) \right. \right\},\] 
thus $\gamma(x)$ has orthogonal eigenvectors for distinct eigenvalues, but even if there are degenerate eigenvalues, it is still possible to find an orthonormal basis consisting of eigenvectors of $\gamma \left( x \right)$.\\
Now, if $\gamma \left( x \right)$ is non-zero, since $\left[ \varphi \left( x \right),\gamma \left( x \right) \right]=0$ it follows that $\varphi \left( x \right)$ preserves the eigenspaces of $\gamma \left( x \right)$ for all $x\in X_{\#}^{\times }$ and so $\left\langle \varphi \left( x \right)v,\varphi \left( x \right)w \right\rangle =\left\langle v,w \right\rangle $ for $v,w\in {{\mathbb{C}}^{2p+1}}$. In other words, $\varphi \left( x \right)$ ought to be an isometry with respect to the usual norm in ${{\mathbb{C}}^{2p+1}}$. Equivalently, $\varphi \left( x \right)$ is unitary for all $x\in X_{\#}^{\times }$. The determinant of the Higgs field $\det \Phi_{0}^\textit{app}$ generically has a simple zero in at least one point in $X_{\#}^{\times }$. For a zero $x_{0}$ chosen, say, on the left hand side surface $X_1$ of $X_{\#}^{\times }$ where we embed via the irreducible representation described in \S \ref{models_via_irreducible_repr} --- let us denote it at present $\phi_\textit{irr}$ --- we see that
	\[\varphi \left( {{x}_{0}} \right)={{\phi }_{\textit{irr}^{*}}}\left( \begin{matrix}
   0 & z  \\
   1 & 0  \\
\end{matrix} \right) = \tilde{e} +ze \]
which is not unitary, for the matrices $e$ and $\tilde{e}$ as in (\ref{regular nilpotent}) and (\ref{e-tilda}) respectively. This is a contradiction and therefore, $\gamma=0$ everywhere.\\
That $\delta$ vanishes, as well as that ${{\psi }_{2}}=0$, is proven entirely similarly.
\end{proof}

\begin{remark}
The assumption of the existence of at least one simple zero of a generic meromorphic quadratic differential allows us to show that the linear operator $L_{(A, \Phi)}$ is injective and thus assure absence of small eigenvalues of this linear operator governing the gluing construction (cf. Swoboda \cite{Swoboda} for a similar application). That a generic solution of the rank 2 Hitchin equations has only simple zeroes is proven in \cite{MSWW}.
\end{remark}
Theorem \ref{main_theorem} adapts in the case $G=\text{SO}\left( p,p+1 \right)$ to provide the following: 

\begin{theorem}
Let $X_{1}$ be a closed Riemann surface of genus $g_{1}$ and ${{D}_{1}}=\left\{ {{p}_{1}},\ldots ,{{p}_{s}} \right\}$ a collection of $s$ distinct points on $X_{1}$. Consider respectively a closed Riemann surface $X_{2}$ of genus $g_{2}$ and a collection of also $s$ distinct points ${{D}_{2}}=\left\{ {{q}_{1}},\ldots ,{{q}_{s}} \right\}$ on $X_{2}$. Let $\left( {{E}_{1}},{{\Phi }_{1}} \right)\to {{X}_{1}}$ and $\left( {{E}_{2}},{{\Phi }_{2}} \right)\to {{X}_{2}}$ be parabolic polystable $\rm{SO}\left( p,p+1 \right)$-Higgs bundles, one from each of the families described in \S \ref{first_models} and \S \ref{second_models} with corresponding solutions to the Hitchin equations $\left( {{A}_{1}},{{\Phi }_{1}} \right)$ and $\left( {{A}_{2}},{{\Phi }_{2}} \right)$. Then there is a polystable $\rm{SO}\left( p,p+1 \right)$-Higgs bundle $\left( E_{\#},\Phi_{\#}  \right)\to {{X}_{\#}}$ over the complex connected sum of Riemann surfaces ${{X}_{\#}}={{X}_{1}}\#{{X}_{2}}$, which agrees with the initial data over ${{X}_{\#}}\backslash {{X}_{1}}$ and ${{X}_{\#}}\backslash {{X}_{2}}$.
\end{theorem}

\begin{definition}
We call such an $\rm{SO}( p,p+1)$-Higgs bundle constructed by the theorem above a \emph{hybrid $\rm{SO}( p,p+1)$-Higgs bundle}\index{hybrid Higgs bundle}\index{Higgs bundle!hybrid}.
\end{definition}

\subsection{Model representations in the exceptional components of $\mathsf{\mathcal{R}}\left( \text{SO}\left( p,p+1 \right) \right)$}

We now show that the specific  hybrid  $\text{SO}\left( p,p+1 \right)$-Higgs bundles constructed in the previous section lie inside the $p\left( 2g-2 \right)-1$ exceptional components of the character variety $\mathsf{\mathcal{R}}\left( \text{SO}\left( p,p+1 \right) \right)$. In fact, by varying the parameters in the construction, namely, the genera $g_{1}$, $g_{2}$ of the Riemann surfaces $X_{1}$, $X_{2}$, the number of points $s$ in the divisors $D_{1}$, $D_{2}$, and the weight $\alpha=2k-1-p$ for the line bundle $\tilde{M}\cong \mathsf{\mathcal{O}}\left( \left( 2k-1-p \right){{D}_{2}} \right)$, one obtains models in \emph{all exceptional components}\index{exceptional component}. This is seen by an explicit computation of the degree of the line bundle $M$ appearing in the description (\ref{203}) of the Higgs bundle data; the exceptional components are fully distinguished by the degree of this line bundle. We have considered: 
\begin{align*}
{{E}_{1}} & =K_{1}^{-p}\otimes \mathsf{\mathcal{O}}\left( -p{{D}_{1}} \right)\oplus K_{1}^{-\left( p-1 \right)}\otimes \mathsf{\mathcal{O}}\left( \left( 1-p \right){{D}_{1}} \right)\oplus \ldots \\
& \ldots \oplus K_{1}^{\left( p-1 \right)}\otimes \mathsf{\mathcal{O}}\left( \left( p-1 \right){{D}_{1}} \right) \oplus K_{1}^{p}\otimes \mathsf{\mathcal{O}}\left( p{{D}_{1}} \right), \text{ and} 
\end{align*}
\begin{align*}
{{E}_{2}} & =V\oplus W=\tilde{M}^{\vee }\oplus {{\tilde{M}}}\oplus K_{2}^{-\left( p-1 \right)}\otimes \mathsf{\mathcal{O}}\left( \left( 1-p \right){{D}_{2}} \right)\oplus \ldots \\
& \ldots \oplus K_{2}^{p-1}\otimes \mathsf{\mathcal{O}}\left( \left( p-1 \right){{D}_{2}} \right),
\end{align*}
with $\tilde{M}\cong \mathsf{\mathcal{O}}\left( \left( 2k-1-p \right){{D}_{2}} \right)$ and  $\text{pardeg} \left( \tilde{M} \right)=\left( 2k-1-p \right)s$, for $k=1,\ldots ,p$. We now use Proposition \ref{additivity_property}, which asserts an additivity property for the parabolic degree of the bundle over the connected sum operation. We thus have that for each $j\in \left\{ 1-p,\ldots ,p-1 \right\}$ the bundle $K_{1}^{{{\otimes }_\textit{par}}j}\#K_{2}^{{{\otimes }_\textit{par}}-j}$ has degree
\begin{align*}
\deg\left( K_{1}^{{{\otimes }_\textit{par}}j}\#K_{2}^{{{\otimes }_\textit{par}}j} \right) & = \text{pardeg} \left( K_{1}^{j}\otimes \mathsf{\mathcal{O}}\left( j{{D}_{1}} \right) \right) + \text{pardeg} \left(K_{2}^{j}\otimes \mathsf{\mathcal{O}}\left( j{{D}_{2}} \right) \right)\\
& = j\left( 2{{g}_{1}}-2+s \right)+j\left( 2{{g}_{2}}-2+s \right)\\
& =2j\left( {{g}_{1}}+{{g}_{2}}+s-1-1 \right)\\
& =2j\left( {{g}_{{{X}_{\#}}}}-1 \right)\\
& = \deg K_{{{X}_{\#}}}^{\otimes j}.
\end{align*}
It is thus a line bundle isomorphic to $K_{{{X}_{\#}}}^{\otimes j}$. \\
Moreover, gluing the parabolic line bundles $K_{1}^{p}\otimes \mathsf{\mathcal{O}}\left( p{{D}_{1}} \right)$ and $\tilde{M}$  provides a line bundle $M\in \text{Pic}\left( {{X}_{\#}} \right)$ with degree
\begin{align*}
\deg \left( M \right) & = \text{pardeg} \left( K_{1}^{p}\otimes \mathsf{\mathcal{O}}\left( p{{D}_{1}} \right) \right)+ \text{pardeg} \left( \tilde{M} \right)\\ 
& = p\left( 2{{g}_{1}}-2+s \right)+ (2k-1-p) s\\
& = 2p\left( {{g}_{1}}-1 \right)+\left( 2k-1 \right)s.
\end{align*}

We deduce that the result of the construction is a Higgs bundle $\left( V,{{W}_{k}},\eta  \right)$ with data $V$ and $\eta$ as in  (\ref{203}) and 
\[{{W}_{k}}:=M\oplus K_{{{X}_{\#}}}^{p-2}\oplus \ldots \oplus K_{{{X}_{\#}}}^{2-p}\oplus {{M}^{-1}}\] 
with $d=\deg \left( M \right)=2p\left( {{g}_{1}}-1 \right)+\left( 2k-1 \right)s$, for  $k=1,\ldots ,p$. One can now check that varying the values of the parameters $g_{1}$, $s$ and $k$, we can obtain model $\text{SO}\left( p,p+1 \right)$-Higgs bundles by gluing, which exhaust all the exceptional smooth $p\left( 2g-2 \right)-1$ components of $\mathsf{\mathcal{M}}\left( \text{SO}\left( p,p+1 \right) \right)$.

\begin{remark}
Notice that the case when $p=1$ actually describes the $\text{Sp}\left( 4,\mathbb{R} \right)$-case from \cite{Kydon-article}. Indeed, we then have $k=1$ and so $\tilde{M}\cong \mathsf{\mathcal{O}}$ with  $d=\deg \left( M \right)=2\left( {{g}_{1}}-1 \right)+s=-\chi \left( {{\Sigma }_{l}} \right)$. The case $p>2$ thus involves an \emph{extra parameter} on the non-trivial line bundle $\tilde{M}$ given by the parabolic structure on a trivial flag.  
\end{remark}

\subsection{Model representations and positivity}\label{gluing_positivity}

The model $\text{SO}\left( p,p+1 \right)$-Higgs bundles obtained above are now all $\Theta$-positive. This follows directly from the recent work of Beyrer and Pozzetti \cite{BePo2}, who showed that the set of $\Theta$-positive representations is closed in the character variety $\mathcal{R}(\text{SO}(p,q))$, for $p \le q$. Moreover, Theorem \ref{GLW_conjecture} asserts that the connected components parameterized by using Higgs bundle methods in \cite{BCGGP} consist solely of $\Theta$-positive representations; the exceptional components of Definition \ref{exceptional_comp_def} do, indeed, fall in these cases (see \cite{Aparicio et al}). 

A more direct way to show that the models in the exceptional $p\left( 2g-2 \right)-1$ smooth components of  $\mathsf{\mathcal{R}}( \text{SO}( p,p+1 ))$ are $\Theta$-positive, is by gluing the positivity condition at the level of infinity of the fundamental group. In fact, a Hitchin representation into $\text{SO}\left( p,p+1 \right)$, and a representation which factors through $\text{SO}\left( p-1,p \right)\times \text{SO}\left( 2 \right)$ with  $\text{SO}\left( p-1,p \right)$-factor in the relative Hitchin component, that is, like the ones we chose, are both $\Theta$-positive (see \cite{Aparicio et al}, \cite{Collier}). On the other hand, in \cite{FG}, pp. 95-100, Fock and Goncharov provide a gluing method for positive local systems on a pair of Riemann surfaces with boundary for the case of split real Lie groups, and so for the group $\text{SO}\left( p,p+1 \right)$ in particular. This involves the requirement that the monodromies along the two boundary components, as well as the assigned configurations of positive flags coincide; see p. 99 of (loc. cit.) for more details.

\vspace{2mm}
\textbf{Acknowledgments}. I would like to warmly thank the editors Ken'ichi Ohshika and Athanase Papadopoulos for their kind invitation to contribute to this collective volume. The work included in \S \ref{Examples-orthogonal} was supported by the Labex IRMIA of the Universit\'{e} de Strasbourg and was undertaken in collaboration with Olivier Guichard. I am also grateful to an anonymous referee as well as the editors for their careful reading of the manuscript and a number of useful suggestions.


\vspace{2mm}




\begin{thebibliography}{99.}
\bibitem{AGRW}
D. Alessandrini, O. Guichard, E. Rogozinnikov, A. Wienhard, Noncommutative coordinates for symplectic representations. arXiv: 1911.08014 (2019)
 
\bibitem{Apanasov}
B. N. Apanasov, Non-triviality of Teichm\"{u}ller space for Kleinian groups in space. in \textit{Riemann surfaces and related topics}, Proceedings of the 1978 Stony Brook Conference (Princeton University Press, 1980), pp. 21--31 

\bibitem{AT}
B. N. Apanasov, A. V. Tetenov, On the existence of non-trivial quasi-conformal deformations of Kleinian groups in space. Soviet Math. Dokl. \textbf{19}, 242--245 (1978)

\bibitem{Aparicio}
M. Aparicio Arroyo, The geometry of $\text{SO}(p,q)$-Higgs bundles, Ph.D. thesis, Universidad de Salamanca, Consejo Superior de Investigaciones Cient\'{i}ficas (2009)

\bibitem{Aparicio et al}
M. Aparicio-Arroyo, S. Bradlow, B. Collier, O. Garc\'{i}a-Prada, P. B. Gothen, A. Oliveira, $\text{SO}\left( p,q \right)$-Higgs bundles and higher Teichm\"{u}ller components. Invent. Math. \textbf{218}, no. 1, 197--299 (2019) 

\bibitem{BeGa} L. Bers, F. Gadiner, Fricke spaces. Adv. in Math. \textbf{62}, no. 3, 249--284  (1986)

\bibitem{Beetal} L. Bessi\`{e}res, G. Besson, M. Boileau, S. Maillot, J. Porti, Geometrization of 3-manifolds.  \textit{EMS Tracts in Mathematics 13}, (European Mathematical Society, Z\"{u}rich 2010),  x+237 pp. 

\bibitem{BePo} J. Beyrer, B. Pozzetti, A collar lemma for partially hyperconvex surface group representations. Trans. Amer. Math. Soc. \textbf{374}, no. 10, 6927--6961 (2021)

\bibitem{BePo2} J. Beyrer, B. Pozzetti, Positive surface group representations in $\text{PO}(p,q)$. arXiv: 2106.14725 (2021)

\bibitem{BiBo} O. Biquard, P. Boalch, Wild non-abelian Hodge theory on curves. Compos. Math. \textbf{140}, no. 1, 179--204  (2004)

\bibitem{BGM}
O. Biquard, O.  Garc\'{i}a-Prada, I. Mundet i Riera, Parabolic Higgs bundles and representations of the fundamental group of a punctured surface into a real group. Adv. Math. \textbf{372}, 107305  (2020)

\bibitem{BAG} I. Biswas, P. Ar\'{e}s-Gastesi, S. Govindarajan, Parabolic Higgs bundles and Teichm\"{u}ller spaces for punctured surfaces. Trans. Amer. Math. Soc. \textbf{349}, no. 4, 1551--1560  (1997)

\bibitem{BoPo} M. Boileau, J. Porti, Geometrization of 3-orbifolds of cyclic type. Ast\'{e}risque No. 272, 208 pp.  (2001)

\bibitem{Boyer} S. Boyer, Dehn surgery on knots, in \textit{Handbook of Geometric Topology}, ed. by R. J. Daverman, R. B. Sher (North-Holland, 2002), pp. 165--218

\bibitem{BCGGP} S. Bradlow, B. Collier, O. Garc\'{i}a-Prada, P. Gothen, A. Oliveira, A general Cayley correspondence and higher Teichm\"{u}ller spaces, arXiv: 2101.09377 (2021)

\bibitem{Br1} M. Bridgeman, Average bending of convex pleated planes in hyperbolic three-space. Invent. Math. \textbf{132}, 381--391 (1998)

\bibitem{Br2} M. Bridgeman, Bounds on the average bending of the convex hull of a Kleinian group. Mich. Math. J. \textbf{51}, 363--378 (2003)

\bibitem{BCLS} M. Bridgeman, R. Canary, F. Labourie, A. Sambarino, The pressure metric for Anosov representations. Geom. Funct. Anal. \textbf{25}, no. 4, 1089--1179  (2015)

\bibitem{BCS} M. Bridgeman, R. Canary, A. Sambarino, An introduction to pressure metrics for higher Teichm\"{u}ller spaces. Ergodic Theory Dyn. Syst. \textbf{38}, no. 6, 2001--2035 (2018)

\bibitem{BILW} M. Burger, A. Iozzi, F. Labourie, A. Wienhard, Maximal representations of surface groups: Symplectic Anosov structures. Pure Appl. Math. Q. \textbf{1}, no. 3, Special issue In memory of Armand Borel. Part 2, 543--590 (2005)

\bibitem{BIW} M. Burger, A. Iozzi, A. Wienhard, Surface group representations with maximal Toledo invariant. Ann. of Math. (2) \textbf{172}, no. 1, 517--566 (2010)

\bibitem{BuPo} M. Burger, M. B. Pozzetti, Maximal representations, non-Archimedean Siegel spaces, and buildings. Geom. Topol. \textbf{21}, no. 6, 3539--3599 (2019)

\bibitem{CLM}
S. Cappell, R. Lee, E. Miller, Self-adjoint elliptic operators and manifold decompositions. Part I: Low eigenmodes and stretching. Comm. Pure Appl. Math. \textbf{49}, 825--866  (1996)

\bibitem{Collier} B. Collier, $\text{SO}\left( n,n+1 \right)$-surface group representations and their Higgs bundles. Ann. Sci. \'{E}c. Norm. Sup\'{e}r. (4) \textbf{63} (6), 1561--1616 (2020)

\bibitem{Corlette} K. Corlette, Flat $G$-bundles with canonical metrics. J. Diff. Geom. \textbf{28}, 361--382 (1988)

\bibitem{SGe} H. P. de Saint-Gervais, Uniformisation des surfaces de Riemann, \textit{ENS Editions} (Lyon, 2010), 544 pp. 

\bibitem{Dehn} M. Dehn, \"{U}ber die Topologie des dreidimensionales Raumes. Math. Ann. \textbf{69}, 137--168 (1910)

\bibitem{Donaldson}
S. K. Donaldson, Twisted harmonic maps and the self-duality equations. Proc. London Math. Soc. (3) \textbf{55}, 127--131 (1987)

\bibitem{DonKron}
S. Donaldson, P. Kronheimer, The geometry of four-manifolds. \emph{Oxford Math. Monographs}, Oxford Science Publications 1990

\bibitem{DuMe} W. D. Dunbar, R. G. Meyerhoff, Volumes of hyperbolic 3-orbifolds. Indiana Univ. Math. J. \textbf{43}, 611--637  (1994)

\bibitem{EM}
D. B. A. Epstein, A. Marden, Convex hulls in hyperbolic space, a theorem of Sullivan, and measured pleated surfaces. in D. B. A. Epstein (ed.), \textit{Analytical and geometric aspects of hyperbolic space}, Warwick and Durham 1984, London Mathematical Society Lecture Note Series 111 (Cambridge University Press 1987), pp. 113--253

\bibitem{FP} F. Fanoni, B. Pozzetti, Basmajian-type inequalities for maximal representations. J. Differential Geom. \textbf{116}, , 405--458 (2020)

\bibitem{FG} V. V. Fock, A. Goncharov, Moduli spaces of local systems and higher Teichm\"{u}ller theory. Publ. Math. Inst. Hautes \'Etudes Sci. \textbf{103}, 1--211 (2006)

\bibitem{Foscolo} L. Foscolo, A gluing construction for periodic monopoles. Int. Math. Res. Not. IMRN, no. 24, 7504--7550
 (2017)

\bibitem{FrKl} R. Fricke, F. Klein, Vorlesungen \"{u}ber die Theorie der automorphen Funktionen, Vols. 1,2, (Teubner, Stuttgart 1986, 1912)

\bibitem{GPsurvey} O. Garc{\'i}a-Prada, Higgs bundles and higher Teichm\"{u}ller spaces, in \textit{Handbook of Teichm\'{u}ller Theory. Vol. VII}, ed. by A. Papadopoulos (IRMA Lectures in Mathematics and Theoretical Physics vol. 30, 2020), pp. 239--285

\bibitem{GGMHitchin-Kob}
O. Garc{\'i}a-Prada, P. B. Gothen, I. Mundet i Riera, The Hitchin-Kobayashi correspondence, Higgs pairs and surface group representations. arXiv:0909.4487 (2009)

\bibitem{Goldman4} W. M. Goldman, The symplectic nature of fundamental groups of surfaces. Adv. Math. \textbf{54}, no. 2, 200--225 (1984)

\bibitem{Goldman3} W. M. Goldman, Topological components of spaces of representations. Invent. Math. \textbf{93}, no. 3, 557--607 (1988)


\bibitem{Gordon2} C. McA. Gordon, Dehn filling: A survey. \textit{Proceedings of the Mini Semester in Knot Theory, Banach Center, Warsaw 1995}. Banach Center Publications, vol. 42 (Institute of Mathematics, Polish Academy of Sciences, Warszawa, 1998), pp. 129--144 

\bibitem{Gordon1} C. McA. Gordon, Dehn surgery on knots. \textit{Proceedings of the International Congress of Mathematicians, Kyoto 1990}, (Springer, Tokyo 1991), pp. 631--642

\bibitem{Gothen} P. B. Gothen, Components of spaces of representations and stable triples. Topology \textbf{40}, no. 4, 823--850 (2001)

\bibitem{Guichard} O. Guichard, Composantes de Hitchin et r\'{e}presentations hyperconvexes de groupes de surface. J. Differential Geom. \textbf{80}, no. 3, 391--431 (2008)

\bibitem{GLW} O. Guichard, F. Labourie and A. Wienhard, Positivity and representations of surface groups. arXiv: 2106.14584 (2021)

\bibitem{GW2} O. Guichard, A. Wienhard, Anosov representations: domains of discontinuity and applications. Invent. Math.  \textbf{190}, no. 2, 357--438 (2012)

\bibitem{GW} O. Guichard, A. Wienhard, Positivity and higher Teichm\"{u}ller theory. Proceedings of the 7th European Congress of Mathematics (2016)

\bibitem{GW3} O. Guichard and A. Wienhard, Topological invariants of Anosov representations. J. Topol. \textbf{3}, no. 3, 578--642  (2010)

\bibitem{Haken1} W. Haken, Theorie der Normalfl\"{a}chen. Ein Isotopiekriterium f\'{u}r den Kreisknoten. Acta Math. \textbf{105}, 245--375 (1961)

\bibitem{Haken2} W. Haken, \"{U}ber das Hom\"{o}omorphieproblem der 3-Mannigfaltigkeiten. I. Math. Z. \textbf{80}, 89--120
 (1962)

\bibitem{He} S. He, A gluing theorem for the Kapustin-Witten equations with a Nahm pole. J. Topol. \textbf{12}, no. 3, 855--915 (2019)

\bibitem{Hit92} N. J. Hitchin, Lie groups and Teichm\"{u}ller space. Topology \textbf{31}, no. 3, 449--473 (1992)

\bibitem{Hit87}
N. J. Hitchin, The self-duality equations on a Riemann surface. Proc. London Math. Soc. (3) \textbf{55}, 59--126 (1987)

\bibitem{HoKe} C. D. Hodgson, S. P. Kerckhoff, Universal bounds for hyperbolic Dehn surgery. Ann. of Math. (2) \textbf{162}, no. 1, 367--421 (2005)

\bibitem{HuSu} Y. Huang, Z. Sun, McShane identities for higher Teichm\"{u}ller theory and the Goncharov-Shen potential. arXiv: 1901.02032 (2019)

\bibitem{Hubbard} J. H. Hubbard, Teichm\"{u}ller theory and applications to geometry, topology, and dynamics, vol. 1, \textit{Matrix Editions}, (Ithaca, NY 2006)

\bibitem{JaOe} W. Jaco, U. Oertel, An algorithm to decide if a 3-manifold is a Haken manifold. Topology \textbf{23}, no. 2, 195--209 (1984)

\bibitem{JM}
D. Johnson, J. J. Millson, Deformation spaces associated to compact hyperbolic manifolds. in \textit{Discrete groups in geometry and analysis}, (New Haven, CT 1984), Progr. Math. \textbf{67}, Birkh\"{a}user Boston (Boston, MA 1987), pp. 48--106

\bibitem{Jost} J. Jost, \textit{Compact Riemann surfaces}. An Introduction to contemporary mathematics, second ed. (Springer-Verlag, Berlin 2002), xvi+278 pp.

\bibitem{Kirby} R. Kirby, A calculus for framed links. Invent. Math. \textbf{45}, 35--56 (1978)

\bibitem{Konno} H. Konno, Construction of the moduli space of stable parabolic Higgs bundles on a Riemann surface. J. Math. Soc. Japan \textbf{45}, no. 2, 253--276 (1993)

\bibitem{Kos} B. Kostant, The principal three-dimensional subgroup and the Betti numbers of a complex simple Lie group. Amer. J. Math. \textbf{81}, 973--1032 (1959)

\bibitem{Kourouniotis1}
C. Kourouniotis, Deformations of hyperbolic structures. Math. Proc. Cambridge Philos. Soc. \textbf{98}, no. 2, 247--261 (1985)

\bibitem{Kourouniotis2}
C. Kourouniotis, Bending in the space of quasi-Fuchsian structures. Glasgow Math. J. \textbf{33}, no. 1, 41--49 (1991)

\bibitem{Kourouniotis3}
C. Kourouniotis, The geometry of bending quasi-Fuchsian groups. in \textit{Discrete groups and geometry}, (Birmingham, 1991), pp. 148--164, London Math. Soc. Lecture Note Ser. 173 (Cambridge University Press, Cambridge 1992)

\bibitem{Kydonakis} G. Kydonakis, Gluing constructions for Higgs bundles over a complex connected sum. Ph.D. thesis, University of Illinois at Urbana-Champaign (2018)

\bibitem{Kydon-article} G. Kydonakis, Model Higgs bundles in exceptional components of the $\text{Sp(4}\text{,}\mathbb{R}\text{)}$-character variety. Internat. J. Math.  \textbf{32}, no. 9, Paper No. 2150067, 50 pp. (2021)

\bibitem{KSZ} G. Kydonakis, H. Sun, L. Zhao, Topological invariants of parabolic $G$-Higgs bundles. Math. Z. \textbf{297}, no. 1-2, 585--632 (2021)

\bibitem{Labourie} F. Labourie, Anosov flows, surface groups and curves in projective space. Invent. Math. \textbf{165}, no. 1, 51--114 (2006)

\bibitem{Labourie2} F. Labourie, Cross ratios, surface groups, $\text{PSL} \left(n, \mathbb{R} \right)$ and diffeomorphisms of the circle. Publ. Math. Inst. Hautes \'{E}tudes Sci. \textbf{106}, 139--213 (2007)

\bibitem{LaMc} F. Labourie, G. McShane, Cross ratios and identities for higher Teichm\"{u}ller-Thurston theory. Duke Math. J.  \textbf{149}, no. 2, 279--345 (2009)


\bibitem{LaMe} M. Lackenby, R. Meyerhoff, The maximal number of exceptional Dehn surgeries. Invent Math. \textbf{191}, no. 2, 341--382 (2013)

\bibitem{Le} I. Le, Higher laminations and affine building. Geom. Topol. \textbf{20}, no. 3, 1673--1735 (2016)

\bibitem{LeeZh} G.-S. Lee, T. Zhang, Collar lemma for Hitchin representations. Geom. Topol. \textbf{21}, no. 4, 2243--2280 (2017)

\bibitem{Lickorish} W. B. R. Lickorish, A representation of orientable combinatorial 3-manifolds. Ann. of Math. \textbf{76}, 531--540 (1962)

\bibitem{Luecke} J. Luecke, Dehn surgery on knots in the 3-sphere.  \textit{Proceedings of the International Congress of Mathematicians, Z\"{u}rich 1994}, (Birkh\"{a}user 1995), pp. 585--594

\bibitem{Lu}
G. Lusztig, Total positivity in reductive groups. \textit{Lie theory and geometry}, Progr. Math. vol. 123, Birkh\"{a}user Boston, Boston, MA 1994, pp. 531--568

\bibitem{MSWW}
R. Mazzeo, J. Swoboda, H. Weiss, F. Witt, Ends of the moduli space of Higgs bundles. Duke Math. J. \textbf{165}, no. 12, 2227--2271 (2016)

\bibitem{MeSe}
V. B. Mehta, C. S. Seshadri, Moduli of vector bundles on curves with parabolic structures. Math. Ann. \textbf{248}, no. 3, 205--239 (1980)

\bibitem{Millson}
J. J. Millson, On the first Betti number of a constant negatively curved manifold. Ann. of Math. \textbf{104}, 235--247 (1976)

\bibitem{Mondello} G. Mondello, Topology of representation spaces of surface groups in $\text{PSL} \left(2, \mathbb{R} \right)$ with assigned boundary monodromy and nonzero Euler number. Pure Appl. Math. Q. \textbf{12}, no. 3, 399--462 (2016)

\bibitem{Morgan} J. W. Morgan, On Thurston's uniformization theorem for three-dimensional manifolds. in J. W. Morgan, H. Bass (eds.) \textit{The Smith conjecture} (New York, 1979), Pure Appl. Math. \textbf{112} (Academic Press, Orlando, FL, 1984), pp. 37--125 

\bibitem{MoTi} J. Morgan, G. Tian, The geometrization conjecture. \textit{Clay Mathematics Monographs 5}. American Mathematical Society, Providence, RI (Clay Mathematics Institute, Cambridge, MA, 2014), x+291 pp

\bibitem{Nicolaescuarticle}
L. Nicolaescu, On the Cappell-Lee-Miller gluing theorem. Pac. J. Math. \textbf{206}, no. 1, 159--185 (2002)

\bibitem{Perelman} G. Perelman, Ricci flow with surgery on three-manifolds. arXiv: math/0303109 (2003)

\bibitem{PoSa} R. Potrie, A. Sambarino, Eigenvalues and entropy of a Hitchin representation. Invent. Math. \textbf{209}, no. 3, 885--925 (2017)

\bibitem{Pozzetti} M. B. Pozzetti, Higher rank Teichm\"{u}ller theories. Ast\'{e}risque \textbf{422}, 327--354 (2020)

\bibitem{Raghunathan} M. S. Raghunathan, Discrete subgroups of Lie groups, (Springer-Verlag, New York-Heidelberg 1972), Ergebnisse der Mathematik und ihrer Grenzgebiete, Band 68

\bibitem{Ratcliffe} J. G. Ratcliffe, Foundations of hyperbolic manifolds, second ed., \textit{Graduate Texts in Mathematics}, vol. 149, Springer, New York 2006

\bibitem{Reid} A. W. Reid, A non-Haken hyperbolic 3-manifold covered by a surface bundle. Pacific J. Math. \textbf{167}, no. 1, 163--182 (1995)

\bibitem{Rolfsen} D. Rolfsen, Rational surgery calculus. Extension of Kirby's theorem. Pacific J. Math. \textbf{110}, 377--386 (1984)

\bibitem{Safari}
P. Safari, A gluing theorem for Seiberg-Witten moduli spaces. Ph. D. Thesis, Columbia University (2000)

\bibitem{Schmitt1}
A. H. W. Schmitt, Geometric invariant theory and decorated principal bundles. \emph{Z{\"u}rich Lectures in Advanced Mathematics}, European Mathematical Society 2008

\bibitem{Schmitt2}
A. H. W. Schmitt, Moduli for decorated tuples of sheaves and representation spaces for quivers. Proc. Indian Acad. Sci. Math. Sci. \textbf{115}, no. 1, 15--49 (2005)

\bibitem{Seshadri}
C. S. Seshadri, Moduli of vector bundles on curves with parabolic structures. Bul. of the AMS \textbf{83}, no. 1, 124--126 (1977)

\bibitem{Simpson-variations}
C. T. Simpson, Constructing variations of Hodge structures using Yang-Mills theory and applications to uniformization. J. Amer. Math. Soc. \textbf{1}, 867--918 (1988)

\bibitem{Simpson-noncompact}
C. T. Simpson, Harmonic bundles on noncompact curves. J. Amer. Math. Soc. \textbf{3}, no. 3, 713--770 (1990)

\bibitem{Simpson-Higgs}
C. T. Simpson, Higgs bundles and local systems. Inst. Hautes {\'E}tudes Sci. Publ. Math. \textbf{75}, 5--95 (1992)

\bibitem{Seifert} H. Seifert, Komplexe mit Seitenzuordnung. Nachr. Akad. Wiss. G\"{o}ttingen Math.Phys. Kl. II, 49--80 (1975)

\bibitem{Swoboda} J. Swoboda, Moduli spaces of Higgs bundles on degenerating Riemann surfaces. Adv. Math. \textbf{322}, 637--681 (2017)

\bibitem{Taubes}
C. H. Taubes, Self-dual Yang-Mills connections over non self-dual 4-manifolds. J. Differential Geom. \textbf{17}, 139--170 (1982)

\bibitem{Thurston} 
W. P. Thurston, \textit{The geometry and topology of three-manifolds}. Princeton University Mathematics Department Lecture Notes (Princeton, New Jersey 1979)

\bibitem{VY} 
N. G. Vlamis, A. Yarmola, Basmajian's identity in higher Teichm\"{u}ller-Thurston theory. J. Topol. \textbf{10}, no. 3, 744--764 (2017)

\bibitem{Wald} F. Waldhausen, On irreducible 3-manifolds which are sufficiently large. Ann. of Math. (Second Series) \textbf{87} (1), 56--88 (1968)

\bibitem{Wall} C. T. C. Wall, On the work of W. Thurston. \textit{Proceedings of the International Congress of Mathematicians, Warsaw 1983}, (Warszawa PWN, 1984), pp. 11--14

\bibitem{Wallace} A. H. Wallace, Modifications and cobounding manifolds. Can. J. Math. \textbf{12}, 503--528 (1960)


\bibitem{Weil} A. Weil, On discrete subgroups of Lie groups. Ann. of Math. (2) \textbf{72}, 369--384 (1960)

\bibitem{Wienhard} A. Wienhard, An invitation to higher Teichm\"{u}ller theory. \textit{Proceedings of the International Congress of Mathematicians, Rio de Janeiro 2018}, vol. II. Invited lectures (World Scientific Publishing, Hackensack, N. J., 2018), pp. 1013--1039

\bibitem{Wolf2} M. Wolf, Infinite energy harmonic maps and degeneration of hyperbolic surfaces in moduli space. J. Diff. Geom. \textbf{33}, 487--539 (1991)

\bibitem{Wolf} M. Wolf, The Teichm\"{u}ller theory of harmonic maps. J. Differential Geom. \textbf{29}, no. 2, 449--479 (1989)

\bibitem{Wolpert}
S. Wolpert, The Fenchel-Nielsen deformation. Ann. of Math. \textbf{115}, 501--528 (1982)

\bibitem{Zhou1} Q. Zhou, The moduli space of hyperbolic cone structures. J. Differential Geom. \textbf{51}, no. 3, 517--550 (1999)




















\end{thebibliography}
\end{document}